\documentclass[10pt]{amsart}
\usepackage{amsfonts,amssymb, amscd}
\usepackage{amsmath}
\usepackage{mathrsfs}
\usepackage{lmodern}
\usepackage{amsthm}
\usepackage{qsymbols}
\usepackage{graphicx}
\usepackage{latexsym}
\usepackage{chngcntr}
\usepackage[noadjust]{cite}
\usepackage{paralist}

\usepackage[plainpages=false,pdfpagelabels]{hyperref} 
\usepackage{xcolor}
\hypersetup{
    colorlinks,
    linkcolor={cyan!90!black},
    citecolor={magenta},
    urlcolor={green!40!black}
} 
\usepackage{enumitem}

\newtheorem{theorem}{Theorem}[section]
\newtheorem{lemma}[theorem]{Lemma}
\newtheorem{proposition}[theorem]{Proposition}
\newtheorem{corollary}[theorem]{Corollary}
\newtheorem{assumption}[theorem]{Assumption}
\theoremstyle{definition}
\newtheorem{definition}[theorem]{Definition}
\newtheorem{remark}[theorem]{Remark}
\numberwithin{equation}{section}
\def\Xint#1{\mathchoice
{\XXint\displaystyle\textstyle{#1}}%
{\XXint\textstyle\scriptstyle{#1}}%
{\XXint\scriptstyle\scriptscriptstyle{#1}}%
{\XXint\scriptscriptstyle%
\scriptscriptstyle{#1}}%
\!\int}
\def\XXint#1#2#3{{\setbox0=\hbox{$#1{#2#3}{%
\int}$ }
\vcenter{\hbox{$#2#3$ }}\kern-.6\wd0}}
\def\barint{\,\Xint-} 
\def\Yint#1{\mathchoice
{\YYint\displaystyle\textstyle{#1}}%
{\YYint\textstyle\scriptstyle{#1}}%
{\YYint\scriptstyle\scriptscriptstyle{#1}}%
{\YYint\scriptscriptstyle%
\scriptscriptstyle{#1}}%
\!\int_{}\!\!\int}
\def\YYint#1#2#3{{\setbox0=\hbox{$#1{#2#3}{%
\int_{}\!\!\int}$ }
\vcenter{\hbox{$#2#3$ }}\kern-.6\wd0}}
\def\bariint{\,\,\Yint{\!-\!\!-}} 
\renewcommand{\iint}{\int_{}\!\!\int} 
\newcommand{\IR}{\mathbb{R}}
\newcommand{\IC}{\mathbb{C}}
\newcommand{\IN}{\mathbb{N}}
\newcommand{\cE}{\mathcal{E}}
\newcommand{\cH}{\mathcal{H}}
\newcommand{\cK}{\mathcal{K}}
\newcommand{\cJ}{\mathcal{J}}
\newcommand{\cX}{\mathcal{X}}
\newcommand{\cY}{\mathcal{Y}}
\renewcommand{\L}{\mathrm{L}}
\newcommand{\C}{\mathrm{C}}
\renewcommand{\H}{\mathrm{H}}
\renewcommand{\S}{\mathrm{S}}
\newcommand{\W}{\mathrm{W}}
\newcommand{\X}{\mathrm{X}}
\newcommand{\Lloc}{\L_{\mathrm{loc}}}
\newcommand{\Wloc}{\W_{\mathrm{loc}}}
\newcommand{\V}{\mathcal{V}}
\newcommand{\abs}[1]{\left|#1\right|}
\newcommand{\scal}[2]{(#1 \mathrel \vert \mathrel{} #2)}
\newcommand{\bigscal}[2]{\big(#1 \mathrel{} \big\vert \mathrel{} #2 \big)}
\newcommand{\biggscal}[2]{\bigg(#1 \mathrel{} \bigg\vert \mathrel{} #2 \bigg)}
\newcommand{\ind}{{\mathbf{1}}}
\newcommand{\g}{\Gamma}
\newcommand{\pb}{\Pi_B}
\newcommand{\D}{D}
\newcommand{\B}{B}
\newcommand{\wE}{\widehat{E}}
\DeclareMathOperator{\gV}{{\nabla_{\V}}}
\DeclareMathOperator{\dV}{{\mathrm{div}_{\V}}}
\newcommand{\pe}{\perp}
\newcommand{\pa}{\parallel}
\newcommand{\NT}{\widetilde{N}_*}
\DeclareMathOperator{\nablatx}{{\nabla_{\! \textit{t,x}}}}
\DeclareMathOperator{\divtx}{{\mathrm{div}_{\! \textit{t,x}}}}
\DeclareMathOperator{\nablax}{{\nabla_{\! \textit{x}}}}
\DeclareMathOperator{\divx}{{\mathrm{div}_{\! \textit{x}}}}
\DeclareMathOperator{\nablaA}{{\nabla_{\! \textit{A}}}}
\newcommand{\e}{\mathrm{e}}
\newcommand{\Dir}{\mathscr{D}}
\let\ii\i
\renewcommand{\i}{\mathrm{i}}
\renewcommand{\d}{\mathrm{d}}
\newcommand{\eps}{\varepsilon}
\renewcommand\Re{\operatorname{Re}}

\newcommand{\Lop}{\mathcal{L}}
\newcommand{\cl}[1]{\overline{#1}}
\DeclareMathOperator{\bd}{\partial}
\DeclareMathOperator{\supp}{supp}
\DeclareMathOperator{\dist}{d}
\DeclareMathOperator{\diam}{diam}
\DeclareMathOperator{\Id}{Id}
\DeclareMathOperator{\Rg}{\mathcal{R}}
\DeclareMathOperator{\Ke}{\mathcal{N}}
\DeclareMathOperator{\dom}{\mathcal{D}}
\usepackage{amsaddr} 
\title{Mixed boundary value problems on cylindrical domains}
\author{Pascal Auscher and Moritz Egert}
\address{Laboratoire de Math\'{e}matiques d'Orsay, Univ.\ Paris-Sud, CNRS, Universit\'{e} Paris-Saclay, \\ 91405 Orsay, France}
\email{pascal.auscher@math.u-psud.fr}
\email{moritz.egert@math.u-psud.fr}
\date{\today}
\begin{document}
\begin{abstract}
We study second-order divergence-form systems on half-infinite cylindrical domains with a bounded and possibly rough base, subject to homogeneous mixed boundary conditions on the lateral boundary and square integrable Dirichlet, Neumann, or regularity data on the cylinder base. Assuming that the coefficients $A$ are close to coefficients $A_0$ that are independent of the unbounded direction with respect to the modified Carleson norm of Dahlberg, we prove \emph{a priori} estimates and establish well-posedness if $A_0$ has a special structure. We obtain a complete characterization of weak solutions whose gradient either has an $\L^2$-bounded non-tangential maximal function or satisfies a Lusin area bound. To this end, we combine the first-order approach to elliptic systems with the Kato square root estimate for operators with mixed boundary conditions.
\end{abstract}
\maketitle
\section{Introduction}
\label{Sec: Introduction}

\noindent We consider elliptic $m \times m$-systems of divergence-form equations\index{elliptic system, coupled}
\begin{align}
\label{ES}
 (Lu)_l(t,x) := -\sum_{i,j = 0}^d \sum_{k = 1}^m\partial_i(a_{i,j}^{l,k}(t, x) \partial_j u_k(t,x)) = 0 \qquad
(l=1,\ldots,m)
\end{align}
posed on a cylindrical domain $\IR^+ \times \Omega$ with a bounded base $\Omega \subseteq \IR^d$, $d \geq 2$. Here, and throughout, we write $(t,x) \in \IR^{1+d}$, where $t \in \IR$ is the distinguished perpendicular coordinate and $x \in \IR^d$ is the tangential coordinate. We have set $\partial_0 = \partial_t$ and $\partial_i = \partial_{x_i}$ for $i \geq 1$, and write $\nabla_{t,x}$ for the gradient in all directions and $\nabla_x$ for the tangential gradient. We assume that the coefficient tensor $A(t,x):= (a_{i,j}^{l,k}(t,x))_{i,j=0,\ldots,d}^{l,k=1,\ldots,m}$ is bounded on $\IR^+ \times \Omega$ and strictly accretive on a certain subspace of $\L^2(\Omega)^m \times \L^2(\Omega)^{dm}$. The equations are complemented with mixed Dirichlet/Neumann conditions
\begin{equation}
\label{BC}
\begin{aligned}
 u &= 0 &\qquad&(\text{on $\IR^+ \times \Dir$}) \\
 \nu \cdot A\nablatx u &= 0 &\qquad&(\text{on $\IR^+ \times (\bd \Omega \setminus \Dir)$})
\end{aligned}
\end{equation}
on the lateral boundary, see Figure~\ref{Fig:cylinder} below for illustration. 
\begin{figure}[ht]
	\centering
	\includegraphics[scale=0.2]{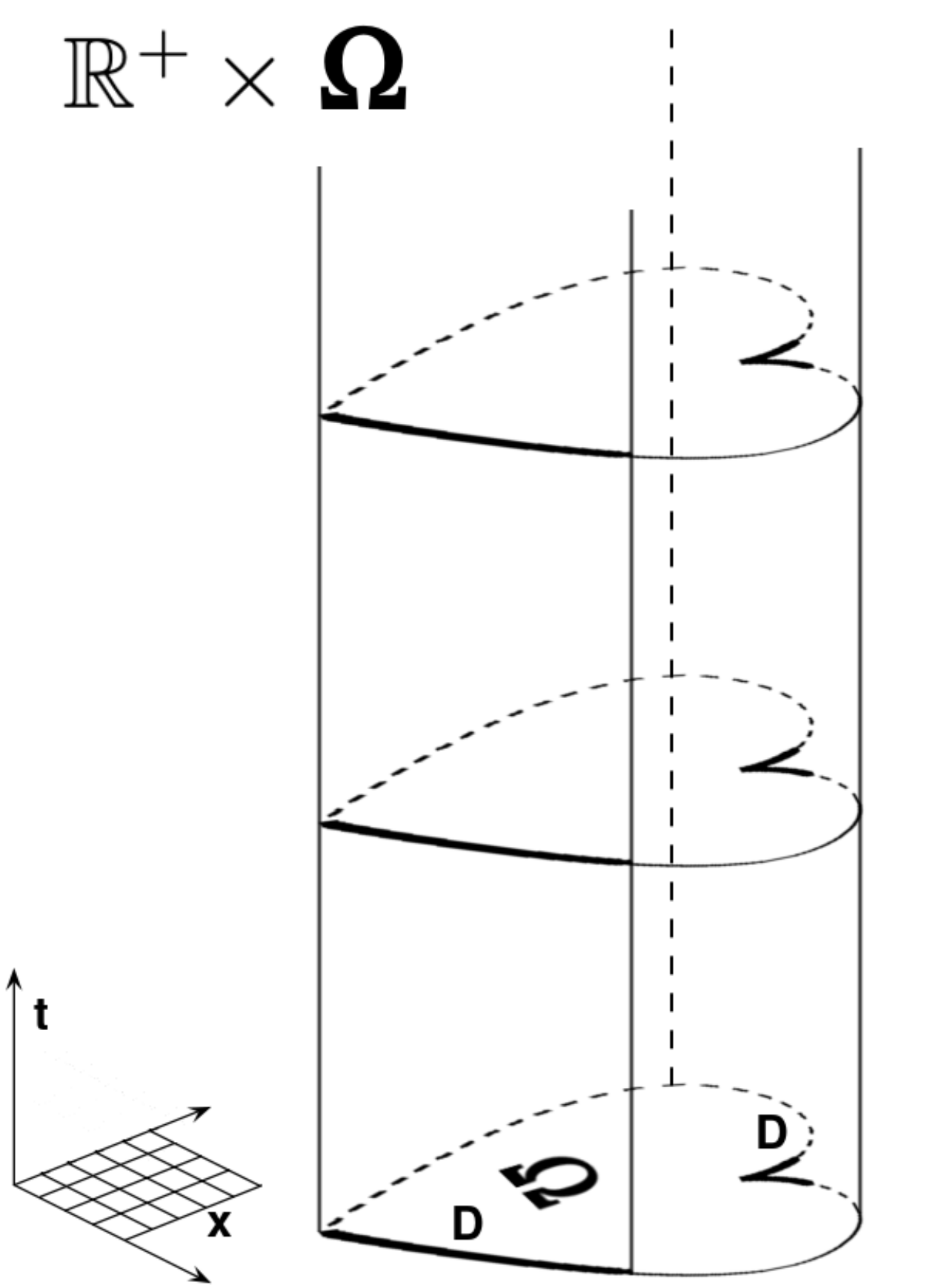}
	\caption{The cylinder $\IR^+ \times \Omega \subseteq \IR^3$ is built from a non-Lipschitzian base $\Omega \subseteq \IR^2$ (the heart) that satisfies the standing geometric assumptions in this article. The lateral boundary splits into a Dirichlet part $\IR^+ \times \Dir$ (highlighted by bold lines) and its complement carrying homogeneous Neumann boundary conditions. On the bottom of our heart, inhomogeneous boundary conditions for either $u|_{t=0}$, $\nu \cdot A\nablatx u|_{t = 0}$, or $\nablax u|_{t=0}$ are imposed.}
\label{Fig:cylinder}
\end{figure}
Here, $\Dir$ is a closed subset of $\bd \Omega$ and $\nu$ denotes the formal outer unit normal vector to the boundary of $\IR^+ \times \Omega$. Our focus lies on rough geometric configurations even beyond the Lipschitz class. So, we assume that $\Omega$ is $d$-Ahlfors regular, that $\Dir$ satisfies the Ahlfors-David condition, and only around the Neumann part of the boundary we require Lipschitz coordinate charts. Let us stress that the pure Dirichlet case $\Dir = \bd \Omega$ and the pure Neumann case $\Dir = \emptyset$ are not excluded from our considerations. 

Our goal is to classify all weak solutions $u$ to these equations that satisfy appropriate interior estimates of $\nablatx u$, such as a square-integrable non-tangential maximal function or Lusin area bounds. Moreover, we aim for well-posedness, that is, unique solvability in the aforementioned spaces, given either Dirichlet data $u|_{t=0}$, Neumann data $\nu \cdot A\nablatx u|_{t = 0}$, or Dirichlet regularity data $\nablax u|_{t=0}$ in $\L^2(\Omega)$. 

Since the coefficients $A$ may depend on all variables, these boundary value problems are not always solvable in general, unless some additional regularity in $t$-direction is imposed, see \cite{Caffaerlli-Fabes-Kenig, Axelsson-CE, AA-Inventiones} for counterexamples and further background. Following the treatment initiated by Axelsson and the first author \cite{AA-Inventiones, AA-2}, we use the modified Carleson norm $\|\cdot\|_C$ originating from the work of Dahlberg \cite{Dahlberg-modifiedCarleson} as a fair means to measure the size of perturbations of $A$ from the class of $t$-independent coefficients $A_0$. 

Assuming finiteness of $\|A - A_0\|_C$, we prove \emph{a priori} estimates and representation formulas for all weak solutions with non-tangential maximal function estimates $\|\NT(\nablatx u)\|_{\L^2(\Omega)} < \infty$ or Lusin area bounds $\int_0^\infty \|\nabla_{t,x} u\|_{\L^2(\Omega)}^2 t \d t < \infty$. Here, $\NT$ is a modified non-tangential maximal function taking $\L^2$-averages over truncated cones. Ever since the work of Kenig and Pipher~\cite{Kenig-Pipher93}, the $\L^2$-bound for $\NT(\nablatx u)$ is considered a natural interior estimate for the Neumann and regularity problem. Given our method, the Lusin area bound is most natural for the Dirichlet problem but we show that any such solution satisfies $\|\NT(u)\|_{\L^2(\Omega)} < \infty$ as well. Moreover, we prove that any solution with non-tangential maximal bound and Lusin area bound attains a trace $\nabla_{t,x} u|_{t=0}$ and $u|_{t = 0}$ on $\{0\} \times \Omega$, respectively, in the sense of almost everywhere convergence of Whitney averages.

Next, assuming smallness of $\|A - A_0\|_C$ and that $A_0$ is either Hermitean or a block matrix (that is, no mixed derivatives $\partial_t \partial_{x_i}$ occur), we obtain well-posedness of the inhomogeneous boundary value problems. For a precise formulation of our main results we refer to Section~\ref{Sec: Main results}. We remark that these results match the status quo for elliptic systems with $\L^2$ boundary data on the upper half space $\IR_+^d$.

Modern theory for real equations on the upper half space, that is, when $m=1$, $\Omega = \IR^d$, and $A(t,x) \in \IR^{(1+d) \times (1+d)}$, dates back to Dahlberg \cite{Dahlberg}, who was first to solve the Dirichlet problem for $\Delta u=0$ on a Lipschitz domain with boundary data $\varphi \in \L^2$. For such equations the picture is rather complete by now, see \cite{Dirichlet-plarge, Jerison-Kenig} just to mention a few. All of these results heavily build on real-variable techniques, such as maximum principles and harmonic measures, which for equations with complex coefficients (let alone coupled systems of such) are not available anymore.

In this paper, we follow a completely different approach that has been proposed and developed to full strength in a series of papers by Axelsson, M\textsuperscript{c}Intosh, and the first author \cite{AAM-ArkMath, AA-Inventiones, AAM-EstimateRelatedToKato, AA-2} and which works equally well for real equations and systems, see also \cite{Auscher-McIntosh-Mourgoglu_L2BVPs, Auscher-Axelsson-Hofmann_ComplexPerturbations} for related results.
To date, this so-called $\D \B$-approach as only been followed for systems on the upper half-space or the unit ball \cite{AA-2}. Much more challenging geometric configurations, such as a cylinder with a rough base bear new and interesting challenges arising from the lateral boundary conditions. These have -- at least to our knowledge -- not been addressed before.

The general idea is to reformulate the second-order system for $u$ as a first-order system for the conormal gradient $f$ of $u$, a vector formed of the conormal derivative and the tangential gradient at each interior point, see Section~\ref{Sec: Equivalence of ES to a first order system} for definitions. The first-order system for $f$ then has the form of a non-autonomous evolution equation
\begin{align*}
 \partial_{t}f_t+ \D \B_t f_t =0 \qquad (t>0)
\end{align*}
for $\D$ a first-order self-adjoint operator acting on the tangential variables and $\B_t$ a bounded accretive multiplication operator. The lateral boundary conditions are hidden in the domain of $\D$. Having rephrased as
\begin{align*}
 \partial_{t}f_t+ \D \B_0 f_t = \D (\B_0 - \B_t) f_t \qquad (t>0),
\end{align*}
where $\B_0$ is independent of $t$ and corresponds to a $t$-independent coefficient tensor $A_0$ just in the same manner as $\B$ corresponds to $A$, it is tempting to solve by the semigroup formula $f_t = \e^{- t \D \B_0} f_0$ if $\B = \B_0$ and then use maximal regularity methods to obtain $f$ via a Duhamel formula in the general case. However, since $\D \B_0$ will have positive and negative spectrum, the underlying evolution for $f$ will be forward on one part of $\L^2$ and backward on another part. In order to master the situation, we have to split $\L^2$ into spectral subspaces. 

In Section~\ref{Sec: Quadratic estimates for DB0} we establish boundedness of the spectral projections $E_0^\pm := \ind_{\IC^\pm}(\D \B_0)$, which is a highly delicate matter in general and would not have been available before the resolution of the Kato square root problem for elliptic systems with mixed boundary conditions acting on the cylinder base $\Omega$ only. On smooth domains and bi-Lipschitz images thereof, the solution of this so-called Lions Problem is due to Axelsson, Keith, and M\textsuperscript{c}Intosh~\cite{AKM}. Their methods have been fine-tuned by Haller-Dintelmann, Tolksdorf, and the second author, who obtained the same result in the general geometric setup treated in the present article~\cite{Darmstadt-KatoMixedBoundary, Laplace-Extrapolation}. 

In Section~\ref{Sec: Semigroup solutions to the first order equation}, which lies at the heart of this article, we present a careful analysis of the semigroup solutions $f_t = \e^{- t \D \B_0} f_0$ to the first-order system for $\B = \B_0$. In particular, we identify them as elements of the natural solution spaces and prove Whitney average convergence as $t \to 0$ toward the data $f_0$. 

As for the extension to $t$-dependent coefficients with modified Carleson control, we can rely on the maximal regularity estimates for elliptic systems on the upper half space due to Ros\'{e}n and the first author \cite{AA-Inventiones}, which are mostly formulated on abstract function spaces and therefore hold for our setup as well. Hence, we shall be rather brief here and suggest to keep a copy of \cite{AA-Inventiones} handy as duplicated arguments will be omitted. Additionally, we will prove almost everywhere convergence of Whitney averages of solutions, which was left as an open problem in \cite{AA-Inventiones} and was partly resolved in \cite{AA-2, Auscher-Stahlhut_APriori}. The so-obtained \emph{a priori} estimates for weak solutions to $t$-dependent systems are presented in Section~\ref{Sec: Representation of solutions}. In the special case of $t$-independent coefficients $A = A_0$, they entail that the semigroup solutions investigated in Section~\ref{Sec: Semigroup solutions to the first order equation} are the only solutions to the first-order system $\partial_{t}f_t+ \D \B_0 f_t = 0$ satisfying the respective interior estimates on $\IR^+ \times \Omega$.

Finally, in Section~\ref{Sec: Well-posedness} we prove well-posedness of the three boundary problems for $t$-independent coefficients $A_0$ that are either Hermitean, of block form, or sufficiently close to one of these classes in the $\L^\infty$-topology. We also show that this result is stable under $t$-dependent perturbations $A$ satisfying a smallness condition on $\|A-A_0\|_C$.

Of course, weak solutions to the elliptic system with mixed lateral boundary conditions can also be constructed using the Lax-Milgram lemma, provided there is an interior control $\nablatx u \in \L^2$ and that the data at $t=0$ is contained in the appropriate trace spaces. If $A = A_0$, then uniqueness of these Lax-Milgram solutions entails that their conormal gradient follows a semigroup flow as well. The difference to the methods we present in this paper is that our semigroup representation really is an \emph{a priori} result obtained independently of any solvability issues. The connection of the first-order $\D \B$-formalism to the classical energy solutions has been closely investigated for boundary value problems on the upper half-space \cite{Auscher-McIntosh-Mourgoglu_L2BVPs}. Similar results hold for our setup as well, but these considerations go beyond the scope of this article. The interested reader may refer to the PhD-thesis \cite{EigeneDiss} of the second author for details. In the context of energy solutions also inhomogeneous data on the lateral boundary can be treated. In fact, well-known trace and extension theorems from the cylinder to its full boundary and vice versa allow to reduce to an inhomogeneous equations $Lu = g$ again with homogeneous boundary conditions. For the two solution classes considered in this article, however, proving analogous trace and extension results is an interesting problem on its own account. Once this is established, one might be able to use the first-order approach with non-zero right-hand side. We leave this as an challenging open problem that is worthwhile pursuing.
\section{Notation and basic assumptions}
\label{Sec: Preliminaries and geomtric setup}

\subsection{General notation}
\label{Subsec: General notation}

\noindent Function spaces in this article are always over the complex number field. For functions $f$ on $\IR^{1+d}$ we let $f_t(x) = f(t,x)$ and frequently identify $\L^2(\IR^{1+d}) \cong \L^2(\IR; \L^2(\IR^d))$ in virtue of Fubini's theorem. We decompose $f \in \IC^n$, where $n = (1+d)m$ and $m$ is the number of equations in our elliptic system \eqref{ES}, as
\begin{align*}
 f = \begin{bmatrix} f_\pe \\ f_\pa \end{bmatrix}
\end{align*}
into its \emph{perpendicular} part $f_\pe \in \IC^m$ and its \emph{tangential} part $f_\pa \in \IC^{dm}$. We denote inner products by $\scal{\cdot}{\cdot}$ and for $f, g \in \IC^n$ we write $f \cdot g := \sum_{j=1}^n f_j g_j$. We let $\dist(E,F)$ be the semi-distance of sets $E, F \subseteq \IR^d$ induced by Euclidean distance on $\IR^d$ and we abbreviate by $\dist_E(x)$ if $F = \{x\}$.

Given compatible Banach spaces $\cX_0$, $\cX_1$, we write $[\cX_0, \cX_1]_\theta$, $0 < \theta < 1$, for the corresponding scale of complex interpolation spaces. For background on interpolation theory the reader can refer e.g.\ to \cite{Bergh-Loefstroem}.

Concerning inequalities we will write $A \lesssim B$ if there exists a constant $c>0$ not depending on the parameters at stake, such that $A \leq c B$. Similarly, we use the symbols $\gtrsim$ and $\simeq$.

\subsection{Geometry of the cylinder base}
\label{Subsec: Geometry}

\noindent We require the following geometric quality of the cylinder base $\Omega$ and the Dirichlet part $\Dir \subseteq \bd \Omega$. These are the same assumptions under which the Kato problem for mixed boundary conditions was solved in \cite{Darmstadt-KatoMixedBoundary}.

\begin{assumption}
\label{Ass: General geometric assumption on Omega BVP}
Let $\Omega \subseteq \IR^d$ be a bounded domain and let $\Dir \subseteq \bd \Omega$ be closed.
\begin{enumerate}
\def\theenumi{\roman{enumi}}
 \item \label{Geomi}  Assume that $\Omega$ is \emph{$d$-Ahlfors regular},
 \begin{align*}
  \qquad \qquad |B(x,r) \cap \Omega| \simeq r^d \qquad (x \in \Omega,\, 0 < r \leq 1).
 \end{align*}

 \item \label{Geomii} Assume that $\Dir$ is either empty or \emph{$(d-1)$-Ahlfors regular},
  \begin{align*}
  \qquad \qquad \cH_{d-1}(B(x,r) \cap \Dir) \simeq r^{d-1} \qquad (x \in \Dir,\, 0 < r \leq 1),
 \end{align*}
 where $\cH_{d-1}$ denotes the $(d-1)$-dimensional Hausdorff measure in $\IR^d$.

 \item \label{Geomiii} The \emph{Lipschitz condition} holds around $\cl{\bd \Omega \setminus \Dir}$: For every $x \in \cl{\bd \Omega \setminus \Dir}$ there is an open neighborhood $U_x$ and a bi-Lipschitz mapping $\Phi_x: U_x \to (-1,1)^d$ such that
 \begin{align*}
  \qquad \qquad \Phi_x(U_x \cap \Omega) = (-1,0) \times (-1,1)^{d-1}, \qquad \Phi_x(U_x \cap \bd \Omega) = \{0\} \times (-1,1)^{d-1}.
 \end{align*}
 \end{enumerate}
\end{assumption}

\begin{remark}
\label{Rem: Omega doubling}
Our assumptions entail that $\Omega$ equipped with the restricted Euclidean distance and the restricted Lebesgue measure becomes a doubling metric measure space, see e.g.\ \cite{Bjorn-Bjorn} for this notion. Moreover, given any $r_0 > 0$, comparability $|B(x,r) \cap \Omega| \simeq r^d$ easily extends to all $0 < r \leq r_0$ upon a change of the implicit constants.
\end{remark}

\subsection{Sobolev spaces}
\label{Subsec: Sobolev spaces}

For $\Xi \subseteq \IR^d$ an open set and $\mathscr{E} \subseteq \bd \Xi$ a closed part of its boundary, we define the Sobolev spaces $\W_{\mathscr{E}}^{1,p}(\Xi)$, $1<p<\infty$, as the closure of the set of test functions
\begin{align*}
 \C_{\mathscr{E}}^\infty(\Xi):= \big\{u|_\Xi;\, u \in \C_c^\infty(\IR^d),\, \dist(\supp u, \mathscr{E}) > 0 \big\}
\end{align*}
with respect to the norm $u \mapsto (\int_\Xi \abs{u}^p + \abs{\nabla u}^p \; \d x)^{1/p}$. These spaces should be thought of as the subspaces of those functions in the ordinary Sobolev spaces $\W^{1,p}(\Xi)$ that vanish on $\mathscr{E}$ in an appropriate sense. For further information on their structure the reader can refer e.g.\ to \cite{Hardy-Poincare, Brewster}. 

Under Assumption~\ref{Ass: General geometric assumption on Omega BVP} there exists a bounded extension operator $E: \W_\Dir^{1,p}(\Omega) \to \W^{1,p}(\IR^d)$ independent of $p$ such that $E u = u$ a.e.\ on $\Omega$ for every $u \in \W_\Dir^{1,p}(\Omega)$, see e.g.\ \cite[Lem.~3.2]{ABHR}. In particular, this gives the compatibility $\W_\emptyset^{1,p}(\Omega) = \W^{1,p}(\Omega)$ and provides the usual Sobolev embeddings of type $\W_\Dir^{1,p}(\Omega) \subseteq \L^q(\Omega)$.

\subsection{Weak solutions}
\label{Subsec: Weak solutions ES}

We write $\V:= \W_\Dir^{1,2}(\Omega)^m$ for the natural $\L^2$-function space allowing to model mixed Dirichlet/Neumann boundary conditions for $m \times m$ elliptic systems and let $\gV$ denote the distributional gradient operator $\L^2(\Omega)^m \to \L^2(\Omega)^{dm}$ with domain $\dom(\gV) := \V$.

\begin{assumption}
\label{Ass: Ellipticity for BVP}
The coefficient tensor $A(t,x)$ is measurable and essentially bounded,
\begin{align*}
 A(t,x) := (a_{i,j}^{l,k}(x))_{i,j=0,\ldots,d}^{l,k = 1,\ldots,m} \in \L^\infty(\IR^+ \times \Omega; \Lop(\IC^{(1+d)m}))
\end{align*}
and there exists some $\lambda > 0$ independent of $t>0$ allowing for the ellipticity/accretivity condition
\begin{align*}
 \Re \int_\Omega A(t, x) f(x) \cdot \cl{f(x)} \; \d x \geq \lambda \int_\Omega \abs{f(x)}^2 \; \d x \qquad (f \in \L^2(\Omega)^m \times \Rg(\gV)).
\end{align*}
\end{assumption}

\begin{remark}
\label{Rem: Ellipticity for BVP stronger than for Lax-Milgram}
Assumption~\ref{Ass: Ellipticity for BVP} is weaker than pointwise uniform accretivity of $A$ and stronger than G\aa{}rding's inequality for $u \in \L^2 (\IR^+; \V) \cap \W^{1,2}(\IR^+; \L^2(\Omega)^m)$. The second statement follows by taking $f= (\nablatx  u)_t$ for fixed $t > 0$ and integrating over $t$. For further information and related ellipticity concepts the reader can refer to \cite[Sec.\ 2]{AAM-ArkMath}.
\end{remark}

A formal integration by parts in \eqref{ES}, taking into account the lateral boundary conditions \eqref{BC}, leads to our notion of $\Lloc^2(\L^2)$-weak solutions.

\begin{definition}
\label{Def: Weak solutions to ES}
If $\Dir \neq \emptyset$, then a \emph{weak solution} to the elliptic system complemented with mixed lateral boundary conditions is a function $u \in \Lloc^2(\IR^+; \V) \cap \Wloc^{1,2}(\IR^+; \L^2(\Omega)^m)$ that satisfies
\begin{align}
\label{Eq: Weak solutions to ES}
 \int_0^\infty \int_\Omega A(t,x) \nablatx  u(t,x) \cdot \cl{\nablatx v(t,x)} \; \d x \; \d t = 0 \qquad (v \in \C_c^\infty(\IR^+; \V)).
\end{align}
If $\Dir = \emptyset$, then it is additionally required that $u$ satisfies the \emph{no-flux condition}
\begin{align*}
 \lim_{t \to \infty} \int_\Omega (A(t,x) \nablatx u(t,x))_\pe \; \d x = 0.
\end{align*}
\end{definition}

The no-flux condition is common to rule out linear growth of solutions at spatial infinity~\cite{Katie}. This specialty of the pure lateral Neumann case is a substitute for the Dirichlet boundary condition at $t=\infty$, which is present in all other cases as the Dirichlet part $\IR^+ \times D$ reaches up to spacial infinity. In fact, the flux $\int_\Omega (A \nablatx u)_\pe \d x$ is independent of $t$.

\begin{lemma}
\label{Lem: Flux independent of t}
Suppose $\Dir = \emptyset$. If $u \in \Lloc^2(\IR^+; \V) \cap \Wloc^{1,2}(\IR^+; \L^2(\Omega)^m)$ satisfies \eqref{Eq: Weak solutions to ES}, then there is a constant $c \in \IC^m$ such that $\int_\Omega (A \nablatx u)_\pe \d x = c$ for all $t>0$. In particular, $c=0$ if $u$ is a weak solution.
\end{lemma}

\begin{proof}
Let $y \in \IC^m$. For every $\eta \in \C_c^\infty(\IR^+; \IR)$ the choice $v_t(x) = \eta(t) y$, $t>0$, is admissible in \eqref{Eq: Weak solutions to ES} and
\begin{align*}
 \int_0^\infty \eta'(t) \int_\Omega \big(A(t,x) \nablatx u(t,x)\big)_\pe \cdot \cl{y} \; \d x  \; \d t = 0
\end{align*}
follows. Hence, the integral over $\Omega$ is independent of $t$. Letting $y$ run through the standard orthonormal basis of $\IC^m$ yields the claim.
\end{proof}

Finally, we define the conormal gradient of weak solutions, a vector formed from the gradient $\nablatx u$ in such a way that its $\pe$-component corresponds to Neumann and its $\pa$-component to regularity boundary conditions.

\begin{definition}
\label{Def: Conormal gradient}
The \emph{conormal gradient} of a function $u \in \Lloc^2(\IR^+; \W^{1,2}(\Omega)^m) \cap \Wloc^{1,2}(\IR^+; \L^2(\Omega)^m)$ is given by
\begin{align*}
 \nablaA  u: = \begin{bmatrix} (A \nablatx  u)_\pe \\ \nablax  u \end{bmatrix} \in \Lloc^2(\IR^+; \L^2(\Omega)^n).
\end{align*}
\end{definition}

\subsection{Modified non-tangential maximal function}
\label{Subsec: Modified non-tangential maximal function}

Following \cite{AA-Inventiones}, we define a modified non-tangential maximal function on the cylinder $\IR^+ \times \Omega$ by $\L^2$-averaging over truncated cones, called Whitney balls below.

\begin{definition}
\label{Def: Non-tangential maximal function}
The \emph{modified non-tangential maximal function} of a function $f$ on $\IR^+ \times \Omega$ is defined by
\begin{align*}
 \NT f(x) := \sup_{t > 0} \bigg(\bariint_{W(t,x)} \abs{f(s,y)}^2\; \d y \; \d s\bigg)^{1/2} \qquad (x \in \Omega),
\end{align*}
where 
\begin{align*}
 W(t,x) = \big\{(s,y) \in \IR^+ \times \Omega;\, c_0^{-1}t < s < c_0 t, \, |y-x| < c_1 t \big \}
\end{align*}
is called \emph{Whitney ball} around $(t,x)$ and $c_0 >1$ and $c_1 >0$ are fixed constants. The \emph{modified Carleson norm} of a function $g$ on $\IR^+ \times \Omega$ is
\begin{align*}
 \|g\|_C:=  \bigg( \sup_B \frac{1}{|B \cap \Omega|} \iint_{(0,r(B)) \times (B \cap \Omega)} \Big( \sup_{W(t,x)} |g|^2 \Big) \; \frac{\d x \; \d t}{t} \bigg)^{1/2},
\end{align*}
where the supremum is taken over all balls $B\subseteq \IR^d$ with center in $\Omega$ and radius $r(B)>0$. 
\end{definition}

The modified Carleson norm will serve as our measure for the deviation of the coefficients $A$ from the class $t$-independent coefficients $A_0(t,x) = A_0(x)$. The reader should think of $\|A - A_0\|_C < \infty$ to mean that ``$A(t,x) = A_0(x)$ holds at $t=0$ but also that $A(t,x)$ is close to $A_0(x)$ at all scales''~\cite{AA-Inventiones}. It turns out that given $A$, such coefficients $A_0$ are unique and satisfy Assumption~\ref{Ass: Ellipticity for BVP} with controlled bounds. The proof of this result is deferred until Section~\ref{Sec: Natural function spaces}.

\begin{lemma}
\label{Lem: Unique t-independent coefficients}
Let $A: \IR^+ \times \Omega \to \Lop(\IC^{(1+d)m})$ satisfy Assumption~\ref{Ass: Ellipticity for BVP} with constant of accretivity $\lambda > 0$. Assume that $A_0: \IR^+ \times \Omega \to \Lop(\IC^{(1+d)m})$ are $t$-independent measurable coefficients such that $\|A - A_0\|_C < \infty$. Then $A_0$ is uniquely determined  by $A$, that is, if $A'_0$ are $t$-independent measurable coefficients such that $\|A - A'_0\|_C < \infty$, then $A'_0 = A_0$ almost everywhere. Furthermore, $A_0$ satisfies Assumption~\ref{Ass: Ellipticity for BVP} with
\begin{align*}
 \lambda \leq \lambda_0 \leq \|A_0\|_\infty \leq \|A\|_\infty,
\end{align*}
where $\lambda_0$ denotes a constant of accretivity for $A_0$.
\end{lemma}
\section{Main results}
\label{Sec: Main results}

\noindent Our first two results provide \emph{a priori} estimates for weak solutions to the system 
\begin{equation*}
\begin{aligned}
 Lu&= 0 &\qquad &(\text{in $\IR^+ \times \Omega$}) \\[-4pt]
 u &= 0 &\qquad&(\text{on $\IR^+ \times D$}) \\[-4pt]
 \nu \cdot A\nablatx u &= 0 &\qquad&(\text{on $\IR^+ \times (\bd \Omega \setminus \Dir)$})
\end{aligned}
\end{equation*}
that satisfy appropriate interior estimates of $\nablatx u$ and one of the following three classical inhomogeneous boundary conditions on the cylinder bottom:
\begin{itemize}
 \item The Dirichlet condition $u = \varphi$ on $\{0\} \times \Omega$, given $\varphi \in \L^2(\Omega)^m$,
 \item The Neumann condition $(A \nablatx u)_\pe = \varphi$ on $\{0\} \times \Omega$, given $\varphi \in \L^2(\Omega)^m$,
 \item The Dirichlet regularity condition $\nablax u = \varphi$ on $\{0\} \times \Omega$, given $\varphi \in \L^2(\Omega)^{dm}$.
\end{itemize}
Note that $(1,0)$ is the inward pointing normal vector to $\{0\} \times \Omega$ (identified with $\Omega$ for simplicity of exposition), so that $(A \nablatx u)_\pe = \varphi$ really is a boundary condition of Neumann type.

For the Neumann and regularity problems we impose an $\L^2$-bound for the non-tangential maximal function of $u$ and obtain the following

\begin{theorem}
\label{Thm: Main result Neu/Reg}
Let $\Omega$ and $\Dir$ satisfy Assumption~\ref{Ass: General geometric assumption on Omega BVP} and let the coefficients $A$ be bounded and elliptic as in Assumption~\ref{Ass: Ellipticity for BVP}.
\begin{enumerate}
\def\theenumi{\roman{enumi}}
 \item \label{Neu/Reg i} \emph{A priori estimates and traces}: Suppose that there exist $t$-independent measurable coefficients $A_0$ such that $\|A - A_0\|_C < \infty$. If $u$ is a weak solution to the elliptic system with estimates $\|\NT(\nablatx u)\|_{\L^2(\Omega)} < \infty$, then $\nablaA u$ has limits
 \begin{align*}
  \lim_{t \to 0} \barint_t^{2t} \|\nablaA u_s - f_0\|_{\L^2(\Omega)^n}^2 \; \d s = 0 = \lim_{t \to \infty} \barint_t^{2t} \|\nablaA u_s\|_{\L^2(\Omega)^n}^2 \; \d s
 \end{align*}
 for some trace function $f_0 \in \L^2(\Omega)^n$ with estimate $\|f_0\|_{\L^2(\Omega)^n} \lesssim \|\NT(\nablatx u)\|_{\L^2(\Omega)}$. Moreover, Whitney averages of $\nablaA u$ converge to $f_0$ almost everywhere,
 \begin{align*}
  \lim_{t \to 0} \bariint_{W(t,x)} \nablaA u \; \d s \; \d y = f_0(x) \qquad (\text{a.e.\ $x \in \Omega$}).
 \end{align*}
 \item \label{Neu/Reg ii} \emph{Regularity for $t$-independent coefficients}: If $A = A_0$ is $t$-independent, then every weak solution $u$ with estimates as in \eqref{Neu/Reg i} has additional regularity
 \begin{align*}
  \nablaA u \in \C([0, \infty); \L^2(\Omega)^n) \cap \C^\infty((0,\infty); \L^2(\Omega)^n)
 \end{align*}
 and converge to $f_0$ and $0$ in the $\L^2(\Omega)^n$-sense as $t \to 0$ and $t \to \infty$, respectively.
\end{enumerate}
\end{theorem}

\begin{proof}
Part \eqref{Neu/Reg i} follows from Theorem~\ref{Thm: Representation for X solutions} and Theorem~\ref{Thm: Almost everywhere convergence of Whitney averages}. Part \eqref{Neu/Reg ii} is due to Corollary~\ref{Cor: X-solutions t-independent}.
\end{proof}

For the Dirichlet problem a Lusin area bound is more feasible (given our method), though we obtain \emph{a priori} non-tangential estimates as well.

\begin{theorem}
\label{Thm: Main result Dir}
Let $\Omega$ and $\Dir$ satisfy Assumption~\ref{Ass: General geometric assumption on Omega BVP} and let the coefficients $A$ be bounded and elliptic as in Assumption~\ref{Ass: Ellipticity for BVP}.
\begin{enumerate}
\def\theenumi{\roman{enumi}}
 \item \label{Dir i} \emph{A priori estimates and traces}: Suppose that there exists $t$-independent measurable coefficients $A_0$ such that $\|A - A_0\|_C < \infty$. If $u$ is a weak solution to the elliptic system with Lusin area bounds $\int_0^\infty \|\nablatx u\|_{\L^2(\Omega)}^2 \; t \d t < \infty$, then $u \in \C([0,\infty); \L^2(\Omega)^m)$ and there are limits
 \begin{align*}
  \lim_{t \to 0} u_t = u_0  \quad \text{and} \quad \lim_{t \to \infty} u_t = u_\infty
 \end{align*}
 in the $\L^2(\Omega)^m$-sense for some trace $u_0 \in \L^2(\Omega)^m$ and a constant $u_\infty \in \IC^m$, which is zero if the lateral Dirichlet part is non-empty. Moreover, there are estimates
 \begin{align*}
  \qquad \qquad \|u_0\|_2 \lesssim \|\NT(u)\|_{\L^2(\Omega)} + \sup_{t> 0} \|u_t\|_{\L^2(\Omega)^m} \lesssim |u_\infty| + \bigg(\int_0^\infty \|\nablatx u\|_{\L^2(\Omega)}^2 \; t \d t \bigg)^{1/2} < \infty
 \end{align*} and Whitney averages of $u$ converge to $u_0$ almost everywhere,
 \begin{align*}
  \lim_{t \to 0} \bariint_{W(t,x)} u \; \d s \; \d y = u_0(x) \qquad (\text{a.e.\ $x \in \Omega$}).
 \end{align*}
 \item \label{Dir ii} \emph{Regularity for $t$-independent coefficients}: If $A = A_0$ is $t$-independent, then every weak solution $u$ with estimates as in \eqref{Dir i} has additional regularity
 \begin{align*}
  u \in \C([0, \infty); \L^2(\Omega)^m) \cap \C^\infty((0,\infty); \V).
 \end{align*}
\end{enumerate}
\end{theorem}

\begin{proof}
Part \eqref{Dir i} is due to Theorem~\ref{Thm: Representation for Dirichlet problem} and Theorem~\ref{Thm: NT bounds Dirichlet}. Part \eqref{Dir ii} follows from Corollary~\ref{Cor: Y-solutions t-independent}.
\end{proof}

Our third main result concerns well-posedness of the three boundary value problems. We say that the Dirichlet problem for $A$ is \emph{well-posed} if for each $\varphi \in \L^2(\Omega)^m$ there exists a unique weak solution $u$ to the elliptic system for $A$ with estimate $\|\NT(\nablatx u)\|_{\L^2(\Omega)} < \infty$ such that Whitney averages of $u$ converge to $\varphi$ a.e.\ as $t \to 0$. 

In the case $\Dir \neq \emptyset$ we similarly say that the Neumann and regularity problem for $A$ are well-posed if for each $\varphi \in \L^2(\Omega)^{m}$ and $\varphi \in \cl{\Rg(\gV)}$ there exist a unique weak solution with estimate $\int_0^\infty \|\nablatx u\|_{\L^2(\Omega)}^2 \; t \d t < \infty$ such that Whitney averages of $(A \nablatx u)_\pe$ and $\nablax u$ converge to $\varphi$ a.e.\ as $t \to 0$, respectively. Note that $\varphi \in \cl{\Rg(\gV)}$ for the regularity problem is a natural compatibility condition for the boundary trace since 
\begin{align*}
 f_0 = \lim_{t \to 0} \barint_t^{2t} \nablaA u \; \d s \in \L^2(\Omega)^m \times \cl{\Rg(\gV)}
\end{align*}
by Theorem~\ref{Thm: Main result Neu/Reg}\eqref{Neu/Reg i}. If $\Dir = \emptyset$, then we have to take care of the constant functions. So, well-posedness for the Neumann and regularity problems is defined similarly as before but we require uniqueness of $u$ only modulo constants on $\IR^+ \times \Omega$ and for the Neumann problem we include the natural compatibility $\int_\Omega \varphi =0$ stemming from the no-flux condition on $u$.

\begin{theorem}
\label{Thm: Well-posedness}
Let $\Omega$ and $\Dir$ satisfy Assumption~\ref{Ass: General geometric assumption on Omega BVP}, let the coefficients $A$ be bounded and elliptic as in Assumption~\ref{Ass: Ellipticity for BVP}, and suppose that there exist $t$-independent measurable coefficients $A_0$ such that $\|A - A_0\|_C < \infty$
\begin{enumerate}
\def\theenumi{\roman{enumi}}
 \item \label{WP i} \emph{Well-posedness for $A_0$}: Each of the three boundary value problems for $A_0$ is well-posed if $A_0$ is either Hermitean, a block matrix with respect to the block decomposition on $\Lop(\IC^m \times \IC^{dm})$, or sufficiently close in the $\L^\infty(\Omega; \Lop(\IC^n))$-topology to such coefficients $A_0'$ satisfying Assumption~\ref{Ass: General geometric assumption on Omega BVP}.
 \item \label {WP ii} \emph{Well-posedness for $A$}: If the Neumann/regularity problem for $A_0$ is well-posed, then there exists $\eps > 0$ depending on $A_0$, dimension, and geometric parameters, such that the Neumann/regularity problem for $A$ is well-posed provided $\|A - A_0\|_C < \eps$. In this case, given appropriate data $\varphi$, the corresponding solution satisfies 
 \begin{align*}
  \|\NT(\nablatx u)\|_{\L^2(\Omega)} \simeq \|\varphi\|_{\L^2(\Omega)}.
 \end{align*}
 A similar perturbation result holds for the Dirichlet problem with solution estimates
 \begin{align*}
  \qquad \quad \|\varphi\|_{\L^2(\Omega)^m} \simeq \|\NT(u)\|_{\L^2(\Omega)} \simeq \sup_{t>0} \|u_t\|_{\L^2(\Omega)^m}^2 \simeq |u_\infty| + \bigg(\int_0^\infty \|\nablatx u \|_{\L^2(\Omega)}^2 \; t \d t\bigg)^{1/2},
 \end{align*}
 where $u_\infty = \lim_{t \to \infty} u_t$ and in particular $u_\infty = 0$ as long as the Dirichlet part $\Dir$ is non-empty.
\end{enumerate}
\end{theorem}

\begin{proof}
This result is proved in the final Section~\ref{Subsec: Proof of well-posedness theorem}.
\end{proof}
\section{Natural function spaces}
\label{Sec: Natural function spaces}

\noindent In this short section we introduce the natural function spaces related to boundary value problems with $\L^2$-data and review some of their basic properties. For the sake of better reference we adopt notation from \cite{AA-Inventiones}.

\begin{definition}
\label{Def: X, Y, Y*}
On $\IR^+ \times \Omega$ define the Banach/Hilbert spaces
\begin{align*}
 \cX&:= \big\{f: \IR^+ \times \Omega \to \IC^n; \, \NT(f) \in \L^2(\Omega) \big\} \\
 \cY&:= \bigg\{f: \IR^+ \times \Omega \to \IC^n; \, \int_0^\infty \|f_t\|_{\L^2(\Omega)^n}^2 \; t \d t < \infty \bigg \}
\end{align*}
with their natural norms. Here, $\NT$ is the modified non-tangential maximal function introduced in Definition~\ref{Def: Non-tangential maximal function}. Let $\cY^*$ be the dual of $\cY$ relative to the unweighted space $\L^2(\IR^+; \L^2(\Omega)^n)$,
\begin{align*}
 \cY^*:= \bigg\{f: \IR^+ \times \Omega \to \IC^n; \, \int_0^\infty \|f_t\|_{\L^2(\Omega)^n}^2 \; \frac{\d t}{t} < \infty \bigg \}.
\end{align*}
\end{definition}

As outlined in the introduction, $\nablatx u \in \cX$ is a natural interior control for the Neumann and regularity problems, whereas we shall impose $\nablatx u \in \cY$ for the Dirichlet problem and deduce $\NT(u) \in \L^2(\Omega)$ \emph{a priori}. The space $\cX$ has $\cY^*$ as a subspace and lies locally inside $\cY$.

\begin{lemma}
\label{Lem: X inside Y*}
For $f: \IR^+ \times \Omega \to \IC^n$ it holds
\begin{align*}
 \sup_{t > 0} \frac{1}{t} \int_t^{2t} \|f_s\|_{\L^2(\Omega)^n}^2 \; \d s 
\lesssim \|\NT(f)\|_{\L^2(\Omega)}^2
\lesssim \int_0^\infty \|f_s\|_{\L^2(\Omega)^n}^2 \; \frac{\d s}{s}.
\end{align*}
In particular, $\cY^* \subseteq \cX$ with continuous inclusion.
\end{lemma}

\begin{proof}
We begin with the lower bound. To this end, we put $t_0:= c_1^{-1} \diam(\Omega)$ and consider the case $t \geq t_0$ first. Then for every $x \in \Omega$,
\begin{align*}
  \frac{1}{t} \int_t^{c_0t} \|f_s\|_2^2 \; \d s 
=  \frac{1}{t} \int_t^{c_0t} \int_{B(x,c_1 t) \cap \Omega} \abs{f_s(y)}^2 \; \d y \; \d s
\leq (c_0 - c_0^{-1})\abs{\Omega} \NT(f)(x)^2
\end{align*}
and integration over $x$ yields $\frac{1}{t} \int_t^{c_0t} \|f_s\|_2^2 \; \d s \lesssim \|\NT(f)\|_2^2$. In order to raise the upper limit for integration to $2t$, we simply have to add the respective estimates for $t= t, c_0t,\ldots,c_0^N t$, where $N \in \IN$ is minimal subject to $c_0^N \geq 2$. In the case $0 < t < t_0$ we pull the supremum outside the integral to obtain 
\begin{align*}
 \|\NT(f)\|_2^2 
\gtrsim \sup_{0<t<t_0} \frac{1}{t^{1+d}} \int_\Omega \int_{c_0^{-1}t}^{c_0t} \int_{B(x,c_1 t) \cap \Omega} \abs{f(s,y)}^2\; \d y \; \d s \; \d x,
\end{align*}
where we implicitly used $d$-Ahlfors regularity of $\Omega$. The right-hand side equals
\begin{align*}
\sup_{0<t<t_0} \frac{1}{t^{1+d}} \int_{c_0^{-1}t}^{c_0t} \int_\Omega  \int_\Omega \abs{f(s,y)}^2 \ind_{B(y,c_1 t) \cap \Omega}(x) \; \d x \; \d y \; \d s
\simeq \sup_{0<t<t_0} \frac{t^d}{t^{1+d}} \int_{c_0^{-1}t}^{c_0t} \|f_s\|_2^2 \; \d s
\end{align*}
and as before we may raise the upper limit for integration to $2t$ without any difficulty.

For the upper bound we use $d$-Ahlfors regularity of $\Omega$ to obtain the pointwise estimate
\begin{align}
\label{Eq1: X  inside Y*}
\bariint_{W(t,x)} |f_s(y)|^2 \; \d y \; \d s
\lesssim \int_{c_0^{-1} t}^{c_0 t} \int_{\Omega} \ind_{B(x, c_0 c_1 r s)}(y) |f_s(y)|^2 \; \d y \; \frac{\d s}{s^{1+d}}
\end{align}
uniformly for $0 < t \leq 1$ and $x \in \Omega$. On the large Whitney balls with $t \geq 1$ we similarly have
\begin{align*}
\bariint_{W(t,x)} |f_s(y)|^2 \; \d y \; \d s
\lesssim \int_{c_0^{-1} t}^{c_0 t} \int_{\Omega} |f_s(y)|^2 \; \d y \; \frac{\d s}{s}.
\end{align*}
From this the claim follows on taking the supremum over $t \leq 1$ and $t \geq 1$, respectively, and integrating with respect to $x \in \Omega$.
\end{proof}

If $f$ is contained in the subspace $\cY^*$ of $\cX$, then Whitney averages $\bariint_{W(t,x)} |f|^2$ are not only uniformly bounded in $t$ for a.e.\ $x \in \Omega$, but vanish in the limit $t \to 0$. More precisely, we have the following

\begin{lemma}
\label{Lem: Whitney averages vanish on Y*}
If $f \in \cY^*$, then averages $\barint_{t}^{2t} \|f_s\|_{\L^2(\Omega)^n}^2 \; \d s$ vanish as $t \to 0$ and $t \to \infty$, respectively and for almost every $x \in \Omega$ it holds
\begin{align*}
 \lim_{t \to 0} \bariint_{W(t,x)} |f(s,y)|^2 \; \d s \; \d y = 0.
\end{align*}
\end{lemma}

\begin{proof}
Since $\barint_{t}^{2t} \|f_s\|_{\L^2(\Omega)^n}^2 \; \d s \simeq \int_t^{2t} \|f_s\|_{\L^2(\Omega)^n}^2 \; \frac{\d s}{s}$, convergence of the averages follows from integrability of $\|f_s\|_{\L^2(\Omega)^n}^2$ with respect to the measure $\frac{\d s}{s}$. For the second claim let $0<t_0\leq 1$ be arbitrary. Taking the supremum over $t \leq t_0$ in \eqref{Eq1: X  inside Y*} and integrating with respect to $x \in \Omega$ leads to
\begin{align*}
\int_\Omega \sup_{0 < t \leq t_0} \bariint_{W(t,x)} |f_s(y)|^2 \; \d y \; \d s
\lesssim \int_0^{c_0 t_0} \|f_s\|_{\L^2(\Omega)^n}^2 \; \frac{\d s}{s}.
\end{align*}
Since $f \in \cY^*$, the right-hand side vanishes in the limit $t_0 \to 0$ and the conclusion follows.
\end{proof}

The following theorem gives a re-interpretation of the modified Carleson norm from Definition~\ref{Def: Non-tangential maximal function} as the norm of pointwise multiplication from $\cX$ into the smaller space $\cY^*$. When dealing with $t$-independent coefficients $A(t,x)$, this will be the manner in which we exploit finiteness of $\|A - A_0\|_C$ qualitatively. On $\Omega = \IR^d$ the first proof was given by Hyt\"onen and Ros\'{e}n \cite{Hytonen-Rosen}. Later, Huang gave a different proof (\cite[Thm.~3.4]{Yi}; the required result corresponds to the multiplication $T_2^{2,2} \leftrightarrow T^{\infty, \infty}_2 \cdot T_\infty^{2,2}$) which in fact only requires that $\Omega$ is doubling \cite[Rem.~6.3f.]{Yi}. In turn, this is guaranteed by our standing assumptions, see Remark~\ref{Rem: Omega doubling}.

\begin{theorem}[{\cite[Thm.~3.4]{Yi}}]
\label{Thm: * norm equivalent to Carleson}
For $\cE: \IR^+ \times \Omega \to \Lop(\IC^n)$ the norm of pointwise multiplication 
\begin{align*}
 \|\cE\|_* := \|\cE\|_{\cX \to \cY^*} = \sup_{\|f\|_\cX = 1} \|\cE f\|_{\cY^*}
\end{align*}
is equivalent to the modified Carleson norm $\|\cdot \|_C$.
\end{theorem}

\begin{remark}
\label{Rem: Comparison to standard Carleson norm}
\begin{enumerate}
 \item For $B$ an open ball with center $x_0 \in \Omega$ and radius $r(B)$ let $f_B$ be the characteristic function of the Carleson box $(0,r(B)) \times B$ (times a unit vector field). Splitting the supremum over $t>0$ in the definition of the non-tangential maximal function at $t=r(B)$, we readily find $\NT(f_B) \leq \ind_{B(x_0, r(B) + c_1 r(B))}$ pointwisely on $\Omega$. From the estimate $\|\cE f_B\|_{\cY^*} \leq \|\cE\|_C \|f_B\|_\cX$ we get that the modified Carleson dominates the standard Carleson norm:
\begin{align*}
 \sup_B \bigg( \frac{1}{|B \cap \Omega|} \iint_{(0,r(B)) \times (B \cap \Omega)} |\cE(t,x)|^2 \; \frac{\d t \d x}{t} \bigg)^{1/2} \lesssim \|\cE\|_C.
\end{align*}
 
 \item It holds $\|\cE\|_* \gtrsim \|\cE\|_\infty$: In fact, given $\eps > 0$ there exist $t>0$ and $f \in \L^2(\IR^+; \L^2(\Omega)^n)$ with support in $(t,2t)$ such that $\|\cE f\|_2/ \|f\|_2 \geq \|\cE\|_\infty - \eps$ and therefore Lemma~\ref{Lem: X inside Y*} implies 
 \begin{align*}
  \|\cE\|_* \geq \frac{\|\cE f\|_{\cY^*}}{\|f\|_\cX} \simeq \frac{t^{-1/2} \|\cE f\|_2}{t^{-1/2} \|f\|_2} \geq \|\cE\|_\infty - \eps.
 \end{align*}
\end{enumerate}
\end{remark}

Finally, we can give the proof of Lemma~\ref{Lem: Unique t-independent coefficients}.

\begin{proof}[Proof of Lemma~\ref{Lem: Unique t-independent coefficients}]
Having at hand the domination of the modified Carleson norm $\|\cdot\|_C$ by the standard Carleson norm, the proof is essentially the same as the one of Lemma~2.2 in \cite{AA-Inventiones}. The only modification is that in our setup $\L^\infty(\Omega) \times \gV \C_\Dir^\infty(\Omega)^m$ plays the role of a dense subset of bounded functions within the space on which $A$ is accretive (in \cite{AA-Inventiones} they use $\C_c^\infty(\IR^d)^n$-functions with curl-free tangential component).
\end{proof}
\section{Equivalence to a first-order system} 
\label{Sec: Equivalence of ES to a first order system}

\noindent In this section we prove that the second-order elliptic system with mixed lateral boundary conditions is equivalent to a non-autonomous evolution equation 
\begin{align}
\label{FO}
 \partial_t f_t + \D \B_t f_t = 0 \qquad (t>0),
\end{align}
where $\D$ is a self-adjoint first-order differential operators acting on the tangential variable $x$ and $\B$ is a bounded multiplication operator related to $A$ by an algebraic matrix transform. 

We begin by defining the relevant operators and function spaces. Recall from Section~\ref{Subsec: Weak solutions ES} that $\gV: \L^2(\Omega)^m \to \L^2(\Omega)^{dm}$ denotes the distributional gradient operator with domain $\V = \W_\Dir^{1,2}(\Omega)^m$. This yields a closed operator. The following Hardy and Poincar\'{e} inequalities entail that its range is closed and that it is injective if $\Dir$ is non-empty and otherwise has an $m$-dimensional nullspace containing only the constants.

\begin{proposition}[{\cite[Thm.~3.2/4]{Hardy-Poincare}, \cite[Thm.~4.4.2]{Ziemer}}]
\label{Prop: Hardy and Poincare}
Let $1 < p < \infty$ and suppose that $\Omega$ and $\Dir$ satisfy Assumption~\ref{Ass: General geometric assumption on Omega BVP}.
\begin{enumerate}
\def\theenumi{\roman{enumi}}
 \item \label{Hardy} If $\Dir \neq \emptyset$, then Hardy's inequality
\begin{align*}
 \int_\Omega |u(x)|^p \; \d x \lesssim \int_\Omega \Big|\frac{u(x)}{\dist_\Dir(x)}\Big|^p \; \d x \lesssim \int_\Omega \abs{\nabla u(x)}^p \; \d x \qquad (u \in \W_\Dir^{1,p}(\Omega))
\end{align*}
holds and $\W_\Dir^{1,p}(\Omega)$ is the largest subset of $\W^{1,p}(\Omega)$ on which the middle term is finite.
 \item \label{Poincare} If $\Dir = \emptyset$, then  Poincar\'{e}'s inequality
\begin{align*}
 \int_\Omega |u(x) - u_\Omega|^p \; \d x \lesssim \int_\Omega \abs{\nabla u(x)}^p \; \d x \qquad (u \in \W_\Dir^{1,p}(\Omega))
\end{align*}
 holds,  where $u_\Omega := \barint_\Omega u$ is the average of $u$ over $\Omega$.
\end{enumerate}
\end{proposition}

Integration by parts reveals $\C_c^\infty(\Omega)^{dm}$ as a subset of the domain $\dom((-\gV)^*)$ of the adjoint $(-\gV)^*: \L^2(\Omega)^{dm} \to \L^2(\Omega)^m$, on which this operator acts as the distributional divergence operator. Hence, we shall more suggestively write $\dV:= (-\gV)^*$. However, note carefully that under our very general geometric assumptions on $\Omega$ we do not have an explicit description for $\dom(\dV)$ as a space of distributions. The self-adjoint differential operator $\D$ in \eqref{FO} will turn out to be
\begin{align*}
 \D:= \begin{bmatrix} 0 & \dV \\ -\gV & 0 \end{bmatrix}
\end{align*}
with natural domain in $\L^2(\Omega)^n = \L^2(\Omega)^m \times \L^2(\Omega)^{dm}$. By
\begin{align*}
 \cH: = \cl{\Rg(\D)} = \Ke(\gV)^\pe \times \Rg(-\gV)
\end{align*}
we denote the closure of its range, where the orthogonal complement $\Ke(\gV)^\pe$ in $\L^2(\Omega)^m$ coincides with $\L^2(\Omega)^m$ provided $\Dir$ is non-empty and otherwise with the space of $\L^2(\Omega)^m$-functions with zero average on $\Omega$. 

In order to define the multiplication operator $\B$, we consider the decomposition $\IC^n = \IC^d \times \IC^{dm}$, which induces a block decomposition
\begin{align*}
 A(t,x) = \begin{bmatrix}
      A_{\pe \pe}(t,x) & A_{\pe \pa}(t,x) \\ A_{\pa \pe}(t,x) & A_{\pa \pa}(t,x)
     \end{bmatrix}
 \in \L^\infty(\IR^+ \times \Omega; \Lop(\IC^m \times \IC^{dm})).
\end{align*}
Choosing $f = \begin{bmatrix} \ind_E w \\ 0 \end{bmatrix}$ for any measurable $E \subseteq \Omega$ and any $w \in \IC^m$ in Assumption~\ref{Ass: Ellipticity for BVP} leads to $\Re (A_{\pe \pe}(t, x)w \cdot \cl{w}) \geq \lambda \abs{w}^2$ for a.e.\ $(t,x) \in \IR^+ \times \Omega$. By separability the exceptional set can be chosen independently of $w$. Hence, $A_{\pe \pe}$ is pointwise strictly accretive and in particular invertible in $\L^\infty(\IR^+ \times \Omega; \Lop(\IC^m))$. In the space $\L^\infty(\IR^+ \times \Omega; \Lop(\IC^n))$ we have the matrix-valued functions
\begin{align*}
 \underline{A}:= \begin{bmatrix}
  \Id & 0 \\ A_{\pa \pe} & A_{\pa \pa} \end{bmatrix},
\quad
 \overline{A}:= \begin{bmatrix}
  A_{\pe \pe} & A_{\pe \pa} \\ 0 & \Id \end{bmatrix},
\quad
 \underline{A} \overline{A}^{-1} = \begin{bmatrix}
  A_{\pe \pe}^{-1} & - A_{\pe \pe}^{-1} A_{\pe \pa} \\ A_{\pa \pe} A_{\pe \pe}^{-1} & A_{\pa \pa} - A_{\pa \pe} A_{\pe \pe}^{-1} A_{\pe \pa} \end{bmatrix}.
\end{align*}
With this notation the conormal gradient $\nabla_A$ can be written as $\cl{A} \nabla_{t,x}$. Finally, we take $\B$ as the bounded multiplication operator on $\L^2(\IR^+ \times \Omega)^n$ induced by $\underline{A} \overline{A}^{-1}$. Strict accretivity is preserved under the transformation $A \mapsto \B$. This follows from the subsequent lemma, whose purely algebraic proof is carried out exactly as in \cite[Prop.~4.1]{AA-Inventiones}.

\begin{lemma}
\label{Lem: B accretive on the range of D}
If $\lambda > 0$ is as in Assumption~\ref{Ass: Ellipticity for BVP}, then 
\begin{align*}
 \Re \scal{\B f}{f}_{\L^2(\Omega)^n} \geq \lambda \|\cl{A}\|_{\L^\infty(\IR^+ \times \Omega; \Lop(\IC^n))}^{-2}  \|f\|_{\L^2(\Omega)^n}^2 \qquad (f \in \L^2(\Omega)^m \times \Rg(-\gV)).
\end{align*}
\end{lemma}

By a \emph{formal} computation we find that 
\begin{align*}
 \divtx A \nablatx  u = 0
\end{align*}
implies
\begin{align*}
 \partial_t \nablaA u
= \begin{bmatrix} -\divx (A \nablatx  u)_\pa \\ \nablax  \partial_t u \end{bmatrix}
=  - \begin{bmatrix} 0 & \divx  \\ -\nablax  & 0 \end{bmatrix} \underline{A} \nablatx  u
= \D \B \nablaA u,
\end{align*}
that is $f = \nablaA u$ satisfies the first-order system \eqref{FO} in a formal sense. This fact is well-known in the case $\Omega = \IR^d$, see, e.g., \cite[Prop.~4.1]{AA-Inventiones}, but we stress that due to the lateral boundary conditions the argument for a bounded cylinder base $\Omega$ is more involved and cannot go through on a purely symbolic (i.e.\ distributional) level. Below, we make this correspondence precise using the following notion of weak solutions to the first-order system.

\begin{definition}
\label{Def: Weak solutions to FO}
A \emph{weak solution} to first-order system \eqref{FO} is a function $f \in \Lloc^2(\IR^+; \cH)$ such that
\begin{align}
\label{Eq: Weak solutions to FO}
 \int_0^\infty \bigscal{f_t}{\partial_t g_t}_{\L^2(\Omega)^n} \; \d t = \int_0^\infty \bigscal{\B_t f_t}{\D g_t}_{\L^2(\Omega)^n} \; \d t \qquad (g \in \C_c^\infty(\IR^+; \dom(\D))).
\end{align}
\end{definition}

\begin{remark}
\label{Rem: No flux condition captured by H}
If $\Dir = \emptyset$, then the tangential component of $\cH$ is the space of average-free $\L^2(\Omega)^m$-functions and thus captures the no-flux condition.
\end{remark}

\begin{proposition}
\label{Prop: f are conormals of u}
If $\Dir$ is non-empty, then there is a one-to-one correspondence between weak solutions $u$ to the second-order system with mixed lateral boundary conditions and weak solutions $f$ to the first-order system given by
\begin{align*}
 f = \nablaA  u.
\end{align*}
If $\Dir$ is empty, then this correspondence becomes one-to-one if $u$ is considered modulo constants.
\end{proposition}

\begin{proof}
The proof is subdivided into three steps. In order to increase readability, all $\L^2$-inner products are abbreviated by $\scal{\cdot}{\cdot}$.

\subsubsection*{\normalfont \itshape Step 1: Weak solutions are mapped to weak solutions}

\noindent Assume that $u$ is a weak solution to the second-order system and put $f:= \nablaA u$. Note that $f \in \Lloc^2(\IR^+; \cH)$ -- in the case $\Dir = \emptyset$ this is guaranteed by Lemma~\ref{Lem: Flux independent of t}. To see that $f$ satisfies \eqref{Eq: Weak solutions to FO}, fix an arbitrary $g \in \C_c^\infty(\IR^+; \dom(\D))$. Then $g_\perp$ is allowed as test function in Definition~\ref{Def: Weak solutions to ES} and  \eqref{Eq: Weak solutions to ES} rewrites as
\begin{align*}
\int_0^\infty \bigscal{(f_t)_\pe}{\partial_t (g_t)_\pe} \; \d t
= \int_0^\infty \bigscal{(\B_t f_t)_\pa}{(\D g_t)_\pa} \; \d t.
\end{align*}
For the tangential parts note $g_\pa \in \C_c^\infty(\IR^+; \dom((-\gV)^*))$, so that
\begin{align*}
\int_0^\infty \bigscal{(f_t)_\pa}{(\partial_t g_t)_\pa} \; \d t
&=-\int_0^\infty \bigscal{u_t}{\partial_t (-\gV)^* (g_t)_\pa} \; \d t.
\end{align*}
Integration by parts, taking into account that $g$ has compact support in the $t$-direction, leads to
\begin{align*}
\int_0^\infty \bigscal{(f_t)_\pa}{(\partial_t g_t)_\pa} \; \d t
=\int_0^\infty \bigscal{\partial_t u_t}{(-\gV)^* (g_t)_\pa} \; \d t
=\int_0^\infty \bigscal{(\B_t f_t)_\pe}{(\D g_t)_\pe} \; \d t.
\end{align*}
Adding the identities obtained for the perpendicular and tangential parts yields \eqref{Eq: Weak solutions to FO}.

\subsubsection*{\normalfont \itshape Step 2: The correspondence is onto}

\noindent Assume that $f \in \Lloc^2(\IR^+; \cH)$ is a weak solution to the first-order system. Then, by definition, $f_\pa \in \Lloc^2(\IR^+; \Rg(-\gV))$. We first consider the case that the Dirichlet part $\Dir$ is non-empty. In virtue of Poincar\'{e}'s inequality, $\gV$ is an isomorphism from $\V$ onto $\Rg(\gV)$. Hence, there exists a potential $u \in \Lloc^2(\IR^+; \V)$ such that $\nablax  u = f_\pa$. We claim 
\begin{align}
\label{Eq3: f are conormals of u}
 u \in \Wloc^{1,2}(\IR^+; \L^2(\Omega)^m) \quad \text{with} \quad \partial_t u = (\B f)_\pe.
\end{align}
Indeed, since $\Rg((-\gV)^*)$ is dense in $\L^2(\Omega; \IC)^m$ by injectivity of $-\gV$, it suffices to prove
\begin{align*}
 \biggscal{\int_0^\infty u_t \partial_t \eta(t) \; \d t}{(-\gV)^* y}
= \biggscal{- \int_0^\infty (\B_t f_t)_\pe \eta(t) \; \d t}{(-\gV)^* y}
\end{align*}
for each $\eta \in \C_c^\infty(\IR^+; \IR)$ and each $y \in \dom((-\gV)^*)$. Here, the left-hand side equals
\begin{align*}
\int_0^\infty \bigscal{(-\gV)u_t}{\partial_t \eta(t)y} \; \d t
&= - \int_0^\infty \bigscal{(f_t)_\pa}{\partial_t \eta(t) y} \; \d t,
\intertext{and we can use $g(t):= \begin{bmatrix} 0 \\ \eta(t)y \end{bmatrix}$ as test function in \eqref{Eq: Weak solutions to FO} to continue the chain of equalities by}
&= - \int_0^\infty \bigscal{f_t}{\partial_t g_t} \; \d t \\
&= - \int_0^\infty \bigscal{\B_t f_t}{\D g_t} \; \d t \\
&= - \int_0^\infty \bigscal{(\B_t f_t)_\pe}{\eta(t) (-\gV)^*y} \; \d t,
\end{align*}
which coincides with the right-hand side of the identity in question. Summing up, $u$ has the required regularity and satisfies 
\begin{align*}
 \nablaA  u = \cl{A} \begin{bmatrix} (Bf)_\pe \\ f_\pa \end{bmatrix}= f.
\end{align*}
To see that $u$ is a weak solution to the second-order system, let $v \in \C_c^\infty(\IR^+; \V)$. As $g:=\begin{bmatrix} v \\ 0 \end{bmatrix}$ is allowed as test function in \eqref{Eq: Weak solutions to FO},
\begin{align*}
0 
= \int_0^\infty \bigscal{(f_t)_\pe}{\partial_t v_t} + \bigscal{(\B f_t)_\pa}{\nablax  v_t} \; \d t
=\int_0^\infty \bigscal{A \nablatx u}{\nablatx  v} \; \d t
\end{align*}
as required.

Now, consider the slightly more involved case that the lateral Dirichlet part is empty. Denote by $\V_0 \subseteq \V$ the subspace of functions with zero average on $\Omega$. Poincar\'{e}'s inequality on $\V_0$ allows to construct a potential $\widetilde{u} \in \Lloc^2(\IR^+; \V_0)$ such that $\nablax  \widetilde{u} = f_\pa$. Repeating the argument succeeding \eqref{Eq3: f are conormals of u}, at least yields that for every $\eta \in \C_c^\infty(\IR^+; \IR)$ the $\L^2$-valued integral
\begin{align*}
 \int_0^\infty \widetilde{u}_t \partial_t \eta(t) \; \d t + \int_0^\infty (\B_t f_t)_\pe \eta(t) \; \d t 
\end{align*}
is contained in $\cl{\Rg((-\gV)^*)}^\pe = \Ke(\gV)$ and hence is a constant function on $\Omega$. Its value is determined as the average integral over $\Omega$. Since $\widetilde{u}_t \in \V_0$ for almost every $t>0$, it follows
\begin{align*}
 \int_0^\infty \widetilde{u}_t \partial_t \eta(t) \; \d t + \int_0^\infty (\B_t f_t)_\pe \eta(t) \; \d t = \int_0^\infty \eta(t) \bigg(\barint_\Omega (\B_t f_t)_\pe \; \d x \bigg) \; \d t
\end{align*}
for every $\eta \in \C_c^\infty(\IR^+; \IR))$, that is, $\partial_t \widetilde{u}_t = (\B_t f_t)_\pe - \barint_\Omega (\B_t f_t)_\pe$ in the sense of $\Wloc^{1,2}(\IR^+; \L^2(\Omega)^m)$. In order to correct the right-hand side, let $H \in \Wloc^{1,2}(\IR^+; \IC)$ be an anti-derivative of $t \mapsto \barint_\Omega (\B_t f_t)_\pe \; \d x$. Note that 
\begin{align*}
 u := \widetilde{u} + H \in \Lloc^2(\IR^+; \V) \cap \Wloc^{1,2}(\IR^+; \L^2(\Omega)^m)
\end{align*}
since constant functions on $\Omega$ are contained in $\V$ and that by construction $\partial_t u = (\B f)_\pe$ and $\nablax  u = \nablax  \widetilde{u} = f_\pa$. As in the case of non-empty Dirichlet part this implies that $u$ is a weak solution to the second-order system satisfying $\nablaA u = f$. Note that the no-flux condition automatically holds since $(\nablaA u)_\pe = f_\pe \in \cH_\pe$ is average-free.

\subsubsection*{\normalfont \itshape Step 3: The correspondence is one-one}

\noindent If $u$ is a weak solution with $\nablaA  u = 0$, then $\nablatx  u = 0$ by invertibility of $\cl{A}$. Thus, $u$ is constant on the domain $\IR^+ \times \Omega$. If in addition $\Dir \neq \emptyset$, then $u = 0$ by Poincar\'{e}'s inequality.
\end{proof}
\section{Quadratic estimates for $\D \B_0$ and $\B_0 \D$}
\label{Sec: Quadratic estimates for DB0}

\noindent We begin our study of the ``infinitesimal generator'' of the first-order system $\partial_t f_t + \D \B_t f_t = 0$ in case of \emph{$t$-independent} coefficients $A(t,x) = A_0(x)$ for all $t>0$. This implies that $\B(t,x) = \B_0(x)$ is $t$-independent as well and it will be convenient to identify $\B_0$ with a bounded accretive multiplication operator on $\L^2(\Omega)^n$.

Recall that an operator $T$ in a Hilbert space $\cK$ is called \emph{bisectorial} of angle $\omega \in (0, \frac{\pi}{2})$ if its spectrum $\sigma(T)$ is contained in the closure of the double sector
\begin{align*}
 \S_\omega := \{z \in \IC; \, |\arg z| < \omega \text{ or } |\arg z - \pi| < \omega \}
\end{align*}
and if the mapping $\lambda \mapsto \lambda(\lambda - T)^{-1}$ is uniformly bounded on $\S_\psi$ for every $\psi \in (\omega, \frac{\pi}{2})$. Thanks to Lemma~\ref{Lem: B accretive on the range of D} the concrete generator $\D \B_0$ defined in the previous section fits the premise of the following classical result.

\begin{proposition}[{\cite[Prop.~3.3]{AAM-ArkMath}, \cite[Prop.~6.2.17]{EigeneDiss}}]
\label{Prop: DB properties} Let $\D$ be a self-adjoint operator in a Hilbert space $\cK$ and let $\B_0 \in \Lop(\cK)$. If $\B_0$ is accretive on $\Rg(\D)$, that is, there exists $\kappa > 0$ such that $\Re \scal{\B_0 u}{u} \geq \kappa \|u\|^2$ for all $u \in \Rg(\D)$, then the following hold true and implicit constant depend only upon $\kappa$ and an upper bound for the norm of $\B_0$.
\begin{enumerate}
 \item The operator $\D \B_0$ has range $\Rg(\D \B_0) = \Rg(\D)$ and null space $\Ke(\D \B_0) = \B_0^{-1} \Ke(\D)$ such that topologically but in general non-orthogonally
 \begin{align*}
  \cK = \Ke(\D \B_0) \oplus \cl{\Rg(\D \B_0)}.
 \end{align*} 
 Similarly, $\B_0 \D$ has range $\Rg(\B_0 \D) = \B_0 \Rg(D)$, null space $\Ke(\B_0 \D) = \Ke(\D)$, and induces a topological splitting
 \begin{align*}
  \cK = \Ke(\B_0 \D) \oplus \cl{\Rg(\B_0 \D)}.
 \end{align*}
 \item The operators $\D \B_0$ and $\B_0 \D$ are bisectorial of angle $\omega: =\arctan(\frac{\|\B_0\|_{\cK \to \cK}}{\kappa})$.
\end{enumerate}
\end{proposition}

Proposition~\ref{Prop: DB properties} holds with $\B_0^*$ in place of $\B_0$ since this operator satisfies the same accretivity condition. It will also be useful to know the adjoint of the \emph{injective part}  $\D \B_0|_{\cl{\Rg(\D)}}$, that is, the maximal restriction of $\D\B_0$ to an operator on $\cl{\Rg(\D \B_0)}$.

\begin{corollary}
\label{Cor: Adjoint of injective part of DB}
In the setup of Proposition~\ref{Prop: DB properties} the Hilbert space adjoint of $\D \B_0|_{\cl{\Rg(\D)}}$ in $\cl{\Rg(\D)}$ is given by $P \B_0^* \D|_{\cl{\Rg(\D)}}$, where $P$ is the orthogonal projection in $\cK$ onto $\cl{\Rg(\D)}$. Moreover, $\|P \B_0^* \D u \| \simeq \|\D u\|$ for all $u \in \dom(\D) \cap \cl{\Rg(\D)}$.
\end{corollary}

\begin{proof}
It is straightforward to check that the adjoint of $\D \B_0|_{\cl{\Rg(\D)}}$ extends $P \B_0^* \D|_{\cl{\Rg(\D)}}$. To obtain equality it suffices to note that these operators share a common resolvent as both are bisectorial: In fact, for the restriction $\D \B_0|_{\cl{\Rg(\D)}}$ this is immediate by abstract properties of bisectorial operators \cite{Haase, EigeneDiss} and $P \B_0^* \D|_{\cl{\Rg(\D)}}$ factorizes as $(P \B_0^*|_{\cl{\Rg(\D)}})(\D|_{\cl{\Rg(\D)}})$ in the sense of Proposition~\ref{Prop: DB properties}. Finally, the required equivalence of norms follows by accretivity of $P \B_0^*|_{\cl{\Rg(\D)}}$.
\end{proof}

As our main result in this section we prove that $\D \B_0$ and the closely related operator $\B_0 \D$ satisfy quadratic estimates. This will pave the way for everything that follows in this paper.

\begin{theorem}
\label{Thm: Quadratic estimates fo DB}
Let Assumption~\ref{Ass: General geometric assumption on Omega BVP} be satisfied. Let $\B_0$ be a multiplication operator induced by an $\L^\infty(\Omega; \Lop(\IC^n))$-function and suppose that $\B_0$ is accretive on $\Rg(\D)$. If $T = \D \B_0$ or $T=\B_0 \D$, then there are quadratic estimates
\begin{align*}
 \int_0^\infty \|t T(1+t^2 T^2)^{-1} u\|_{\L^2(\Omega)^n}^2 \; \frac{\d t}{t} \simeq \|u\|_{\L^2(\Omega)^n}^2 \qquad (u \in \cl{\Rg(T)}).
\end{align*}
Implicit constants can be chosen uniformly for $\B_0$ in a bounded subset of $\L^\infty(\Omega; \Lop(\IC^n))$ whose members satisfy a uniform lower bound in the accretivity condition.
\end{theorem}

Before we give the proof of this theorem, let us point out its important consequences. In the following we require basic knowledge on the holomorphic functional calculus for bisectorial operators, allowing to plug in such operators into suitable holomorphic functions defined on a complex bisector enclosing their spectrum. A reader without background in this field can refer to the various comprehensive treatments in the literature, for instance \cite{Haase, OperHam, EigeneDiss}.

\begin{enumerate}[label=(\Alph*)]
\def\theenumi{\Alph{enumi}}
 \item \label{Hinfty 1}The quadratic estimates in Theorem~\ref{Thm: Quadratic estimates fo DB} remain true for $f(tT)$ in place of $t T(1+t^2 T^2)^{-1}$ for every holomorphic $f$ defined on a bisector $\S_\psi$ with opening angle $\psi \in (\omega, \frac{\pi}{2})$ that decays polynomially to zero at $0$ and $\infty$ and is non-zero on both connected components of $\S_\psi$. Quadratic estimates on $\cl{\Rg(T)}$ imply that $T$ has a bounded $\H^\infty(\S_\psi)$-calculus on $\cl{\Rg(T)}$, i.e., for each bounded holomorphic function $f$ defined on $\S_\psi$ the operator $f(T)$ in $\cl{\Rg(T)}$ satisfies
\begin{align*}
 \|f(T)\|_{\cl{\Rg(T)} \to \cl{\Rg(T)}} \lesssim \|f\|_{\L^\infty(\S_\psi)}.
\end{align*}
Implicit constants depend only on $\psi$ and the constants in Theorem~\ref{Thm: Quadratic estimates fo DB}. Consequently, the bounds for the $\H^\infty(\S_\psi)$-calculus enjoy again a uniformity property in $\B_0$. 
\end{enumerate}
The most important operators defined in the functional calculus for $\D \B_0$ will be listed below. Proofs of all further statements are carried out in detail e.g.\ in \cite[Sec.~3.3.4]{EigeneDiss}.
\begin{enumerate}[label=(\Alph*)]
\def\theenumi{\Alph{enumi}}
\setcounter{enumi}{1}
 \item \label{Hinfty 2} The characteristic functions $\ind_{\IC^\pm}$ of the right and left complex half planes give the \emph{generalized Hardy projections} $E_0^\pm:= \ind_{\IC^\pm}(\D \B_0)$ on $\cl{\Rg(\D \B_0)} = \cH$, see Proposition~\ref{Prop: DB properties} for the last equality. Their boundedness yields a topological spectral decomposition $\cH = E_0^+ \cH \oplus E_0^- \cH$.
 \item \label{Hinfty 3}For $z \in \IC$ let $[z]:= \sqrt{z^2}$. The exponential functions $z \mapsto \e^{-t [z]}$, $t\geq 0$, give the operators $\e^{-t [\D \B_0]}$, $t \geq 0$, on $\L^2(\Omega)^n$. They form the bounded holomorphic semigroup generated by $-[\D \B_0]$. Their restrictions to $E_0^\pm \cH$ are the bounded holomorphic semigroups generated by $\mp \D \B_0|_{E_0^\pm \cH}$.
\end{enumerate}
Similar operators can of course be defined in the functional calculus for $\B_0 \D$. They are related by the following intertwining and  duality relations.
\begin{enumerate}[label=(\Alph*)]
\def\theenumi{\Alph{enumi}}
\setcounter{enumi}{3}
 \item \label{Hinfty 4} For $f$ bounded and holomorphic on $\S_\psi$, $\psi \in (\omega, \frac{\pi}{2})$,  it holds
\begin{align*}
 \B_0 f(\D \B_0)u = f(\B_0 \D) \B_0u \qquad (u \in \cH)
\end{align*}
In fact, this relation is readily checked for resolvents $f(z) = (\lambda - z)^{-1}$, $\lambda \in \IC \setminus \S_\omega$, and extends to general $f$ by the construction of the functional calculus.

 \item \label{Hinfty 5} Since $(\D \B_0)^* = \B_0^* \D$, for every holomorphic function $f$ on $\S_\psi$, $\psi \in (\omega, \frac{\pi}{2})$ with at most polynomial growth at $|z|=0$ and $|z| = \infty$ it holds
\begin{align*}
 f(\D \B_0)^* = f^*(\B_0^* \D) \quad \text{where} \quad f^*(z) = \cl{f(\cl{z})}.
\end{align*} 
\end{enumerate}

Uniformity of the bounds in \eqref{Hinfty 1} entails holomorphic dependence of the $\H^\infty$-calculus for $\D \B_0$ with respect to the multiplicative perturbation $\B_0$. Most importantly for us, the Hardy projections $E_0^\pm$ depend continuously on $\B_0$. For the reader's convenience, we shortly sketch the standard argument allowing to prove

\begin{proposition}
\label{Prop: Holomorphic dependence of Hinfty calculus}
Let $U \subseteq \IC$ be open and let $\B_0: U \to \Lop(\L^2(\Omega)^n)$ be a holomorphic function. Assume that each operator $\B_0(z)$, $z \in U$, is induced by an $\L^\infty(\Omega; \Lop(\IC^n))$-function and that there exists $K, \kappa > 0$ such that
\begin{align*}
 \Re \scal{\B_0(z) u}{u}_{\L^2(\Omega)^n} \geq \kappa \|u\|_{\L^2(\Omega)^n}^2 \quad \text{and} \quad \|\B_0(z) \|_{\L^\infty(\Omega; \Lop(\IC^n))} \leq K \qquad (z \in U, u \in \Rg(\D)).
\end{align*}
Then for each $\psi \in (\arctan(\frac{K}{\kappa}), \frac{\pi}{2})$ and each $f \in \H^\infty(\S_\psi)$ the function $z \mapsto f(\D \B_0(z)): U \to \Lop(\cH)$ is holomorphic.
\end{proposition}

\begin{proof}
We abbreviate $\L^2 := \L^2(\Omega)^n$. First note that $\B_0^*: U \to \Lop(\L^2)$ is holomorphic as well. Given $\lambda \in \IC \setminus \S_\psi$, holomorphic dependence of $(\lambda - \B_0(z)^*\D)^{-1} \in \Lop(\L^2)$ on $z$ follows on using the identity
\begin{align*}
 &(\lambda - \B_0(z_0)^*\D)^{-1} - (\lambda - \B_0(z_1)^*\D)^{-1} \\
&= (\lambda - \B_0(z_0)^*\D)^{-1}\; (\B_0(z_0)^* -\B_0(z_1)^*) \; \D (\lambda - \B_0(z_1)^*\D)^{-1} \qquad (z_0, z_1 \in U)
\end{align*}
on difference quotients. For this we have crucially employed that the domain of $\B_0(z)^* \D$ is independent of $z$. Taking adjoints, holomorphy of $(\cl{\lambda} - \D \B_0(z))^{-1}$ follows. Next, if $f \in \H_0^\infty(\S_\psi)$, the subset of functions in $\H^\infty(\S_\psi)$ decaying polynomially to zero at $\abs{z} =0$ and $\abs{z} =\infty$, then $f(\D \B_0(z)) \in \Lop(\L^2)$ is defined via a contour integral and holomorphic dependence on $z$ can be inferred from Morera's theorem. Finally, let $f \in \H^\infty(\S_\psi)$. By equivalence of weak and strong holomorphy \cite[Prop.~A.3]{ABHN} it suffices to prove holomorphic dependence of $f(\D \B_0(z))u \in \cH$ on $z$ for each fixed $u \in \cH$. Take a bounded sequence $\{f_n\}_n \subseteq \H_0^\infty(\S_\psi)$ that converges to $f$ pointwisely on $\S_\psi$. Thanks to \eqref{Hinfty 2}, $\{f_n(\D \B_0(z))u\}_n$ is a bounded sequence of bounded $\cH$-valued holomorphic functions on $U$. The convergence lemma adapted to bisectorial operators (\cite[Prop.~5.1.4]{Haase} or \cite[Prop.~3.3.5]{EigeneDiss}) yields pointwise convergence toward $f(\D \B_0(z))u$. So, holomorphy follows from Vitali's theorem from complex analysis \cite[Thm.~A.5]{ABHN}.
\end{proof}

\begin{proof}[Proof of Theorem~\ref{Thm: Quadratic estimates fo DB}]
The proof builds upon the tools introduced in \cite{AKM} and refined in \cite{Laplace-Extrapolation, Darmstadt-KatoMixedBoundary} in order to resolve the Kato problem for mixed boundary conditions under Assumption~\ref{Ass: General geometric assumption on Omega BVP}. The first ingredient are quadratic estimates for perturbed Dirac type operators acting on $\L^2(\Omega)$.

\begin{proposition}[{\cite[Thm.~3.3]{Laplace-Extrapolation}}]
\label{Prop: Pb theorem}
Let Assumption~\ref{Ass: General geometric assumption on Omega BVP} be satisfied and let $k \in \IN$. On the Hilbert space $\L^2 = \L^2(\Omega)^{mk}$ consider a triple of operators $\{\g, B_1, B_2\}$ satisfying the following hypotheses. 
\begin{enumerate}
   \item[$\mathrm{(H1)}$] $\g$ is \emph{nilpotent}\index{nilpotent operator}, i.e.\ closed, densely defined, and satisfies $\Rg(\g) \subseteq \Ke(\g)$. 

   \item[$\mathrm{(H2)}$] $B_1$ and $B_2$ are defined on the whole of $\L^2$. There exist $\kappa_1 , \kappa_2 > 0$ such that they satisfy the \emph{accretivity conditions}
    \begin{align*}
    \Re \scal{B_1 u}{u}_2 &\geq \kappa_1 \| u \|_2^2 \qquad (u \in \Rg(\g^*)), \\
    \Re \scal{B_2 u}{u}_2 &\geq \kappa_2 \| u \|_2^2 \qquad (u \in \Rg(\g))
    \end{align*}
   and there exist $K_1, K_2$ such that they satisfy the \emph{boundedness conditions}
    \begin{align*}
     \|B_1 u\|_2 \leq K_1 \|u\|_2 \quad \text{and} \quad \|B_2 u\|_2 \leq K_2 \|u\|_2 \qquad (u \in \L^2).
    \end{align*}

   \item[$\mathrm{(H3)}$] $B_2 B_1$ maps $\Rg(\g^*)$ into $\Ke(\g^*)$ and $B_1 B_2$ maps $\Rg(\g)$ into $\Ke(\g)$. 
   \item[$\mathrm{(H4)}$] $B_1$ and $B_2$ are multiplication operators induced by $\L^\infty(\Omega; \Lop(\IC^{mk}))$-functions.
 
   \item[$\mathrm{(H5)}$]For every $\varphi \in \C_c^{\infty} (\IR^d ; \IC)$ multiplication by $\varphi$ maps $\dom(\g)$ into itself. The commutator $[\g , \varphi]$ is defined on $\dom(\g)$ and acts by multiplication with some $c_\varphi \in \L^\infty(\Omega; \Lop(\IC^{mk}))$ satisfying pointwise bounds $\lvert c_{\varphi}^{i,j} (x) \rvert \lesssim \lvert \nabla \varphi(x) \rvert$ almost everywhere on $\Omega$.
   \item[$\mathrm{(H6)}$] Let $\Upsilon$ by either $\g$ or $\g^*$. For every open ball $B$ centered in $\Omega$ and for all $u \in \dom(\Upsilon)$ with compact support in $B \cap \Omega$ it holds $\lvert \int_{\Omega} \Upsilon u  \rvert \lesssim \lvert B \rvert^{\frac{1}{2}} \| u \|_2$.
   \item[$\mathrm{(H7)}$] There exist $\beta_1 , \beta_2 \in (0 , 1]$ such that the fractional powers of $\Pi:= \g + \g^*$ satisfy
    \begin{align*}
    \| u \|_{[\cH , \V^k]_{\beta_1}} \lesssim \| ( \Pi^2 )^{\beta_1 / 2} u \|_2 \qquad \text{and} \qquad \| v \|_{[\cH , \V^k]_{\beta_2}} \lesssim \| ( \Pi^2 )^{\beta_2 / 2} v \|_2
    \end{align*}
   for all $u \in \Rg(\g^*) \cap \dom(\Pi^2)$ and all $v \in \Rg(\g) \cap \dom(\Pi^2)$.
\end{enumerate}
Then $\pb:= \g + B_1 \g^* B_2$ satisfies quadratic estimates
\begin{align*}
 \int_0^\infty \|t\pb(1+ t^2 \pb^2)^{-1}u\|_2^2 \; \frac{\d t}{t} \simeq \|u\|_2^2 \qquad (u \in \cl{\Rg(\pb)}),
\end{align*}
where implicit constants depend on $B_1$ and $B_2$ only through the constants quantified in $\mathrm{(H2)}$.
\end{proposition}

The second ingredient are extrapolation properties for the \emph{weak Laplacian} with form domain $\V$ defined by $\Delta_\V  := -\dV \gV$. For this we need the \emph{$\L^2$-Bessel potential spaces} $\H^{\alpha,2}(\Omega)$, $\alpha > 0$, on $\Omega$, defined as the restrictions of the ordinary Bessel potential spaces $\H^{\alpha,2}(\IR^d)$. In \cite{Darmstadt-KatoMixedBoundary} the subsequently listed results have been for obtained spaces of scalar-valued functions but they extend to finite Cartesian products in an obvious manner.

\begin{proposition}[{\cite[Thm.~7.1]{Darmstadt-KatoMixedBoundary}}]
\label{Prop: Form domain interpolation}
Let Assumption~\ref{Ass: General geometric assumption on Omega BVP} be satisfied and let $k \in \IN$. Then up to equivalent norms $[\L^2(\Omega)^{mk}, \V^k]_\alpha = \H^{\alpha,2}(\Omega)^{mk}$ for every $\alpha \in (0, \frac{1}{2})$.
\end{proposition}

\begin{proposition}[{\cite[Thm.~4.4]{Darmstadt-KatoMixedBoundary}}]
\label{Prop: Fractional domains of Laplacian}
Let Assumption~\ref{Ass: General geometric assumption on Omega BVP} be satisfied. Then there exists $\alpha \in (0, \frac{1}{2})$ such that $\dom((-\Delta_\V)^{\alpha/2}) = \H^{\alpha,2}(\Omega)^m$ with equivalent norms and $\dom((-\Delta_\V)^{1/2 + \alpha/2}) \subseteq \H^{1+\alpha,2}(\Omega)^m$ with continuous inclusion.
\end{proposition}

\begin{lemma}[Fractional Poincar\'{e} inequality]
\label{Lem: Fractional Poincare}
Let $\alpha \in (0,1)$. Under Assumption~\ref{Ass: General geometric assumption on Omega BVP} it holds 
\begin{align*}
 \|u\|_{\L^2(\Omega)^m} \lesssim \|(-\Delta_\V)^\alpha u\|_{\L^2(\Omega)^m} \qquad (u \in \dom(\Delta_\V) \cap \cl{\Rg(\Delta_\V)}).
\end{align*}
\end{lemma}

\begin{proof}
The restriction $B:=-\Delta_\V|_{\cl{\Rg(\Delta_\V)}}$ is an invertible maximal accretivity operator on $\cl{\Rg(\Delta_V)}$. Essentially, this is by Poincar\'{e}'s inequality, see also \cite[p.~1431]{Darmstadt-KatoMixedBoundary}. Invertibility inherits to the fractional powers \cite[Prop.~3.1.1]{Haase}, so that $\|u\|_2 \lesssim \|B^\alpha u\|_2$ holds for all $u \in \dom(B^\alpha)$. The conclusion follows since $B^\alpha$ is the restriction of $(-\Delta_\V)^\alpha$ to $\cl{\Rg(\Delta_\V)}$ with domain $\dom((-\Delta_\V)^\alpha) \cap \cl{\Rg(\Delta_\V)}$, see \cite[Prop.~2.6.5]{Haase}.
\end{proof}

\begin{remark}
\label{Rem: Poincare for fractional powers of Laplacian}
If $\alpha = \frac{1}{2}$, then Lemma~\ref{Lem: Fractional Poincare} is the Poincar\'{e} inequality $\|u\|_{\L^2(\Omega)^m} \lesssim \|\nabla u \|_{\L^2(\Omega)^{dm}}$. This is due to the Kato estimate $(-\Delta_V)^{1/2} \sim \gV$, see \cite[Lem.~4.3]{Darmstadt-KatoMixedBoundary} for details.
\end{remark}

In order to complete the proof of Theorem~\ref{Thm: Quadratic estimates fo DB}, we apply Proposition~\ref{Prop: Pb theorem} on $\L^2(\Omega)^n \times \L^2(\Omega)^n$ to the operator matrices
\begin{align*}
 \g:= \begin{bmatrix} 0 & 0 \\ \D & 0 \end{bmatrix},
\quad
 B_1:= \begin{bmatrix} \B_0 & 0 \\ 0 & 0 \end{bmatrix},
\quad \text{and} \quad
 B_2:= \begin{bmatrix} 0 & 0 \\ 0 & \B_0 \end{bmatrix}.
\end{align*}
For these choices
\begin{align*}
 \pb := \begin{bmatrix} 0 & \B_0 \D \B_0 \\ \D & 0 \end{bmatrix},
\quad
  t \pb (1+ t^2 \pb^2)^{-1} = \begin{bmatrix} 0 & t \B_0 \D \B_0 (1+t^2 (\D \B_0)^2)^{-1} \\ \D (1+t^2 (\B_0 \D)^2)^{-1} & 0 \end{bmatrix}.
\end{align*}
Since $\Rg(\B_0 \D) = \Rg (\B_0 \D \B_0)$ and $\Rg(\D \B_0) = \Rg(\D)$ by Proposition~\ref{Prop: DB properties} and as $\B_0$ is bounded and accretive on $\cl{\Rg(\D)}$, we readily see that both quadratic estimates required in the theorem follow from quadratic estimates for $\pb$. 

This being said, it remains to check (H1) - (H7). In fact (H1) - (H4) are met by definition and (H5) and (H6) follow from the product rule and since the integral over the gradient of a compactly supported function vanishes. The only hypothesis that requires a closer inspection is the last one, which due to the symmetry of $\Pi$ is equivalent to the following:
\begin{align*}
\text{There exists $\alpha \in (0,1]$ such that $\|u\|_{[\L^2, \V]_\alpha} \lesssim \|(\D^2)^{\alpha/2} u \|_2$ for all $u \in \Rg(\D) \cap \dom(\D^2)$.}
\end{align*}
The difficulty lies in that this is a coercivity estimate for a pure first order differential operator. Inevitably, we have to factor out constants if the Dirichlet part of $\bd \Omega$ is empty: Let $\alpha$ be as in Proposition~\ref{Prop: Fractional domains of Laplacian} and fix $u \in \Rg(\D) \cap \dom(\D^2)$. Since
\begin{align*}
 \D = \begin{bmatrix} 0 & \dV \\ -\gV & 0 \end{bmatrix} 
\quad \text{and} \quad
 \D^2 = \begin{bmatrix} -\Delta_\V & 0 \\ 0 & (-\gV)\dV \end{bmatrix}
\end{align*}
it follows $u_\pe \in \dom(\Delta_\V) \cap \Rg(\dV)$ and $u_\pa = -\nabla_\V v_\pe$ for some $v_\pe \in \dom(\Delta_\V)$. Note that $-\gV$ and $(-\Delta_\V)^{1/2}$ share the same nullspace -- this is due to the Kato estimate $(-\Delta_V)^{1/2} \sim \gV$ for the self-adjoint operator $-\Delta_\V$. Since the nullspace of fractional powers is independent of their positive exponent \cite[Prop.~3.1.1]{Haase},
\begin{align*}
 \cl{\Rg(\dV)} = \Ke(-\gV)^\pe = \Ke((- \Delta_\V)^{1/2})^\pe = \Ke(-\Delta_\V)^\pe = \cl{\Rg(\Delta_\V)}
\end{align*}
showing $u_\pe \in \cl{\Rg(\Delta_\V)}$. Due $\Ke(\gV) = \Ke(-\Delta_\V)$ and $\L^2(\Omega)^m = \Ke(-\Delta_\V) \oplus  \cl{\Rg(-\Delta_V)}$ we can also assume $v_\pe \in \cl{\Rg(\Delta_V)}$. Starting out with the identity
\begin{align*}
 (\D^2)^{\alpha/2} u 
= (\D^2)^{\alpha/2} \bigg(\begin{bmatrix} u_\pe \\ 0 \end{bmatrix} + \D \begin{bmatrix} v_\pe \\ 0 \end{bmatrix} \bigg)
= \begin{bmatrix} (-\Delta_\V)^{\alpha/2} u_\pe \\ - \nabla_\V (-\Delta_\V)^{\alpha/2} v_\pe \end{bmatrix},
\end{align*}
where due to the Kato estimate we may freely replace $\gV$ by $(-\Delta_\V)^{1/2}$ as soon as it comes to $\L^2$-norms, Lemma~\ref{Lem: Fractional Poincare} yields
\begin{align*}
 \|(\D^2)^{\alpha/2} u \|_2^2 \simeq \|u_\pe\|_{\dom((-\Delta_\V)^{\alpha/2})}^2 + \|v_\pe\|_{\dom((-\Delta_\V)^{1/2+ \alpha/2})}^2.
\end{align*}
On the other hand, Proposition~\ref{Prop: Form domain interpolation} yields
\begin{align*}
 \|u\|_{[\L^2, \V^{1+d}]_\alpha} ^2
\simeq \|u_\pe\|_{\H^{\alpha,2}}^2 + \|\nabla_\V v_\pe\|_{\H^{\alpha,2}}^2 
\leq \|u_\pe\|_{\H^{\alpha,2}}^2 + \|v_\pe\|_{\H^{1+\alpha,2}}^2,
\end{align*}
and invoking Proposition~\ref{Prop: Fractional domains of Laplacian} the required estimate $\|u\|_{[\L^2, \V]_\alpha} \lesssim \|(\D^2)^{\alpha/2} u \|_2$ follows.
\end{proof}
\section{Analysis of semigroup solutions to $t$-independent systems}
\label{Sec: Semigroup solutions to the first order equation}

\noindent In this this section we restrict ourselves to fixed $t$-independent coefficients $A(t,x) = A_0(x)$. The infinitesimal generator of the corresponding first-order system $\partial_t f + \D \B_0 f = 0$ is bisectorial and hence generates a bounded holomorphic semigroup on the positive Hardy space $E_0^+ \cH$ as we have seen in Section~\ref{Sec: Quadratic estimates for DB0}. Thus, we can construct semigroup solutions to the first-order system with the following additional limits and regularity.

\begin{proposition}
\label{Prop: Semigroup solution etDB}
To each $h^+ \in E_0^+ \cH$ corresponds a weak solution $f_t = \e^{-t[\D \B_0]}h^+$, $t\geq 0$, of the first-order system for $\B_0$. It has additional regularity $f \in \C([0,\infty); E_0^+ \cH) \cap \C^\infty((0,\infty); E_0^+ \cH)$, converges to $h^+$ and $0$ in the $\L^2(\Omega)$-sense as $t \to 0$ and $t \to \infty$, respectively, and there are
equivalences
\begin{align*}
 \sup_{t \geq 0} \|f_t\|_{\L^2(\Omega)^n} \simeq \|h^+\|_{\L^2(\Omega)^n} \simeq \|\partial_t f\|_\cY.
\end{align*}
\end{proposition}

\begin{proof}
The restriction of $\{\e^{-t [\D \B_0]}\}_{t \geq 0}$ to $E_0^+ \cH$ is the bounded holomorphic semigroup generated by $-\D \B_0|_{E_0^+ \cH}$, see \eqref{Hinfty 3} in Section~\ref{Sec: Quadratic estimates for DB0}. Hence, $\partial_t f_t + \D \B_0 f_t = 0$ on $\IR^+$ in the classical sense and in particular, $f$ is a weak solution in the sense of Definition~\ref{Def: Weak solutions to FO}. The additional regularity and limits follow from abstract semigroup theory, see, e.g., \cite[Sec.~3.4]{Haase}. The first of the equivalences is by boundedness of the semigroup and the second one is by quadratic estimates for $\D \B_0$ with regularly decaying holomorphic function $[z] \e^{-[z]}$.
\end{proof}

\begin{remark}
\label{Rem: Semigroup solutions without an L2 trace}
If $u$ is a weak solution to the second-order system that satisfies an $\L^2$-Dirichlet condition on the cylinder base, then we expect $f = \nablaA  u$ to be a weak solution to the first-order system without a trace at $t=0$ in the $\L^2$-sense. By the same argument as above, such solutions can be constructed as $f_t = [\D \B_0]^\alpha \e^{-t [\D \B_0]} h^+$, where $\alpha>0$ and $h^+ \in E_0^+ \cH$.
\end{remark}

Below, we present a careful analysis of these semigroup solutions to the first-order system and in particular prove that they are contained in the natural solution space $\cX$ for the Neumann and regularity problems.
\subsection{Off-diagonal decay}
\label{Subsec: Off-diagonal decay}

As a technical tool to be utilized in the following, we establish $\L^p$ off-diagonal decay of arbitrary polynomial order for the resolvents of $\D \B_0$ if $|p-2|$ is sufficiently small. The case $p=2$ is a standard result for perturbed Dirac-type operators, once the subsequent localization and commutator properties for $\D$ have been verified \cite[Prop.~5.1]{AAM-EstimateRelatedToKato}.

\begin{lemma}
\label{Lem: Localization property for D}
Let $\varphi \in \IC_c^\infty(\IR^d; \IR)$ and let $M_\varphi$ be the associated multiplication operator on $\L^2(\Omega)^n$. Then $M_\varphi \dom(\D) \subseteq \dom(\D)$ and the commutator $[\D, M_\varphi]$ acts on $\dom(\D)$ as a multiplication operator induced by some $c_\varphi \in \L^\infty(\Omega; \Lop(\IC^n))$ with pointwise bound $\abs{c(x)} \lesssim \abs{\nabla \varphi(x)}$ on $\Omega$.
\end{lemma}

\begin{proof}
Since $\varphi \V \subseteq \V$ by definition of $\V$, the claim for $-\gV$ in place of $\D$ is immediate by the product rule. By duality, the same holds true for $\dV = (-\gV)^*$ and thus for $\D$ itself.
\end{proof}

\begin{proposition}[$\L^2$ off-diagonal estimates]
\label{Prop: L2 off diagonals for DB}
Let $T =\D \B_0$ or $T = \B_0 \D$. Then for every $l \in \IN_0$ there exists a constant $C_l>0$ such that
\begin{align*}
 \|\ind_F (1+\i s T)^{-1} \ind_E u\|_{\L^2(\Omega)^n} \leq C_l \bigg(1+\frac{\dist(E,F)}{s} \bigg)^{-l} \|\ind_E u\|_{\L^2(\Omega)^n}
\end{align*}
holds for all $u \in \L^2(\Omega)^n$, all $s >0$, and all Borel sets $E, F \subseteq \Omega$.
\end{proposition}

We will appeal to {\u{S}}ne{\u{\ii}}berg's stability result on complex interpolation scales in order to extend the off-diagonal bounds to the $\L^p$-scale nearby $\L^2$. In the case $m=1$ the following complex interpolation identities for the scale of Banach spaces $\{\W_\Dir^{1,p}(\Omega)^m\}_{1<p<\infty}$ have been established in \cite[Cor.~8.3]{ABHR}. As usual, interchangeability of interpolation functors and Cartesian products allows to extend the claim to $m>1$. Scale invariance as stated below can basically be obtained by their method as well. Complete details of this somewhat tedious argument have been worked out in \cite[Sec.~2.5]{EigeneDiss}.

\begin{proposition}
\label{Prop: Interpolation for W1pD}
Let $0< \theta < 1$, let $1< p_0, p_1 < \infty$, and $\frac{1}{p_\theta} = \frac{1-\theta}{p_0} + \frac{\theta}{p_1}$. Then the complex interpolation identity 
\begin{align*}
 \big[\W_\Dir^{1,p_0}(\Omega)^m, \W_\Dir^{1,p_1}(\Omega)^m \big]_\theta = \W_\Dir^{1,p_\theta}(\Omega)^m
\end{align*}
holds up to equivalent norms. Moreover, implicit constants can be chosen uniformly when replacing $(\Omega, \Dir)$ by $(\frac{1}{s}\Omega, \frac{1}{s} \Dir)$ for any $0<s \leq 1$.
\end{proposition}

In the following $p' = \frac{p}{p-1}$ denotes the H\"older conjugate of $p \in [1,\infty]$ and we write $\mathcal{Z}^*$ for the space of bounded conjugate-linear functionals on a Banach space $\mathcal{Z}$.

\begin{corollary}
\label{Cor: Interpolation for Xsp}
Let $0 < s \leq 1$. For $1 < p < \infty$ let $\X_s^p(\Omega)$ denote the Banach space $\W_\Dir^{1,p}(\Omega)^m$ with equivalent norm $u \mapsto (\int_\Omega \abs{u}^p + \abs{s \nabla u}^p \; \d x)^{1/p}$. Let $\theta$, $p_0$, $p_1$, and $p_\theta$ be as in Proposition~\ref{Prop: Interpolation for W1pD}. Then
\begin{align*}
 \big[\X_s^{p_0}(\Omega), \X_s^{p_1}(\Omega) \big]_\theta = \X_s^{p_\theta}(\Omega)
\quad \text{and} \quad
 \big[\X_s^{p_0}(\Omega)^*, \X_s^{p_1}(\Omega)^* \big]_\theta = \X_s^{p_\theta}(\Omega)^* 
\end{align*}
up to equivalent norms and the equivalence constants can be chosen independently of $s$.
\end{corollary}

\begin{proof}
There is a canonical isometric isomorphism $T: \X_s^p(\Omega) \to \W_{s^{-1}D}^{1,p}(s^{-1}\Omega)^m$ given by $(Tu)(x) = s^{d/p} u(sx)$. Hence, the first identity follows from Proposition~\ref{Prop: Interpolation for W1pD} and exactness of complex interpolation \cite[Thm.~4.1.2]{Bergh-Loefstroem}. Since the spaces $\X_s^p(\Omega)$ share the common dense set $\C_\Dir^\infty(\Omega)^m$ and are reflexive as closed subspaces of the reflexive spaces $\W^{1,p}(\Omega)^m$, the second identity follows by the duality principle for complex interpolation \cite[Cor.~4.5.2]{Bergh-Loefstroem}
\end{proof}

Next, we extend resolvents of $\D \B_0$ to bounded operators on $\L^p$ for $p$ in a neighborhood of $2$.

\begin{proposition}
\label{Prop: Lp boundedness}
There exists $\eps>0$ such that if $|p-2| < \eps$, then for every $s_0 > 0$ the resolvents $\{(1+\i s \D \B_0)^{-1}\}_{0<s \leq s_0}$ extend/restrict to a uniformly bounded family of bounded operators on $\L^p(\Omega)^n$.
\end{proposition}

\begin{proof}
We can assume $s_0 = 1$ since $s_0 \B_0$ is an operator in the same class as $\B_0$. Given $0<s\leq1$ and $f \in \L^2(\Omega)^n \cap \L^p(\Omega)^n$, we define $g \in  \L^2(\Omega)^n$ by $g:= \overline{A}^{-1}(1+\i s \D \B_0)^{-1}f$. On recalling $\B_0 = \underline{A} \overline{A}^{-1}$, we obtain $f = \overline{A}g + \i s \D \underline{A} g$, that is,
\begin{align*}
 \begin{bmatrix} f_\pe \\ f_\pa \end{bmatrix}
=\begin{bmatrix} (A g)_\pe \\ g_\pa \end{bmatrix}
+\i s \begin{bmatrix} (-\gV)^* (Ag)_\pa \\ -\gV g_\pe \end{bmatrix}.
\end{align*}
We use the second equation to eliminate $g_\pe$ in the first one and separate the terms containing $g_\pe$ from those containing $f$. This reveals $g_\pe \in \V$ as a solution of the divergence-form problem
\begin{align}
\label{Eq: gpe is solution}
 \int_\Omega A \begin{bmatrix} g_\pe \\ \i s \nablax  g_\pe \end{bmatrix} \cdot \cl{\begin{bmatrix} v \\ \i s \nablax  v \end{bmatrix}} \; \d x
=\int_\Omega \left( \begin{bmatrix} f_\pe \\ 0 \end{bmatrix} - A \begin{bmatrix} 0 \\ f_\pa \end{bmatrix} \right) \cdot \cl{\begin{bmatrix} v \\ \i s \nablax  v \end{bmatrix}} \; \d x \qquad (v \in \V).
\end{align}
Due to their intrinsic scaling with respect to $s$, the natural framework to study such problems in $\L^p$ are the spaces $\X_s^p(\Omega)$. We write the right-hand of \eqref{Eq: gpe is solution} as $T(g_\pe)(v)$ for a bounded operator $T: \X_s^2(\Omega) \to \X_s^{2}(\Omega)^*$. Then 
\begin{align*}
 \|T u\|_{\X_s^2(\Omega)^*} \geq \lambda \|u\|_{\X_s^2(\Omega)} \quad (u \in \X_s^2(\Omega))
\end{align*}
and
\begin{align*}
 \|T\|_{\X_s^p(\Omega) \to \X_s^{p'}(\Omega)^*} \leq \|A\|_\infty \quad (1<p<\infty)
\end{align*}
Our main point is that these bounds do not depend on $s$. Moreover, $T: \X_s^2(\Omega) \to \X_s^{2}(\Omega)^*$ is an isomorphism by the very Lax-Milgram lemma.

Now, fix $1 < p_- < 2 < p_+ < \infty $. If $p \in (p_-, p_+)$, then Corollary~\ref{Cor: Interpolation for Xsp} allows to replace $\X_p^s(\Omega)$-norms by the norms of the corresponding interpolation space between $\X_{p_-}^s(\Omega)$ and $\X_{p_+}^s(\Omega)$, each time collecting a constant that depends on the respective value of $p$ but not on $s$. Hence we may apply  \u{S}ne{\u{\ii}}berg's stability theorem \cite{Sneiberg-Original} in its quantitative version as stated, e.g., in \cite[Thm.~1.3.25]{EigeneDiss} in order to obtain $\eps > 0$ such that for $|p-2| < \eps$ the operator $T: \X_s^p(\Omega) \to \X_s^{p'}(\Omega)^*$ is an isomorphism with lower bound
\begin{align*}
 \|T u\|_{\X_s^{p'}(\Omega)^*} \geq c_p \frac{\lambda}{5} \|u\|_{\X_s^p(\Omega)} \qquad (u \in \X_s^p(\Omega)).
\end{align*}
Here, neither $\eps$ nor $c_p$ depend on $s$. Moreover, the inverses agree on the intersection of any two of these $\X_s^{p'}(\Omega)^*$-spaces \cite[Thm.~8.1]{KMM}. Now, \eqref{Eq: gpe is solution} implies $\|T g_\pe\|_{\X_s^{p'}(\Omega)^*} \lesssim \|f\|_p$. Hence, if $|p-2| < \eps$, then $\|g_\pe \|_{\X_s^p(\Omega)} \lesssim \|f\|_p$. Since $g= \overline{A}^{-1}(1+\i s \D \B_0)^{-1}f$, we find
\begin{align*}
 \|(1+\i s \D \B_0)^{-1}f\|_p
\lesssim \|g\|_p
= \Big(\|g_\pe\|_p^p + \|f_\pa +  \i s \gV g_\pe\|_p^p \Big)^{1/p}
\lesssim \|f\|_p
\end{align*}
with implicit constants independent of $s$.
\end{proof}

\begin{corollary}[$\L^p$ off-diagonal estimates]
\label{Cor: Lp off diagonals}
Let $|p-2|<\eps$, where $\eps > 0$ is as in Proposition~\ref{Prop: Lp boundedness}, and let $s_0 > 0$. For every $l \in \IN_0$ and there exists a constant $C_l>0$ such that
\begin{align*}
 \|\ind_F (1+\i s \D \B_0)^{-1} \ind_E u\|_{\L^p(\Omega)^n} \leq C_l \bigg(1+ \frac{\dist(E,F)}{s} \bigg)^{-l} \|\ind_E u\|_{\L^p(\Omega)^n}
\end{align*}
holds for all $u \in \L^2(\Omega)^n \cap \L^p(\Omega)^n$, all $0<s\leq s_0$, and all Borel sets $E, F \subseteq \Omega$.
\end{corollary}

\begin{proof}
The claim follows by complex interpolation of the assertions of Proposition~\ref{Prop: L2 off diagonals for DB} and \ref{Prop: Lp boundedness} using the Riesz-Thorin convexity theorem.
\end{proof}

\subsection{Reverse H\"older estimates}
\label{Subsec: Reverse Holder estimates}

As a second tool toward proving non-tangential estimates for semigroup solutions to the first-order system, we need weak reverse H\"older-type estimates for solutions of the second-order system. For a later purpose we directly prove them for general coefficients $A$ satisfying Assumption~\ref{Ass: Ellipticity for BVP}. For elliptic partial differential equations on the whole space or an upper half-space, the classical estimates are already found in \cite{Giaquinta-book}. In the case of mixed boundary value problems such estimates have more recently been studied in \cite{Brown-Ott} but -- to the best of our knowledge -- none of the existing results comprises our geometric setup beyond Lipschitz domains. 

Below, we denote by $\frac{1}{p_*} = \frac{1}{p} + \frac{1}{d}$ and $\frac{1}{p^*} = \frac{1}{p} - \frac{1}{d}$ the lower and upper Sobolev conjugate of $1 \leq p \leq \infty$, respectively. We agree on $p^* = \infty$ if $p \geq d$.

We need four preparatory lemmas. The first one is a variant of Caccioppoli's inequality and is proved exactly as the classical estimate in \cite{Giaquinta-book}, see also \cite[Lem.~6.3.14]{EigeneDiss}. The restriction on $z$ stems from the fact that $(u-z)$ multiplied by a cut-off function with support in $(t-2r, t+2r) \times B(x,2r)$ has to be admissible as a test function in Definition~\ref{Def: Weak solutions to ES}.

\begin{lemma}[Caccioppoli inequality]
\label{Lem: Cacciopoli}
Let $u$ be a weak solution to the second-order system for $A$ and let $t>0$, $x \in \Omega$. Let $r \in (0, \frac{t}{2})$ and let $z \in \IC^m$ be arbitrary if $B(x, 2 r) \cap \Dir = \emptyset$ and otherwise let $z = 0$. Then the estimate
\begin{align*}
 \int_{t-r}^{t+r} \int_{B(x,r) \cap \Omega} |r \nablatx  u|^2 \; \d y \; \d s \lesssim \int_{t-2r}^{t+2r} \int_{B(x,2r) \cap \Omega} \abs{u-z}^2 \; \d y \; \d s
\end{align*}
holds for an implicit constant depending on $A$ and $d$.
\end{lemma}

The second ingredient is the classical Poincar\'{e} inequality found in \cite[Lem.~7.12/16]{Gilbarg-Trudinger}.

\begin{lemma}
\label{Lem: Convex Poincare inequality}
Let $1 \leq p \leq q < p^* < \infty$ and put $\delta:= \frac{1}{p} - \frac{1}{q}$. Let $\Xi \subseteq \IR^l$, $l \geq 2$, be bounded, open, and convex, and let $S$ be a Borel subset of $\Xi$ with $\abs{S} > 0$. Then
\begin{align*}
 \|u - u_S\|_{\L^q(\Xi)} \leq \frac{(1-\delta)^{1-\delta}}{d (1/d - \delta)^{1-\delta}} \cdot \frac{(\diam \Omega)^d \abs{B(0,1)}^{1-1/d} \abs{\Omega}^{1/d}}{\abs{S}} \|\nabla u\|_{\L^p(\Xi)^l}
\end{align*}
for all $u \in \W^{1,p}(\Xi)$, where $u_S:= \barint_S u \; \d x$ denotes the \emph{mean value}\index{mean value, of a function} of $u$ on $S$.
\end{lemma}

We also require a Poincar\'{e} inequality on the Sobolev spaces with partially vanishing trace as it can be deduced from \cite[Cor.~4.5.3]{Ziemer}, see also \cite[Cor.~2.3.3]{EigeneDiss} for a self-contained proof. Here we write
\begin{align*}
 \cH_{l-1}^\infty(E):= \inf \Big\{ \sum_{j = 1}^\infty r_j^{l-1}; \, x_j \in \IR^l, r_j > 0, E \subseteq \bigcup_{j=1}^\infty B(x_j, r_j) \Big\}
\end{align*}
for the $(l-1)$-dimensional \emph{Hausdorff content} of a set $E \subseteq \IR^l$.

\begin{lemma}
\label{Lem: Poincare with Hausdorff content}
Let $\Xi \subseteq \IR^l$, $l \geq 2$, be a bounded Lipschitz domain, let $1 < p < l$, and $p \leq q \leq p^*$. There exists a constant $C>0$ such that for all compact sets $\mathscr{E} \subseteq \Xi$ and every $u \in \W_{\mathscr{E}}^{1,p}(\Xi)$ it holds
\begin{align*}
 \|u\|_{\L^q(\Xi)} \leq \frac{C}{\cH_{l-1}^\infty(\mathscr{E})^{1/p}} \|\nabla u\|_{\L^p(\Xi)^l}.
\end{align*}
\end{lemma}

As our fourth and final ingredient we rephrase regularity of weak solutions to the second-order system -- which was defined somewhat from the perspective of evolution equations by separating the variables $t \in \IR$ and $x \in \IR^d$ -- using a function space on $\IR^{1+d}$. This amounts to finding the pre-image of $\L^2(\IR; \V) \cap \W^{1,2}(\IR; \L^2(\Omega)^m)$ under the identification
\begin{align*}
 \L^2(\IR \times \Omega) \cong \L^2(\IR^+; \L^2(\Omega)) \qquad \text{via} \qquad u \mapsto u_\otimes, \quad u_\otimes(t) = u(t, \cdot).
\end{align*}

\begin{lemma}
\label{Lem: WD1p on cylinder}
Let $1<p<\infty$. The map $u \mapsto u_\otimes$ extends from $\C_\Dir^\infty(\IR \times \Omega)$ by density to an isometric isomorphism
\begin{align*}
 \W_{\IR \times \Dir}^{1,p}(\IR \times \Omega) \cong \W^{1,p}(\IR; \L^p(\Omega)) \cap \L^p(\IR; \W_\Dir^{1,p}(\Omega)).
\end{align*}
\end{lemma}

\begin{proof}
We omit the dependence of vector-valued spaces on $\IR$ and write for example $\W^{1,p}(\L^p(\Omega))$. By Fubini's theorem
\begin{align*}
 (u \mapsto u_\otimes): \W_{\IR \times \Dir}^{1,p}(\IR \times \Omega) \to \W^{1,p}(\L^p(\Omega)) \cap \L^p(\W_\Dir^{1,p}(\Omega))
\end{align*}
provides an isometry and it suffices to show that each $f \in \W^{1,p}(\L^p(\Omega)) \cap \L^p(\W_\Dir^{1,p}(\Omega))$ with bounded support in $\IR$ is contained in the range. To this end, let us recall from Section~\ref{Subsec: Geometry} that $\W_\Dir^{1,p}(\Omega)$ admits a Sobolev extension operator by which means we can construct an extension of $f$ in the space $\W^{1,p}(\L^p(\IR^d)) \cap \L^p(\W_\Dir^{1,p}(\IR^d))$. Fubini's theorem allows to identify this extension with a function in $\W^{1,p}(\IR^{1+d})$. Restricting to $\IR \times \Omega$, we can therefore represent $f = h_\otimes$, where $h \in \W^{1,p}(\IR \times \Omega)$ has bounded support in $\IR \times \cl{\Omega}$. So, if $\Dir$ is empty, then we are done. Otherwise we obtain from Hardy's inequality as stated in Proposition~\ref{Prop: Hardy and Poincare} the estimate
\begin{align*}
\iint_{\IR \times \Omega} \bigg| \frac{h(t,x)}{\dist_{\IR \times \Dir}(t,x)}\bigg|^p \; \d x \; \d t
= \int_{-\infty}^\infty \int_\Omega \bigg| \frac{f_t(x)}{\dist_\Dir(x)}\bigg|^p \; \d x \; \d t
\lesssim \int_{-\infty}^\infty \int_\Omega \abs{\nablax f_t(x)}^p \; \d x \; \d t < \infty.
\end{align*}
We conclude $h \in \W_{\IR \times \Dir}^{1,p}(\IR \times \Omega)$ by the converse of Hardy's inequality also stated in Proposition~\ref{Prop: Hardy and Poincare}, applied on a suitably large cylinder $(-T,T) \times \Omega$ outside of which $h$ vanishes.
\end{proof}

\begin{remark}
\label{Rem: WD1p on cylinder}
If $u$ is a weak solution to the second-order system for $A$, then Lemma~\ref{Lem: WD1p on cylinder} applies to $\eta u$ for all $\eta \in \C_c^\infty(\IR^+; \IR)$. This shows $u \in \W_{(a,b) \times \Dir}^{1,2}(I \times \Omega)^m$ for all $0<a<b<\infty$.
\end{remark}

Our central result in this section reads as follows.

\begin{theorem}[Reverse H\"older inequality]
\label{Thm: Reverse Holder inequality}
Let $u$ be a weak solution to the second-order system for $A$ and let $2_* < p < 2$. Then for all $t>0$, all $x \in \Omega$, and all $r \in (0, \frac{t}{2})$,
\begin{align}
\label{Eq: Reverse Holder in Theorem}
 \bigg(\barint_{t-r}^{t+r} \barint_{B(x,r) \cap \Omega} \abs{\nablatx  u}^2 \; \d y \; \d s\bigg)^{1/2} \lesssim \bigg(\barint_{t-2r}^{t+2r} \barint_{B(x,2r) \cap \Omega} \abs{\nablatx  u}^p \; \d y \; \d s\bigg)^{1/p}
\end{align}
with an implicit constant depending on $p$, $A$, and the geometric parameters.
\end{theorem}

\begin{proof}
We claim that it suffices to prove that there exist $c > 0$ and $C>1$ such that
\begin{align}
\label{Eq: Reverse Holder}
  \bigg(\barint_{t-r}^{t+r} \barint_{B(x,r) \cap \Omega} \abs{\nablatx  u}^2 \; \d y \; \d s\bigg)^{1/2} \lesssim \bigg(\barint_{t-Cr}^{t+Cr} \barint_{B(x,Cr) \cap \Omega} \abs{\nablatx  u}^p \; \d y \; \d s\bigg)^{1/p}
\end{align}
for all $t >0 $, all $x \in \Omega$, and all radii $r$ that are either small in that $r < \min\{c, \frac{t}{2 C}\}$ or large in that $\diam(\Omega)<r<\frac{t}{2C}$. In fact, by an easy covering argument our estimate for small radii implies the one claimed in the theorem for all cylinders with $r\leq 2C \diam(\Omega)$ and our estimate for large radii implies the one in the theorem for all cylinders with $r>2C \diam(\Omega)$.

\subsubsection*{Step 1: Strategy for small cylinders}

We fix $t$, $x$ and define for $0<r<t$,
\begin{align*}
 V_r:= (t-r, t+r) \times B(x, r), \qquad
 W_r:= (t-r, t+r) \times (B(x, r) \cap \Omega).
\end{align*}
For the time being assume that we can extend $u$ across the boundary to a function in $\W^{1,2}(V_{4r})$ in such a way that there is control $\|\nablatx u\|_{\L^p(V_{4r})} \lesssim \|\nablatx u\|_{\L^p(W_{Cr})}$ for some $C \geq 4$ independent of $t$ and $x$. (Of course we restrict to $r<t/{2C}$ implicitly). Firstly, we apply Caccioppoli's estimate (C) with some admissible $z$ to obtain
\begin{align*}
\bigg( \bariint_{W_r} |r \nablatx u|^2 \; \d y \; \d s \bigg)^{1/2} 
&\stackrel{\mathrm{(C)}}{\lesssim} \bigg( \bariint_{V_{2r}} |u-z|^2 \; \d y \; \d s \bigg)^{1/2}.
\intertext{Secondly, we transform to a reference domain $\Xi = r^{-1}(-x + V_{4r})$ which neither depends on $t$ nor on $x$, apply a suitable Poincar\'{e} inequality (P) thereon, and transform back. This is necessary since constants in the Poincar\'{e} inequalities may depend on the underlying domain in an uncontrollable way and also may not scale appropriately with respect to $r$. The result is}
&\stackrel{\mathrm{(P)}}{\lesssim} \bigg( \bariint_{V_{4r}} |r\nablatx u|^p \; \d y \; \d s\bigg)^{1/p}\\
&\stackrel{\mathrm{(E)}}{\lesssim} \bigg( \bariint_{W_{Cr}} |r\nablatx u|^p \; \d y \; \d s \bigg)^{1/p},
\end{align*}
the second step following by the control on the extension (E).

\subsubsection*{Step 2: Details for small cylinders}

We let $U_1,\ldots,U_N$ be a covering of the compact set $\cl{\bd \Omega \setminus \Dir}$ by open sets provided by the Lipschitz condition around $\cl{\bd \Omega \setminus \Dir}$ according to Assumption~\ref{Ass: General geometric assumption on Omega BVP}. We denote the corresponding bi-Lipschitz mappings by $\Phi_j$ and let $L \geq 1$ be the supremum of the Lipschitz constants of $\Phi_j^\pm$. Next, we fix $\kappa > 0$ such that $U_\Dir := \big \{x \in \IR^d; \, \dist(x, \Dir) < \kappa < \dist(x, \cl{\bd \Omega \setminus \Dir}  \big\}$ lets $\Omega$, $U_\Dir$, $U_1, \dots, U_N$ become an open cover of the compact set $\cl{\Omega}$. By $\rho > 0$ we denote a subordinated Lebesgue number, meaning that every ball in $\IR^d$ with radius less than $\rho$ and center in $\cl{\Omega}$ is entirely contained in one of the sets used for the covering. 

We shall prove \eqref{Eq: Reverse Holder} for $t >0 $, $x \in \Omega$, and $r < \min\{c, \frac{t}{2 C}\}$, where $c:= \frac{\varrho}{6}$ and $C := 4L^2$. By the defining property of the Lebesgue number it suffices to get the estimate sketched in Step~1 started in the following cases.
\begin{enumerate}
 \item[1.] Suppose $B(x,2r) \subseteq \Omega$. Then $W_{2r} = V_{2r}$ are subsets of $\IR^+ \times \Omega$ and the extension can be omitted. So, we may apply (C) with $z = \bariint_{W_{2r}} u$ and use Lemma~\ref{Lem: Convex Poincare inequality} with $\Xi = r^{-1}(-x + V_{2r})$ and $S= r^{-1}(-x + W_{2r})$ for the Poincar\'{e} estimate (P). This yields the required estimate even with $C=2$.
 \item[2.] Suppose $B(x,6r) \subseteq U_j$ for some $j$ and in addition that $B(x,2r)$ does not intersect $\Dir$. Utilizing the bi-Lipschitz coordinate charts, we may extend $u$ to $V_{2r}$ by even reflection. Since the changes of coordinates increase distances by a factor of at most $L$, we have control on the extension even with $C=L^2$. Now, we can complete the proof as in the first case.
 \item[3.] Completing the second case, we assume now that already $B(x,2r)$ intersect $\Dir$. This forces $z=0$ in (C). The Ahlfors-David condition implies the closely related thickness condition
 \begin{align*}
  \cH_{d-1}^\infty(\Dir \cap B(x,4r)) \simeq r^{d-1}
 \end{align*}
 see \cite[Lem.~4.5]{Hardy-Poincare}. Consequently, $V_{4r}$ contains a portion $\mathscr{E}$ of $\IR^+ \times \Dir$ with $d$-dimensional Hausdorff content comparable to $r^d$. Having extended $u$ by reflection as before, we can rely on Lemma~\ref{Lem: Poincare with Hausdorff content} with $\Xi = r^{-1}(-x+V_{4r})$ for the estimate (P). Note that $u$ is in fact contained in the appropriate function space thanks to Remark~\ref{Rem: WD1p on cylinder}. In this manner, the claim follows with $C = 4L^2$.
 \item[4.] Finally assume $B(x,6r) \subseteq U_\Dir$. In view of the first case we may additionally assume that $B(x,2r)$ is not entirely contained in $\Omega$ and hence contains a point of $\Dir$. Extending $u$ to $V_{4r}$ by zero, the exact same reasoning as in the previous case yields the claim even with $C=4$.
\end{enumerate}

\subsubsection*{Step 3: Proof for big cylinders}

Finally we prove \eqref{Eq: Reverse Holder} for $x \in \Omega$ and large radii $\diam(\Omega)<r<\frac{t}{2C}$. For this step it will be sufficient to set $C=3$. By assumption we have $B(x,r) \cap \Omega = \Omega$. We fix a smooth cut-off function $\eta$ with support in $(t-3r, t+ 3r)$, identically $1$ on $(t-2r, t+2r)$, and estimates $\|\eta\| + \|r \eta'\|_\infty \lesssim 1$ and introduce
\begin{align*}
 v:= \begin{cases}
    \eta u &\text{if $\Dir \neq \emptyset$,}\\
    u-\barint_{t-3r}^{t+3r} \barint_\Omega u &\text{if $\Dir = \emptyset$}.
 \end{cases}
\end{align*}
From Caccioppoli's inequality we can infer
\begin{align*}
 \bigg(\barint_{t-r}^{t+r} \barint_\Omega |r \nablatx u|^2 \; \d y \; \d s \bigg)^{1/2} 
 &\lesssim \bigg(\barint_{t-3r}^{t+3r} \barint_\Omega |v|^2 \; \d y \; \d s \bigg)^{1/2}. 
\end{align*}
Under an affine transformation in $t$-direction, the domain of integration on the right-hand side corresponds to $\widetilde{\Omega} := (-3,3) \times \Omega$. Let $v$ correspond to $\widetilde{v}$ under this transformation. If we declare $\widetilde{\Dir}:= \bd \widetilde{\Omega} \setminus ((-2,2) \times \Dir)$ to be the Dirichlet part of $\widetilde{\Omega}$, then $\widetilde{v} \in \W_{\widetilde{\Dir}}^{1,2}(\widetilde{\Omega})$. Moreover, this geometric setup satisfies Assumption~\ref{Ass: General geometric assumption on Omega BVP} in $\IR^{1+d}$. (Here we make essential use of the fact that bottom and top of $\widetilde{\Omega}$ belong to the Dirichlet part). Hence, we have at hand the Sobolev embedding $\W_{\widetilde{\Dir}}^{1,p}(\widetilde{\Omega}) \subseteq \L^2(\widetilde{\Omega})$, which under the aforementioned linear transformation corresponds to
\begin{align*}
 \bigg(\barint_{t-3r}^{t+3r} \barint_\Omega |v|^2 \; \d y \; \d s \bigg)^{1/2} \lesssim \bigg(\barint_{t-3r}^{t+3r} \barint_\Omega |v|^p+|r \partial_t v|^p + |\nablax v|^p \; \d y \; \d s \bigg)^{1/p}.
\end{align*}
Thus, so far we have proved
\begin{align}
\label{Eq1: Reverse Holder}
 \bigg(\barint_{t-r}^{t+r} \barint_\Omega |\nablatx u|^2 \; \d y \; \d s \bigg)^{1/2}
\lesssim \bigg(\barint_{t-3r}^{t+3r} \barint_\Omega |r^{-1} v|^p+|\partial_t v|^p + |r^{-1} \nablax v|^p \; \d y \; \d s \bigg)^{1/p},
\end{align}
where $r^{-1}$ can be bounded by the geometrical parameter $\diam(\Omega)^{-1}$ if convenient.
If $\Dir \neq \emptyset$, then the pointwise estimates for $\eta$ and Poincar\'{e}'s inequality from Proposition~\ref{Prop: Hardy and Poincare}\eqref{Hardy} applied to each function $u_t \in \W_{\Dir}^{1,p}(\Omega)^m$ allow to bound the right-hand side of \eqref{Eq1: Reverse Holder} further by
\begin{align*}
\bigg(\barint_{t-3r}^{t+3r} \barint_\Omega |u|^p+|\nablatx u|^p \; \d y \; \d s \bigg)^{1/p}
\lesssim \bigg(\barint_{t-3r}^{t+3r} \barint_\Omega |\nablatx u|^p \; \d y \; \d s \bigg)^{1/p}
\end{align*}
and the proof is complete. Similarly, if $\Dir = \emptyset$, then Poincar\'{e}'s inequality on the Lipschitz domain $(t-3r,t+3r) \times \Omega$ (\cite[Thm.~4.4.2]{Ziemer}) allows to bound the right-hand side of \eqref{Eq1: Reverse Holder} by
\begin{align*}
 \bigg(\barint_{t-3r}^{t+3r} \barint_\Omega |u-z|^p+|\nablatx u|^p \; \d y \; \d s \bigg)^{1/p}
\lesssim \bigg(\barint_{t-3r}^{t+3r} \barint_\Omega |\nablatx u|^p \; \d y \; \d s \bigg)^{1/p}.
\end{align*}
Implicitly, we have used an affine change of variables to obtain independence of the constants on $t$ and $r$.
\end{proof}

Let us consider $X := \IR \times \Omega$ as a metric space with distance $d((t,x),(s,y)) := \max\{|t-s|, |x-y|\}$. Since $\Omega$ is $d$-Ahlfors regular, the restricted Lebesgue measure is a doubling measure $\mu$ on $X$ in the usual sense \cite{Bjorn-Bjorn}. If we let $Y:=\IR^+ \times \Omega \subseteq X$, then \eqref{Eq: Reverse Holder in Theorem} simply amounts to the estimate
\begin{align*}
 \bigg(\barint_B |\nablatx u|^2 \; \d \mu \bigg)^{1/2} \lesssim \bigg(\barint_{2B} |\nablatx u|^p \; \d \mu \bigg)^{1/p}
\end{align*}
for all balls $B$ in $X$ such that $2B \subseteq Y$. In this context, the self-improving character of reverse H\"older estimates (often referred to as Gehring's lemma) allows to increase the left-hand integrability index to some $q>2$. For a poof see either \cite[Thm.~3.3]{Zatorska-Goldstein} or the textbook \cite[Thm.~3.22]{Bjorn-Bjorn}. The latter reference gives a very transparent proof for $Y=X$ that literally applies in the general case: In fact, in order to achieve the improved estimate involving $q$ on a ball $B$, the argument makes use of the reverse H\"older estimate only on sub-balls of $B$. Thus, we obtain

\begin{corollary}
\label{Cor: Self-Improvement Reverse Holder}
Let $u$ be a weak solution to the second-order system for $A$ and let $2_* < p < 2$. Then there exists $2<q<\infty$ such that for all $t>0$, all $x \in \Omega$, and all $r \in (0, \frac{t}{2})$,
\begin{align*}
 \bigg(\barint_{t-r}^{t+r} \barint_{B(x,r) \cap \Omega} \abs{\nablatx  u}^q \; \d y \; \d s\bigg)^{1/q} \lesssim \bigg(\barint_{t-2r}^{t+2r} \barint_{B(x,2r) \cap \Omega} \abs{\nablatx  u}^p \; \d y \; \d s\bigg)^{1/p}.
\end{align*}
The implicit constant as well as $q$ depend only on $p$, $A$, and the geometric parameters.
\end{corollary}

In view of Proposition~\ref{Prop: f are conormals of u} we can formulate a similar result for weak solutions to the first-order equation. In this context it will be convenient to work with the Whitney regions, to which we can pass from the cylinders by a straightforward covering argument.

\begin{corollary}
\label{Cor: Reverse Holder for semigroup solutions}
Let $f$ be a weak solution to the first-order system for $\B$ and let $2_* < p < 2$. Then there exists $2<q<\infty$ such that for all $t>0$ and all $x \in \Omega$,
\begin{align*}
 \bigg(\bariint_{W(t,x)} |f|^2 \; \d y \; \d s \bigg)^{1/2} 
 \leq  \bigg(\bariint_{W(t,x)} |f|^q \; \d y \; \d s \bigg)^{1/q} 
 &\lesssim \bigg( \bariint_{2 W(t,x)} |f|^p \; \d y \; \d s \bigg)^{1/p} \\
 &\leq \bigg( \bariint_{2 W(t,x)} |f|^2 \; \d y \; \d s \bigg)^{1/2}.
\end{align*}
Here, $2W(t,x)$ is an enlarged Whitney region obtain from $W(t,x)$ upon replacing $c_0$ and $c_1$ by $2c_0$ and $2c_1$, respectively. The implicit constant as well as $q$ depend only on $p$, $A$, and the geometric parameters. In particular, this estimate applies to $f(t,x) = \e^{-t [\D \B_0]}h^+(x)$, where $h^+ \in E_0^+ \cH$.
\end{corollary}

For a later use we also record a side result of the proof of Theorem~\ref{Thm: Reverse Holder inequality}.

\begin{corollary}[Poincar\'{e} inequality on Whitney balls]
\label{Cor: Poincare-Whitney estimate}
Let $u$ be a weak solution to the second-order system for $A$. Then for all $0<t<1$ and all $x \in \Omega$,
\begin{align*}
 \bariint_{W(t,x)} \abs{u}^2 \; \d y \; \d s \lesssim \bariint_{2W(t,x)} \abs{t \nablatx u}^2 \; \d y \; \d s + \bigg(\bariint_{2W(t,x)} \abs{u} \; \d y \; \d s \bigg)^2
\end{align*}
with an implicit constant depending only on $p$, $A$, and the geometric parameters.
\end{corollary}

\begin{proof}
Recall the following Poincar\'{e} inequality from Step~2 of the Theorem~\ref{Thm: Reverse Holder inequality}: If $t >0 $, $x \in \Omega$, and $r < \min\{c, \frac{t}{2 C}\}$, where $c, C > 0$ are geometrical constants, then
\begin{align*}
 \bigg(\barint_{t-2r}^{t+2r} \barint_{B(x,2r) \cap \Omega} |u-z|^2 \; \d y \; \d s \bigg)^{1/2}
 \lesssim \bigg(\barint_{t-Cr}^{t+Cr} \barint_{B(x,Cr) \cap \Omega} |r \nablatx u|^p  \; \d y \; \d s \bigg)^{1/p}.
\end{align*}
Here, $2_*<p<2$ and either $z = 0$ or $z = \barint_{t-2r}^{t+2r} \barint_{B(x,2r) \cap \Omega} |u|$. In any case,
\begin{align*}
  \bigg(\barint_{t-2r}^{t+2r} \barint_{B(x,2r) \cap \Omega} |u|^2 \; \d y \; \d s \bigg)^{1/2}
 &\lesssim \bigg(\barint_{t-Cr}^{t+Cr} \barint_{B(x,Cr) \cap \Omega} |r \nablatx u|^2 \; \d y \; \d s \bigg)^{1/2} \\
 &\quad+ \barint_{t-2r}^{t+2r} \barint_{B(x,2r) \cap \Omega} |u| \; \d y \; \d s
\end{align*}
and the required estimate follow from covering the Whitney boxes by suitable cylinders of comparable size.
\end{proof}

\subsection{Non-tangential estimates and Whitney average convergence}
\label{Subsec: Non-tangential estimates and Whitney average convergence}

We are in a position to prove that the semigroup solutions constructed in Proposition~\ref{Prop: Semigroup solution etDB} are contained in the solution space $\cX$ for the Neumann and regularity problems.

\begin{theorem}
\label{Thm: NT bound for semigroup solutions}
If $h^+ \in E_0^+ \cH$ and $f_t = \e^{- t [\D \B_0]} h^+$, $t>0$, then there is comparability
\begin{align*}
 \|\NT(f)\|_{\L^2(\Omega)} \simeq \|h^+\|_{\L^2(\Omega)^n}.
\end{align*}
\end{theorem}

\begin{proof}
The lower bound follows on letting $t \to 0$ in the estimate $\|\NT(f)\|_2^2 \gtrsim \frac{1}{t} \int_t^{2t} \|f_s\|_2^2 \; \d s$ provided by Lemma~\ref{Lem: X inside Y*}. 

For the upper estimate we fix $p<2$ sufficiently large in order to have at hand both Corollary~\ref{Cor: Reverse Holder for semigroup solutions} and Corollary~\ref{Cor: Lp off diagonals}. We abbreviate $R_s^{\D \B_0} = (1+\i s \D \B_0)^{-1}$ and we let $\zeta = \e^{-[z]} - (1 + \i z)^{-1}$, so that $\zeta(t \D \B_0)h^+ = f_t - (1+ \i t \D \B_0)^{-1}h^+$. Splitting the non-tangential maximal function at $t=c_0^{-1}$ and employing Corollary~\ref{Cor: Reverse Holder for semigroup solutions}, we obtain the pointwise bound
\begin{align*}
 \NT(f)(x) 
 &\leq \sup_{t \geq c_0^{-1}} \bigg(\bariint_{W(t,x)} \abs{f}^2 \bigg)^{1/2} + \sup_{0 < t < c_0^{-1}} \bigg(\bariint_{2W(t,x)} \abs{\zeta(s \D \B_0)h^+(y)}^p \; \d y \; \d s \bigg)^{1/p} \\
 &\quad + \sup_{0 < t < c_0^{-1}} \bigg(\bariint_{2W(t,x)} |R_s^{\D \B_0}h^+(y)|^p \; \d y \; \d s \bigg)^{1/p}.
\end{align*}
We shall estimate the three suprema separately in $\L^2(\Omega)$ by a multiple of $\|h^+\|_2$. 

(i): Since $\Omega$ is $d$-Ahlfors regular, there is a uniform lower bound for the measure of $B(x, c_1 t) \cap \Omega$, where $x \in \Omega$, $t \geq c_0^{-1}$. Using the uniform bound for the $[\D \B_0]$-semigroup, we obtain that the first supremum is uniformly bounded on $\Omega$ by
\begin{align*}
  \sup_{t \geq c_0^{-1}} \bigg(\frac{1}{t} \int_{c_0^{-1} t}^{c_0 t} \int_\Omega \abs{f_s(y)}^2 \; \d y \; \d s \ \bigg)^{1/2}
= \sup_{t \geq c_0^{-1}} \bigg(\frac{1}{t} \int_{c_0^{-1} t}^{c_0 t} \|\e^{-s [\D \B_0]}h^+\|_2^2 \; \d s \bigg)^{1/2}
\lesssim \|h^+\|_2.
\end{align*}
Since $\Omega$ is bounded, the required $\L^2$-bound follows.

(ii): As for the second supremum, Jensen's inequality, Lemma~\ref{Lem: X inside Y*}, and quadratic estimates for $\D \B_0$ bound its $\L^2$-norm by
\begin{align*}
 \| \NT(\zeta(t \D \B_0)h^+)\|_2^2 \lesssim \int_0^\infty \|\zeta(t \D \B_0)h^+\|_2^2 \; \frac{\d t}{t} \simeq \|h^+\|_2^2.
\end{align*}
Note that here $\NT$ takes averages over enlarged regions $2W(t,x)$, which simply amounts to replacing the generic constants $c_0$ and $c_1$ by $2c_0$ and $2c_1$, respectively. 

(iii): For the third term, we perform a rough estimate as in the proof of Lemma~\ref{Lem: X inside Y*} to find
\begin{align*}
\bariint_{2W(t,x)} \abs{R_s^{\D \B_0}h^+(y)}^p \; \d y \; \d s
\lesssim &\int_{c_0^{-1}t}^{c_0 t} \int_{\Omega} \ind_{B(x, 2c_0 c_1 s)}(y) |R_s^{\D \B_0}h^+(y)|^p \; \d y \; \frac{\d s}{s^{1+d}}
\end{align*}
uniformly for all $0<t\leq c_0^{-1}$ and all $x \in \Omega$. Since $s \leq 1$ in the above domain of integration,
\begin{align}
\label{Eq1: NT bound for resolvents}
\sup_{0<t \leq c_0^{-1}} \bariint_{W(t,x)} |R_s^{\D \B_0}h^+(y)|^p \; \d y \; \d s
\lesssim \sup_{0 < s \leq 1} \frac{1}{s^d} \int_{\Omega} \ind_{B(x, c_0 c_1 r s)}(y) |R_s^{\D \B_0}h^+(y)|^p \; \d y.
\end{align}
For the moment we fix $0< s < 1$ and $x \in \Omega$. In order to control the integral on the right-hand side of \eqref{Eq1: NT bound for resolvents} we put $B(k):= B(x,  2^{k+1}c_0 c_1 s)$, $k \geq 0$, and split $\IR^d$ into annuli $C(0):= B(0)$ and $C(k):= B(k) \setminus B(k-1)$, $k \geq 1$. Corollary~\ref{Cor: Lp off diagonals} on $\L^p$ off-diagonal estimates yields 
\begin{align*}
 \|\ind_{B_0} R_s^{\D \B_0} h^+\|_{\L^p(\Omega)^n}
\leq \|\ind_{B(0)} h^+\|_{\L^p(\Omega)^n} + \sum_{k \geq 1} \big(1+ (2^{k} -2)c_0 c_1 \big)^{-l} \|\ind_{B(k)} h^+\|_{\L^p(\Omega)^n}
\end{align*}
for some natural number $l$ to be specified below. Denoting by $M$ the classical Hardy-Littlewood maximal operator on $\Lloc^1(\IR^d)$, we find 
\begin{align*}
\|\ind_{B(k)} h^+\|_{\L^p(\Omega)^n} \lesssim 2^{dk/p} s^{d/p} M(|\ind_\Omega h^+|^p)(x)^{1/p} \qquad (k \geq 0).
\end{align*}
Specializing to a fixed $l > d/p$, we discover
\begin{align*}
 \|\ind_{B(0)} R_s^{\D \B_0} h^+\|_{\L^p(\Omega)^n} 
\lesssim s^{d/p} M(|\ind_\Omega h^+|^p)(x)^{1/p}.
\end{align*}
This estimate inserted back on the right-hand side of \eqref{Eq1: NT bound for resolvents} leads us to
\begin{align*}
 \sup_{0<t < c_0^{-1}} \bariint_{2W(t,x)} |R_s^{\D \B_0} h^+(y)|^p \; \d y \; \d s
\lesssim M(|\ind_\Omega h^+|^p)(x) \qquad(x \in \Omega),
\end{align*}
from which the appropriate bound for the $\L^2$-norm follows on integrating the $\frac{2}{p}$-th power with respect to $x \in \Omega$, taking into account that the maximal operator is bounded on $\L^{2/p}(\IR^d)$. Note that it is only this final step of the proof where we make use of $p<2$.
\end{proof}

\begin{remark}
\label{Rem: NT bound on negative Hardy space}
The orientation of our half-infinite cylindrical domain does not matter and all results remain true for second and first-order systems on the lower half-infinite cylinder $\IR^- \times \Omega$ (with the obvious modifications of definitions). Since each $h^- \in E_0^- \cH$ corresponds to a classical/weak solution $f_t:= \e^{t [\D \B_0]}h^-$ of $\partial_t f_t + \D \B_0 f_t = 0$ for $t<0$, we similarly obtain $\|\NT(\e^{- t [\D \B_0]} h^-)\|_2 \simeq \|h^-\|_2$ for $h^- \in E_0^- \cH$. This implies
\begin{align*}
 \|\NT(\e^{- t [\D \B_0]} h)\|_2 \lesssim \|h\|_2 \qquad (h \in \cH)
\end{align*}
since $\cH$ is the topological sum of the two Hardy spaces.
\end{remark}

Besides $\L^2$-convergence of $\e^{-t [\D \B_0]} h^+$ toward the boundary data $h^+$ as $t \to 0$, we also obtain pointwise almost everywhere convergence of Whitney averages.

\begin{theorem}
\label{Thm: Almost everywhere convergence of Whitney averages}
Let $T = \D \B_0$ or $T= \B_0 \D$. For every $h \in \L^2(\Omega)^n$ there is convergence 
\begin{align*}
 \lim_{t \to 0} \bariint_{W(t,x)} |\e^{-s[T]}h(y) - h(x)|^2 \; \d y \; \d s = 0
\end{align*}
for almost every $x \in \Omega$ and in particular
\begin{align*}
 \lim_{t \to 0}  \bariint_{W(t,x)} \e^{-s[T]}h(y) \; \d y \; \d s = h(x).
\end{align*}
\end{theorem}

For the proof we need the following auxiliary estimate.

\begin{lemma}[Local coercivity estimate]
\label{Lem: Local coercivity estimate}
There exists a constant $C>0$ such that for every $x \in \Omega$, every $r > 0$ such that $B(x, 2r) \subseteq \Omega$, and every $u \in \dom(\D)$ it holds
\begin{align*}
 \int_{B(x,r)} \abs{\D u}^2 \leq C \bigg(\int_{B(x,2r)} \abs{\B_0 \D u}^2 + \frac{1}{r^2} \int_{B(x,2r)} \abs{u}^2 \bigg).
\end{align*}
\end{lemma}

\begin{proof}
Let $\eta$ be a smooth function with range in $[0,1]$, identically $1$ on $B(x,r)$, support in $B(x,2r)$, and $\abs{\nablax  \eta} \leq \frac{c_d}{r}$ for a constant $c_d$ depending only on $d$. Using the pointwise control of the commutator $[\eta, \D]$ provided by Lemma~\ref{Lem: Localization property for D}, we find
\begin{align*}
\int_{B(x,r)}\hspace{-2.2pt} \abs{\D u}^2 
\leq \int_{\Omega} \abs{\eta \D u}^2 
\lesssim \int_\Omega |\D(\eta u)|^2 + \frac{1}{r^2} \int_{B(x,2r)} \abs{u}^2
\end{align*}
with implicit constants independent of $r$. As $\B_0$ is accretive on $\Rg(\D)$ we have $\int_\Omega |\D(\eta u)|^2 \lesssim \int_\Omega |\B_0 \D(\eta u)|^2$. Now, the claim follows from boundedness of $\B_0$ and once again the pointwise commutator estimate.
\end{proof}

\begin{proof}[Proof of Theorem~\ref{Thm: Almost everywhere convergence of Whitney averages}] Throughout the proof we fix a representative for $h$. For resolvents of $T$ we use the shorthand notation $R_s^T = (1 + \i s T)^{-1}$, $s>0$. The argument is subdivided into four consecutive steps.

\subsubsection*{\normalfont \itshape Step 1: Preliminaries for the case $T = \B_0 \D$}

\noindent Given $x \in \Omega$, let $t_x:= \frac{1}{2} \dist(x, \bd \Omega)$ and let $\eta_x$ be a smooth $\IC^n$-valued function with compact support in $\Omega$ that takes the constant value $h(x)$ everywhere on $B(x, t_x)$. Clearly $\eta_x \in \dom(\D) = \dom (\B_0 \D)$, see Section~\ref{Sec: Equivalence of ES to a first order system}. If $t \leq t_x c_1^{-1}$, then
\begin{align*}
 \bariint_{W(t,x)} |\e^{-s[T]}h(y) - h(x)|^2 \; \d y \; \d s 
\end{align*}
is bounded from above by
\begin{align}
\label{Eq1: Almost everywhere convergence of Whitney averages}
2 \bariint_{W(t,x)} & |(\e^{-s[T]}-R_s^T) h(y)|^2
+ |R_s^T (h-\eta_x)(y)|^2
+ |R_s^T \eta_x(y) - \eta_x(y)|^2 \; \d y \; \d s.
\end{align}
We claim that each of these three terms vanishes as $t \to 0$ for almost every $x \in \Omega$. For the first term this follows from Lemma~\ref{Lem: Whitney averages vanish on Y*} and quadratic estimates for $\D \B_0$ with holomorphic function $\zeta = \e^{-[z]} - (1+\i z)^{-1}$. The other two terms require a closer inspection.

\subsubsection*{\normalfont \itshape Step 2: Second term estimate}

\noindent Throughout we may assume $t < 1$. Let $B(k) = B(x, 2^k c_1 t)$, $k \geq 0$, and split $\IR^d$ into annuli $C(0):= B(0)$ and $C(k):= B(k) \setminus B(k-1)$, $k \geq 1$. By $\L^2$ off-diagonal decay for the resolvents of $T$ we can infer an estimate
\begin{align*}
 \|\ind_{B(0)} R_s^T(h-\eta_{x})\|_{\L^2(\Omega)}^2
\leq \sum_{k \geq 0} 2^{-dk-k} \|\ind_{C(k)}(h-\eta_{x})\|_{\L^2(\Omega)}^2
\end{align*}
for $s$ in the range $[c_0^{-1} t, c_0 t]$ in which it is comparable to $t$. Integration with respect to $s$ leads to
\begin{align}
\label{Eq2: Almost everywhere convergence of Whitney averages} 
\bariint_{W(t,x)} |R_s^T(h-\eta_{x})(y)|^2 \; \d y \; \d s
\lesssim \sum_{k \geq 0} 2^{-k} \barint_{B(k)} |\ind_\Omega h(y)-\eta_{x}(y)|^2 \; \d y,
\end{align}
where implicitly we have used $d$-Ahlfors regularity of $\Omega$ on the left-hand side. We break the sum at $k_0$ characterized by $2^{-k_0 -1} \leq \sqrt{t} < 2^{-k_0}$ and use the Hardy-Littlewood maximal operator $M$ to control the integrals on the large balls with $k \geq k_0$. In this manner the right-hand side of \eqref{Eq2: Almost everywhere convergence of Whitney averages} is bounded by
\begin{align*}
\sum_{k = 0}^{k_0 -1} 2^{-k} \barint_{B(k)} |\ind_\Omega h(y)-\eta_{x}(y)|^2 \; \d y
+ \sum_{k = k_0}^\infty 2^{-k} M(|\ind_\Omega h-\eta_x|^2)(x).
\end{align*}
Balls occurring in the first sum are of radius less than $c_1 \sqrt{t}$. Hence, if even $c_1 \sqrt{t} < t_x$, then we have $\eta_{x}(y) = h(x)$ on each ball. For the second sum we utilize $|\eta_x| \leq |h(x)|$ and $\sum_{k=k_0}^\infty 2^{-k} \leq 4\sqrt{t}$. Altogether, an upper bound up to multiplicative constants for the right-hand side of \eqref{Eq2: Almost everywhere convergence of Whitney averages} is provided by
\begin{align}
\label{Eq3: Almost everywhere convergence of Whitney averages} 
\sup_{\tau \leq c_1 \sqrt{t}} \barint_{B(x, \tau)} \abs{\ind_\Omega h(y) - \ind_\Omega h(x)}^2 \; \d y
+ \sqrt{t} M(\abs{\ind_\Omega h}^2)(x) + \sqrt{t} \abs{h(x)}^2
\end{align}
if $t > 0$ is sufficiently small. In the limit $t \to 0$ the following hold: The first term in \eqref{Eq3: Almost everywhere convergence of Whitney averages} vanishes for every Lebesgue point of $\ind_\Omega h \in \L^2(\IR^d)^n$. The middle term vanishes provided $M(\abs{\ind_\Omega h}^2)(x)$ is finite, which by the weak-$(1,1)$ estimate for $M$ applies again for almost every $x \in \Omega$. Finally, the third term vanishes for every $x \in \Omega$. 

Note carefully that in the end the exceptional sets for $x$ did not depend on $t_x$ and $\eta_{x}$ although they had been involved in some of the calculations.

\subsubsection*{\normalfont \itshape Step 3: Third term estimate}

\noindent As $\eta_{x} \in \C_c^\infty(\Omega)^n$ is constant on the set $B(x, t_x)$, we can actually compute in the classical sense
\begin{align*}
 T \eta_{x}(y) = (\B_0 \D \eta_{x})(y) = \B_0(y) \begin{bmatrix} \divx (\eta_{x})_\pa(y) \\ -\nabla (\eta_{x})_\pe (y) \end{bmatrix} = 0 \qquad (y \in B(x, t_x)).
\end{align*}
We may assume $t \leq \frac{t_x}{2c_1}$ right from the start, so that we have
\begin{align*}
 \frac{1}{s} \dist\big(B(x, c_1t) \cap \Omega, \supp(T\eta_{x})\big) \geq \frac{t_x - c_1 t}{s} \geq \frac{t_x}{2c_0 t} \qquad (c_0^{-1}t \leq s \leq c_0 t).
\end{align*}
On writing $(R_s^T - 1)\eta_{x} = -\i s R_s^T T \eta_{x}$, the $\L^2$ off-diagonal estimates for $R_s^T$ yield
\begin{align*}
  \|\ind_{B(x, c_1 t) \cap \Omega}(R_s^T - 1)\eta_{x}\|_{\L^2(\Omega)}^2 \lesssim s^2 t^{2d-2} \|T \eta_x \|_{\L^2(\Omega)}^2 \qquad (c_0^{-1}t \leq s \leq c_0 t)
\end{align*}
with implicit constants depending also on $t_x$. Integration reveals
\begin{align*}
  \bariint_{W(t,x)} |R_s^T \eta_{x} - \eta_{x}|^2 \; \d y \; \d s
\lesssim t^{-1-d} t^{2d} \|T \eta_{x} \|_2^2,
\end{align*}
which in the limit $t\to0$ tends to $0$ for \emph{every} $x \in \Omega$ anyway.

\subsubsection*{\normalfont \itshape Step 4: The case $T = \D \B_0$}

\noindent Similar to the case $T = \B_0 \D$ we bound the average integrals over $W(t,x)$ by
\begin{align}
\label{Eq4: Almost everywhere convergence of Whitney averages} 
\bariint_{W(t,x)} |(\e^{-s[\D \B_0]}-R_s^{\D \B_0})h|^2
+ |R_s^{\D \B_0} h-h|^2
+ |h - h(x)|^2 \; \d y \; \d s.
\end{align}
Here, the integral over the first term vanishes in the limit $t \to 0$ for a.e.\ $x \in \Omega$ thanks to Lemma~\ref{Lem: Whitney averages vanish on Y*} and quadratic estimates for $\D \B_0$. The integral over the last term vanishes for every Lebesgue point $x$ of $\ind_\Omega h \in \L^2(\IR^d)^n$. It remains to consider the middle term in \eqref{Eq4: Almost everywhere convergence of Whitney averages}. Here, we cannot perform a localization argument as we did for $\B_0 \D$ since now $\D$ is applied after $\B_0$. However, by the intertwining relation $R_s^{\D \B_0} - 1 =  - \i s \D R_s^{\B_0 \D}\B_0$ it suffices to prove
\begin{align*}
\lim_{t \to 0} \bariint_{W(t,x)} |\i s \D R_s^{\B_0 \D} \B_0 h|^2 \; \d y \; \d s = 0 \qquad (\text{a.e.\ $x \in \Omega$}).
\end{align*}
To this end, let $x \in \Omega$. We abbreviate $\widehat{h} := \B_0 h$ and associate with it a function $\widehat{\eta}_{x} \in \C_c^\infty(\Omega)^n$ that takes the constant value $\widehat{h}(x)$ on $B(x, t_x)$. As in Step~4 we have $\D \widehat{\eta}_{x} = 0$ almost everywhere on $B(x,t_x)$. If $t<\frac{t_x}{2c_1}$, then Lemma~\ref{Lem: Local coercivity estimate} applies on the ball $B(x, c_1 t)$ with $u= \i s R_s^{\B_0 \D} \widehat{h} - \i s\widehat{\eta}_{x}$ as follows:
\begin{align*}
\int_{c_0^{-1}t}^{c_0 t} \int_{B(x, c_1 t)} |\i s \D R_s^{\B_0 \D}\widehat{\eta}_{x}(y)|^2 \; \d y \; \d s
\lesssim \int_{c_0^{-1}t}^{c_0 t} \int_{B(x, 2c_1 t)} &|\i s \B_0 \D R_s^{\B_0 \D} \widehat{h}(y)|^2 \\
&+ |R_s^{\B_0 \D} \widehat{h}(y) - \widehat{\eta}_{x}(y)|^2 \; \d y \; \d s.
\end{align*}
Hence, 
\begin{align*}
\bariint_{W(t,x)} |\i s \D R_s^{\B_0 \D}\widehat{\eta}_{x}(y)|^2 \; \d y \; \d s 
&\lesssim \bariint_{\widehat{W}(t,x)} |R_s^{\B_0 \D}(\widehat{h} - \widehat{\eta}_{x})(y)|^2 + |R_s^{\B_0 \D}\widehat{\eta}_{x}(y) - \widehat{\eta}_{x}(y)|^2 \; \d y \; \d s \\
&\quad + \barint_{B(x, 2c_1t)} |\ind_\Omega \widehat{h}(y) - \ind_\Omega \widehat{h}(x)|^2 \; \d y,
\end{align*}
where $\widehat{W}(t,x) := 2W(t,x)$. Upon replacing all `hatted' variables by their `unhatted' counterparts, almost everywhere convergence of the first two terms is precisely the statement of Steps~3 and 4, whereas the third term vanishes at every Lebesgue point of $\ind_\Omega \widehat{h}$. This completes the proof.
\end{proof}

\begin{remark}
\label{Rem: Almost everywhere convergence of Whitney averages}
The organization of the proof of Theorem~\ref{Thm: Almost everywhere convergence of Whitney averages} is inspired by \cite[Sec.~9.1]{Auscher-Stahlhut_APriori}. However, our setup bears the significant difficulty that $\D$ is not defined on constant functions on $\Omega$ -- at least when the Dirichlet part $\Dir$ is non-empty.
Surprisingly, the additional localization argument involving $\eta_x$ provides a slick way out.
\end{remark}
\section{A priori representation of solutions}
\label{Sec: Representation of solutions}

\noindent In this section we turn to systems with $t$-dependent coefficients and prove the \emph{a priori} estimates claimed in our main results. Throughout this section we fix $t$-dependent coefficients $A$ and $t$-independent coefficients $A_0$ satisfying Assumption~\ref{Ass: Ellipticity for BVP}. As before we let $\B$ and $\B_0$ correspond to $A$ and $A_0$, respectively. We study the first-order system for the $t$-dependent coefficients $\B$, which we formally rewrite as 
\begin{align*}
 \partial_t f_t + \D \B_0 f_t = \D \cE_t f_t, \quad \text{where $\cE_t = \B_0 - \B_t$}.
\end{align*}
For our results we will impose a Carleson condition on $\cE$. We remark that the modified Carleson norms of $A_0 - A$ and $\cE$ are comparable: Indeed, the identity
\begin{align*}
 \cE = \B_0 - \B = (\underline{A_0} - \underline{A})\overline{A_0}^{-1} + \underline{A} \overline{A_0}^{-1}(\overline{A} - \overline{A_0}) \overline{A}^{-1}
\end{align*}
along with $\overline{A_0}^{-1}, \underline{A}, \overline{A}^{-1} \in \L^\infty(\IR^+ \times \Omega; \Lop(\IC^n))$ shows that the norms of $A_0 - A$ dominate those of $\cE$. The reverse estimate follow since the transformation mapping $A \mapsto \B$ is an involution. Similarly, $\cX$ and $\cY$ norms of $\nablaA u$ and $\nablatx u$ are equivalent.

The starting point is a Duhamel-type formula for weak solutions to the first-order system. This uses the operators
\begin{align*}
 \widehat{E}_0^\pm := E_0^\pm \B_0^{-1} P_{\B_0 \cH},
\end{align*}
where $P_{\B_0 \cH}$ is the projection onto $\B_0 \cH$ along the splitting $\L^2(\Omega)^n = \Ke(\D) \oplus \B_0 \cH$, see Proposition~\ref{Prop: DB properties}, and $\B_0^{-1}$ is the inverse of $\B_0: \cH \to \B_0 \cH$, which exists by accretivity of $\B_0$ on $\cH$, see Lemma~\ref{Lem: B accretive on the range of D}.

\begin{lemma}
\label{Lem: Towards semigroup representation}
If $f$ is a weak solution to the first-order system for $\B$, then
\begin{align*}
 - \int_0^t \partial_s\eta^{+}(s) \e^{-(t-s) [\D \B_0]} E_0^+ f_s \; \d s &= \int_0^t \eta^+(s) \D \B_0 \e^{-(t-s) [\D \B_0]} \wE_0^+ \cE_s f_s \; \d s, \\
- \int_t^\infty \partial_s \eta^{-}(s) \e^{-(s-t) [\D \B_0]} E_0^{-} f_s \; \d s &= \int_t^\infty \eta^-(s) \D \B_0 \e^{-(s-t) [\D \B_0]} \wE_0^- \cE_s f_s \; \d s
\end{align*}
for all $t>0$ and all Lipschitz functions $\eta^{\pm}: \IR^+ \to \IR$ such that $\eta^+$ is compactly supported in $(0,t)$ and $\eta^-$ is compactly supported in $(t,\infty)$.
\end{lemma}

\begin{proof}
By density it suffices to consider smooth functions $\eta^{\pm}$ sharing the respective support properties. We concentrate on the identity on $(0,t)$ noting that the $(t,\infty)$-integral formula is established in exactly the same way. Throughout we abbreviate $\L^2$ inner products by $\scal{\cdot}{\cdot}$. Since both integrals in the identity in question are absolutely convergent in $\cl{\Rg(\D)} = \cH$, it suffices to prove
\begin{align}
\label{Eq1: Towards semigroup representation}
 -\int_0^t \bigscal{\partial_s\eta^{+}(s) \e^{-(t-s) [\D \B_0]} E_0^+ f_s}{h} \; \d s 
= \int_0^t \bigscal{\eta^+(s) \D \B_0 \e^{-(t-s) [\D \B_0]} \wE_0^+ \cE_s f_s}{h} \; \d s
\end{align}
for all $h \in \cH$. Since $f$ is a weak solution,
\begin{align}
\label{Eq2: Towards semigroup representation}
 -\int_0^\infty \scal{f_s}{\partial_s g_s} \; \d s + \int_0^\infty \scal{\B_0 f_s}{\D g_s} \; \d s = \int_0^\infty \scal{\cE_s f_s}{\D g_s} \; \d s 
\end{align}
holds for all test functions $g \in \C_c^\infty(\IR^+; \dom(\D))$. We shall show that for the special choice $g_s:= \eta^+(s) (\e^{-(t-s)[\D \B_0]} E_0^+)^*h$ equation \eqref{Eq2: Towards semigroup representation} transforms into \eqref{Eq1: Towards semigroup representation}. Here and throughout, adjoints are taken with respect to $\cH$ as ambient Hilbert space. This choice is admissible, since by stability of the functional calculus under restrictions and adjoints \cite[Sec.~2.6]{Haase} the map
\begin{align*}
 (0,\infty) \to \cH, \quad s\mapsto (\e^{-s[\D\B_0]} E_0^+)^*h = (\chi^+(z) \e^{-s [z]})(\D \B_0)^*h
\end{align*}
is an orbit of the holomorphic semigroup generated by $-(\D \B_0|_\cH)^*|_{E_0^+ \cH}$ on $E_0^+ \cH$ and as such, it is holomorphic with values in $\dom((\D \B_0|_\cH)^*)$. Recall from Corollary~\ref{Cor: Adjoint of injective part of DB} that the latter domain is continuously included in $\dom(\D)$ and that in fact $(\D \B_0|_\cH)^* = P \B_0^* \D|_\cH$ with $P$ the orthogonal projection in $\L^2(\Omega)^n$ onto $\cH$. So, for this choice of $g$ the left-hand side of \eqref{Eq2: Towards semigroup representation} becomes
\begin{align*}
& -\int_0^t \bigscal{f_s}{\partial_s \eta^+(s) (\e^{-(t-s)[\D \B_0]} E_0^+)^*h} \; \d s
- \int_0^t \bigscal{f_s}{\eta^+(s) P \B_0^* \D (\e^{-(t-s)[\D \B_0]} E_0^+)^*h} \; \d s \\
&+ \int_0^t \bigscal{\B_0 f_s}{\eta^+(s) \D  (\e^{-(t-s)[\D \B_0]} E_0^+)^*h} \; \d s.
\end{align*}
Note that if $u \in \cH$ and $v \in \dom(\D)$, then $\scal{\B_0 u}{\D v} = \scal{u}{P \B_0^* \D v}$. Since $f$ is $\cH$-valued, the last two terms above cancel and the result is the left-hand side of \eqref{Eq1: Towards semigroup representation}. The right-hand side of \eqref{Eq2: Towards semigroup representation} can be written as
\begin{align*}
 \int_0^t \scal{P_{\B_0 \cH} \cE_s f_s}{\eta^+(s) \D (\e^{-(t-s)[\D \B_0]} E_0^+)^*h} \; \d s
\end{align*}
since $1-P_{\B_0 \cH}$ projects onto $\Ke(\D) = \cH^\pe$. Similar as above we have for $u \in \L^2(\Omega)^n$ and $v \in \dom(\D)$ that $\scal{P_{\B_0 \cH} u}{\D v} = \scal{\B_0^{-1} P_{\B_0 \cH} u}{P \B_0^* \D v}$ so that altogether the right-hand side of \eqref{Eq2: Towards semigroup representation} equals
\begin{align*}
\int_0^t \scal{\B_0^{-1} P_{\B_0 \cH} \cE_s f_s}{\eta^+(s) (\D \B_0 \e^{-(t-s)[\D \B_0]} E_0^+)^*h} \; \d s,
\end{align*}
which by definition of $\wE_0^+$ coincides with the right-hand side of \eqref{Eq1: Towards semigroup representation}.
\end{proof}

Formally taking limits $\eta^+ \to \ind_{(0,t)}$ and $\eta^- \to \ind_{(t, \infty)}$, in which the derivatives approach certain differences of Dirac distributions, the Duhamel-type formulas in Lemma~\ref{Lem: Towards semigroup representation} become
\begin{align*}
 E_0^+ f_t - \e^{-t[\D \B_0]} E_0^+ f_0 &= \int_0^t \D \B_0 \e^{-(t-s)[\D \B_0]} \widehat{E}_0^+ \cE_s f_s \; \d s, \\
 0 - E_0^- f_t &= \int_t^\infty \D \B_0 \e^{-(s-t)[\D \B_0]} \widehat{E}_0^- \cE_s f_s \; \d s, 
\end{align*}
that is, since $f$ is $\cH$-valued, $f_t = \e^{-t[\D \B_0]} E_0^+ f_0 + S_A f_t$ with the singular integral operator 
\begin{align*}
 S_A f_t = \int_0^t \D \B_0 \e^{-(t-s)[\D \B_0]} \widehat{E}_0^+ \cE_s f_s \; \d s -  \int_t^\infty \D \B_0 \e^{-(s-t)[\D \B_0]} \widehat{E}_0^- \cE_s f_s \; \d s.
\end{align*}
A rigorous argument for this limiting process, as well as a rigorous definition of the maximal regularity operator $S_A$ has been established by Ros\'{e}n and the first author in \cite{AA-Inventiones} under the additional assumption that either $f \in \cX$ or $f \in \cY$. Note that \cite{AA-Inventiones} deals with elliptic systems on the upper half space but taking limits in Lemma~\ref{Lem: Towards semigroup representation} has been established in an abstract framework, consisting of a Hilbert space $\cH$, spaces $\cY:= \L^2(\IR^+, t \d t; \cH)$, $\cY^*:= \L^2(\IR^+, \frac{\d t}{t}; \cH)$, and $\cX$ with continuous embeddings
\begin{align*}
 \cY^* \subseteq \cX \subseteq \Lloc^2(\IR^+, \d t; \cH),
\end{align*}
and a semigroup generator $-[\D \B_0]$ on $\cH$ such that $h \mapsto \{\e^{-t [\D \B_0]}h\}_{t>0}$ is bounded from $\cH$ into $\cX$. As for our setup, the required embeddings have been established in Lemma~\ref{Lem: X inside Y*} and $\cH \to \cX$ boundedness of the semigroup is due to Theorem~\ref{Thm: NT bound for semigroup solutions} and the subsequent remark. This being said, we may freely use the results from \cite{AA-Inventiones} and we suggest to keep a copy of this article handy as we shall only outline the necessary changes for our setup. 
\subsection{The Neumann and regularity problems}
\label{Subsec: The Neumann and regularity problems}

We begin with the \emph{a priori} estimates for weak solutions with Neumann data $(\nablaA u)_\pe|_{t=0}$ or regularity data $(\nablaA u)_\pa|_{t=0} =\nablax u|_{t=0}$ and interior control $\nablaA u \in \cX$. In view of Proposition~\ref{Prop: f are conormals of u} and since all these are boundary conditions for the conormal gradient rather than the potential $u$ itself, it suffices to prove a priori estimates for weak solutions $f \in \cX$ to the first-order system. 

Before we can state and prove the main result, we need to rigorously define the maximal regularity operators $S_A$ on $\cX$ (and simultaneously do so on $\cY$ for a later use). This uses a family of pointwise approximations to the characteristic functions of $(0,t)$ and $(t,\infty)$ defined by
\begin{align*}
 \eta_\eps^\pm(t,s) = \eta^0(\pm \tfrac{t-s}{\eps}) \eta_\eps(t) \eta_\eps(s),
\end{align*}
where $\eta^0$ is the piecewise linear function with support $(1,\infty)$ equal to $1$ on $(2,\infty)$ and $\eta_\eps$ is the piecewise linear function with support $(\eps, \frac{1}{2 \eps})$ equal to $1$ on $(2 \eps, \frac{1}{\eps})$.

\begin{proposition}[{\cite[Prop.~7.1]{AA-Inventiones}}]
\label{Prop: Boundedness of SA}
Suppose $\|\cE\|_C < \infty$. For $\eps > 0$ the operators
\begin{align*}
 S_A^\eps f_t = \int_0^t \eta_\eps^+(t,s) \D \B_0 \e^{-(t-s)[\D \B_0]} \widehat{E}_0^+ \cE_s f_s \; \d s -  \int_t^\infty \eta_\eps^-(t,s) \D \B_0 \e^{-(s-t) [\D \B_0]} \widehat{E}_0^- \cE_s f_s \; \d s
\end{align*}
are bounded $\|S_A^\eps\|_{\cX \to \cX} \lesssim \|\cE\|_C$ and $\|S_A^\eps\|_{\cY \to \cY} \lesssim \|\cE\|_C$ uniformly in $\eps>0$. In $\cX$ there is a limit operator $S_A \in \Lop(\cX)$ such that
\begin{align*}
 \lim_{\eps \to 0} \|S_A^\eps f - S_A f\|_{\L^2(a,b; \L^2(\Omega)^n)} = 0 \qquad (f \in \cX, \, 0<a<b<\infty).
\end{align*}
In $\cY$ there is a limit operator $S_A \in \Lop(\cY)$ such that $S_A^\eps f \to S_A f$ in $\cY$ for every $f \in \cY$. The limit operator for both spaces is given by
\begin{align*}
 S_A f_t = \lim_{\eps \to 0} \bigg( \int_\eps^{t-\eps} \D \B_0 \e^{-(t-s)[\D \B_0]} \widehat{E}_0^+ \cE_s f_s \; \d s -  \int_{t+\eps}^{\eps^{-1}} \D \B_0 \e^{-(s-t)[\D \B_0]} \widehat{E}_0^- \cE_s f_s \; \d s \bigg)
\end{align*}
with convergence in $\L^2(a,b; \L^2(\Omega)^n)$ for any $0< a < b < \infty$.
\end{proposition}

\begin{theorem}
\label{Thm: Representation for X solutions}
Assume that $\|\cE\|_C < \infty$ and let $f \in \cX$. Then $f$ is a weak solution to the first-order system for $\B$ if and only if for some $h^+ \in E_0^+ \cH$, which then is unique, it satisfies
\begin{align}
\label{Eq: Representation for X solutions}
 f_t = \e^{-t [\D \B_0]} h^+ + S_A f_t \qquad (\text{a.e.\ $t>0$}).
\end{align}
In this case let $h^- := \int_0^\infty \D \B_0 \e^{-s [\D \B_0]} \widehat{E}_0^- \cE_s f_s \; \d s \in E_0^- \cH$. Then $f$ converges to $f_0:= h^+ + h^-$ in the sense of Whitney averages
\begin{align*}
 \lim_{t \to 0} \bariint_{W(t,x)} |f(s,y) - f_0(x)|^2 \; \d y \; \d s = 0 \qquad (\text{a.e.\ $x \in \Omega$})
\end{align*}
as well as in the square Dini sense
\begin{align*}
 \lim_{t \to 0} \barint_t^{2t} \|f_s - f_0\|_{\L^2(\Omega)^n}^2 \; \d s = 0. 
\end{align*}
Moreover, $f$ vanishes at spatial infinity in the square Dini sense
\begin{align*}
 \lim_{t \to \infty} \barint_t^{2t} \|f_s\|_{\L^2(\Omega)^n}^2 \; \d s = 0.
\end{align*}
Finally, there are estimates
\begin{align*}
 \|h^+\|_{\L^2(\Omega)^n} + \|h^-\|_{\L^2(\Omega)^n} \simeq \|f_0\|_{\L^2(\Omega)^n} \lesssim \|f\|_\cX
\end{align*}
and if $\|\cE\|_C$ is sufficiently small, then all three quantities above are comparable to $\|h^+\|_{\L^2(\Omega)^n}$.
\end{theorem}

\begin{proof}
Necessity of \eqref{Eq: Representation for X solutions} and the estimates are proved in parts (i) and (iv) of \cite[Thm.~8.2]{AA-Inventiones} by taking limits $\eps \to 0$ in the Duhamel-type formulas from Lemma~\ref{Lem: Towards semigroup representation} for $\eta^\pm := \eta_\eps^\pm$. Moreover, part (iii) of \cite[Thm.~8.2]{AA-Inventiones} tells
\begin{align*}
 f - \e^{-t[\D \B_0]}f_0 = E_0^+ f - \e^{-t[\D \B_0]}h^+ + E_0^-f - \e^{-t[\D \B_0]} h^- \in \cY^*,
\end{align*}
where in the first step we have used that $f$ is $\cH$-valued. Now, square Dini convergence follows from Lemma~\ref{Lem: Whitney averages vanish on Y*} taking into account that $\e^{-t[\D \B_0]}f_0 \to 0$ in $\L^2(\Omega)^n$ as $ t\to \infty$, see Proposition~\ref{Prop: Semigroup solution etDB}, and a.e.\ convergence of Whitney averages is a consequence of Lemma~\ref{Lem: Whitney averages vanish on Y*} and Theorem~\ref{Thm: Almost everywhere convergence of Whitney averages}. Finally, $h^+$ is uniquely determined by $f$, since by strong continuity of the $\D \B_0$-semigroup we have $f_t - S_A f_t \to h^+$ as $t \to 0$ in $\L^2(\Omega)^n$. 

Sufficiency  requires a new argument since our notion of weak solutions is different than the purely distributional notion in \cite{AA-Inventiones}. Assume $f \in \cX$ satisfies \eqref{Eq: Representation for X solutions}. First of all, $f \in \Lloc^2(\IR^+; \cH)$: Indeed $\e^{-t [\D \B_0]} h^+$ is continuous and $\cH$-valued. By Proposition~\ref{Prop: Boundedness of SA} we have $S_A f \in \cX$, which implies local $\L^2$-integrability by Lemma~\ref{Lem: X inside Y*}, and also $S_A f_t \in \cH$ for a.e.\ $t>0$ since this is true for the approximants $S_A^\eps f_t$. 

From Proposition~\ref{Prop: Semigroup solution etDB} we already know that $\e^{-t [\D \B_0]} h^+$ is a weak solution to the first-order system for $\B_0$. So, in order to conclude that $f$ is a weak solution to the system for $\B$, it remains to prove
\begin{align}
\label{Eq1: Representation for X solutions}
 \int_0^\infty \bigscal{S_A f_t}{\partial_t g_t - \B_0^* \D g_t}_2 \; \d t = - \int_0^\infty \bigscal{\cE_t f_t}{\D g_t}_2 \qquad (g \in \C_c^\infty(\IR^+; \dom(\D)).
\end{align}
Fix $g \in \C_c^\infty(\IR^+; \dom(\D))$. In view of Proposition~\ref{Prop: Boundedness of SA} it suffices to replace $S_A$ by $S_A^\eps$ and prove that the left-hand side converges to the right-hand side as $\eps \to 0$. For the $(0,t)$-integral in the definition of $S_A^\eps$ a short calculation, using Fubini's theorem and integration by parts in the $t$-variable in the first step, reveals that
\begin{align*}
 &\quad \int_0^\infty \int_0^t \eta_\eps^+(t,s) \bigscal{\D \B_0 \e^{-(t-s)[\D \B_0]} \widehat{E}_0^+ \cE_s f_s}{\partial_t g_t - \B_0^* \D g_t}_2 \; \d s \; \d t \\
 &=-\int_0^\infty \int_s^\infty \bigscal{\partial_t \eta_\eps^+(t,s) \D \B_0 \e^{-(t-s)[\D \B_0]} \widehat{E}_0^+ \cE_s f_s}{g_t}_2 \; \d t \; \d s \\
 &=-\int_0^\infty \biggscal{\widehat{E}_0^+ \cE_s f_s}{\int_s^\infty \partial_t \eta_\eps^+(t,s) \e^{-(t-s)[\B^*_0 \D]} \B^*_0 \D  g_t \; \d t}_2  \; \d s.
\end{align*}
If we let $\eps$ so small that $\supp g \subseteq (2 \eps, \frac{1}{\eps})$, then this equals
\begin{align*}
 -\int_0^\infty \biggscal{\widehat{E}_0^+ \cE_s f_s}{\eta_\eps(s) \barint_{s+\eps}^{s+2\eps} \e^{-(t-s)[\B^*_0 \D]} \B^*_0 \D  g_t \; \d t}_2  \; \d s
\end{align*}
by definition of $\eta_\eps^\pm$. Since $\e^{-(t-s)[\B^*_0 \D]} \B^*_0 \D  g_t$ is uniformly bounded and continuous in $t$ with respect to the $\L^2(\Omega)^n$-topology, the $\d t$-integrals are uniformly bounded in $s$ and converge locally uniformly to $\B^*_0 \D g_s$ as $\eps \to 0$. Note that these integrals are non-zero only when $\dist(s, \supp g) < 2 \eps$, so that for $\eps$ small we are in fact integrating in $s$ over a compact subset of $(0,\infty)$. Due to $\widehat{E}_0^+ \cE f \in \Lloc^2(0,\infty; \cH)$, dominated convergence applies as $\eps \to 0$ and yields the limit
\begin{align*}
 -\int_0^\infty \bigscal{\widehat{E}_0^+ \cE_s f_s}{\B^*_0 \D  g_s}  \; \d s.
\end{align*}
A similar calculation applies to the $(t,\infty)$-integral in the definition of $S_A^\eps$, so that altogether
\begin{align*}
 \int_0^\infty \bigscal{S_A^\eps f_t}{\partial_t g_t - \B_0^* \D g_t}_2 \; \d t \stackrel{\eps \to 0}{\longrightarrow} -\int_0^\infty \bigscal{( \widehat{E}_0^+  + \widehat{E}_0^- )\cE_t f_t}{\B^*_0 \D  g_t}_2  \; \d t.
\end{align*}
Now, $\widehat{E}_0^+  + \widehat{E}_0^- = \B_0^{-1} P_{\B_0 \cH}$, where $P_{\B_0 \cH}$ annihilates $\Ke(\D) = \cH^\pe$, see Proposition~\ref{Prop: DB properties}, so that $\scal{( \widehat{E}_0^+  + \widehat{E}_0^- )\cE_t f_t}{\B^*_0 \D  g_t}_2 = \scal{P_{\B_0 \cH} \cE_t f_t}{\D g_t}_2 = \scal{\cE_t f_t}{\D g_t}_2$. Hence, our goal \eqref{Eq1: Representation for X solutions} follows.
\end{proof}

We record an immediate corollary for systems with $t$-independent coefficients.

\begin{corollary}
\label{Cor: X-solutions t-independent}
Assume that the coefficients $A = A_0$ are $t$-independent and let $f \in \cX$. Then $f$ is a weak solution to the first-order system for $\B = \B_0$ if and only if for some $h^+ \in E_0^+ \cH$, which then is unique, it satisfies
\begin{align*}
 f_t = \e^{-t [\D \B_0]} h^+ \qquad (\text{a.e.\ $t>0$}).
\end{align*}
In this case, $f$ has additional regularity as specified in Proposition~\ref{Prop: Semigroup solution etDB}.
\end{corollary}
\subsection{The Dirichlet problem}
\label{Subsec: The Dirichlet problem}

Things are a little more involved for the Dirichlet problem since here we cannot work with the first-order system only. In particular, similar to the proof of Proposition~\ref{Prop: f are conormals of u} a dichotomy between the cases $\Dir \neq \emptyset$ and $\Dir = \emptyset$ occurs when it comes to recovering the potential $u$ from its conormal gradient.

We begin with a representation theorem for weak solutions $f \in \cY$ to the first-order system. It uses the bounded projections
\begin{align*}
 \widetilde{E}_0^\pm := \ind_{\IC^\pm}(\B_0 \D) P_{\B_0 \cH}
\end{align*}
on $\L^2(\Omega)^n$, where as before $P_{\B_0 \cH}$ is the projection onto $\B_0 \cH$ along the splitting $\L^2(\Omega)^n = \Ke(\D) \oplus \B_0 \cH$ and the Hardy projection $\ind_{\IC^\pm}(\B_0 \D)$ is defined on $\Rg(\B_0 \D) = \B_0 \cH$ by means of the $\H^\infty$-calculus for $\B_0 \D$, see Section~\ref{Sec: Quadratic estimates for DB0} for details.

\begin{theorem}
\label{Thm: Representation for Y solutions}
Assume that $\|\cE\|_C < \infty$ and let $f \in \cY$. Then $f$ is a weak solution to the first-order system for $\B$ if and only if for some $\widetilde{h}^+ \in \widetilde{E}_0^+ \L^2(\Omega)^n$, which then is unique, it satisfies
\begin{align*}
 f_t = \D \e^{-t [\B_0 \D]} \widetilde{h}^+ + S_A f_t \qquad (\text{a.e.\ $t>0$}).
\end{align*}
\end{theorem}

\begin{proof}
Necessity of \eqref{Eq: Representation for X solutions} has been proved in \cite[Thm.~9.2]{AA-Inventiones} by taking limits $\eps \to 0$ in the Duhamel-type formulas from Lemma~\ref{Lem: Towards semigroup representation} for $\eta^\pm = \eta_\eps^\pm$. As for uniqueness of $\widetilde{h}^+$, we assume $\D \e^{-t [\B_0 \D]} \widetilde{h}^+ = 0$ for a.e.\ $t>0$ and check $\widetilde{h}^+ = 0$: In fact, due to $\e^{-t [\B_0 \D]} \widetilde{h}^+ \in \cl{\Rg(\B_0 \D)}$ we first conclude $\e^{-t [\B_0 \D]} \widetilde{h}^+ = 0$ from Proposition~\ref{Prop: DB properties} and then $\widetilde{h}^+ = 0$ follows from strong continuity of the $\B_0 \D$-semigroup.

For sufficiency we note
\begin{align}
\label{Eq: Relation DSA = SATilde}
 \e^{-t [\B_0 \D]} \widetilde{E}_0^\pm = \e^{-t [\B_0 \D]} \ind_{\IC^\pm}(\B_0 \D) \B_0 \B_0^{-1} P_{\B_0 \cH} = \B_0 \e^{- t[ \D \B_0]} \widehat{E}_0^\pm \qquad (t>0).
\end{align}
due to the intertwining property, see \eqref{Hinfty 4} in Section~\ref{Sec: Quadratic estimates for DB0}. In particular, $\D \e^{-t [\B_0 \D]} \widetilde{h}^+$, $t>0$, is a weak solution to the first-order system for $\B_0$  due to Remark~\ref{Rem: Semigroup solutions without an L2 trace}. This being said, the exact same reasoning as in the proof of Theorem~\ref{Thm: Representation for X solutions} yields the claim.
\end{proof}

In order to recover a potential $u$ from the representation for $f = \nablaA u$ provided by the previous theorem, we introduce integral operators $\widetilde{S}_A^\eps$ similar to $S_A^\eps$. For the definition of $\eta_\eps^\pm$ see Proposition~\ref{Prop: Boundedness of SA}. Below, we denote spaces of bounded continuous functions by $\C_b$.

\begin{proposition}[{\cite[Prop.~7.2]{AA-Inventiones}}]
\label{Prop: Boundedness of SAtilde}
Suppose $\|\cE\|_C < \infty$. For $\eps > 0$ the operators
\begin{align*}
 \widetilde{S}_A^\eps f_t = \int_0^t \eta_\eps^+(t,s) \e^{-(t-s)[\B_0 \D]} \widetilde{E}_0^+ \cE_s f_s \; \d s -  \int_t^\infty \eta_\eps^-(t,s) \e^{-(s-t)[\B_0 \D]} \widetilde{E}_0^- \cE_s f_s \; \d s
\end{align*}
are bounded $\cY \to \C_b([0,\infty); \L^2(\Omega)^n)$ with $\sup_{t>0} \|\widetilde{S}_A^\eps f_t\|_{\L^2(\Omega)^n} \lesssim \|\cE\|_C \|f\|_\cY$ uniformly in $\eps > 0$. There is a limit operator $\widetilde{S}_A \in \Lop(\cY; \C_b([0,\infty); \L^2(\Omega)^n))$ such that $\lim_{\eps \to 0} \|\widetilde{S}_A^\eps f_t - \widetilde{S}_A f_t\|_{\L^2(\Omega)^n}  = 0$ locally uniformly in $t>0$ for any $f \in \cY$. This operator is given by
\begin{align*}
 \widetilde{S}_A f_t = \int_0^t \e^{-(t-s)[\B_0 \D]} \widetilde{E}_0^+ \cE_s f_s \; \d s -  \int_t^\infty \e^{-(s-t)[\B_0 \D]} \widetilde{E}_0^- \cE_s f_s \; \d s \qquad (f \in \cY),
\end{align*}
where the integrals exist as weak integrals in $\L^2(\Omega)^n$, and it has $\L^2(\Omega)^n$-limits
\begin{align*}
 \lim_{t \to 0} \widetilde{S}_A f_t = - \int_0^\infty \e^{-s [\B_0 \D]} \widetilde{E}_0^- \cE_s f_s \; \d s \in \widetilde{E}_0^- \L^2(\Omega)^n \quad \text{and} \quad \lim_{t \to \infty} \widetilde{S}_A f_t = 0.
\end{align*}
\end{proposition}

\begin{corollary}
\label{Cor: Identities for SAtilde}
Suppose $\|\cE\|_C < \infty$ and let $f \in \cY$. Then,
\begin{enumerate}
 \item \label{Cor: Identities for SAtilde i} $\widetilde{S}_A f_t \in \dom(\D)$ and $\D \widetilde{S}_A f_t = S_A f_t$ for almost every $t>0$,
 \item $\widetilde{S}_A f \in \Wloc^{1,2}(\IR^+; \L^2(\Omega)^n)$ with $\partial_t \widetilde{S}_A f = - \B_0 S_A f + P_{\B_0 \cH} \cE f$.
\end{enumerate}
\end{corollary}

\begin{proof}
\begin{enumerate} 
 \item Due to \eqref{Eq: Relation DSA = SATilde} and since the integrals defining $\widetilde{S}_A^\eps f_t$ and $S_A^\eps f_t$ are absolutely convergent Bochner integrals in $\L^2(\Omega)^n$, we have $\D \widetilde{S}_A^\eps f_t = S_A^\eps f_t$ for every $t>0$. In the limit $\eps \to 0$ there is convergence $\widetilde{S}_A^\eps f_t \to \widetilde{S}_A f_t$ in $\L^2(\Omega)^n$ for every $t>0$, see Proposition~\ref{Prop: Boundedness of SAtilde}. Moreover, for a subsequence of $\eps$ there also is convergence $S_A^\eps f_t \to S_A f_t$ in $\L^2(\Omega)^n$ for almost every $t>0$. This follows from Proposition~\ref{Prop: Boundedness of SA} using that convergence in $\cY = \L^2(0,\infty, t \d t; \L^2(\Omega)^n)$ implies pointwise almost everywhere convergence of a subsequence. As $\D$ is a closed operator in $\L^2(\Omega)^n$, the conclusion follows.
 \item Let $g \in \C_c^\infty(\IR^+; \L^2(\Omega)^n)$. A calculation identical to the one in the proof of Theorem~\ref{Thm: Representation for X solutions} reveals
 \begin{align*}
 \qquad \qquad \int_0^\infty \bigscal{\widetilde{S}_A^\eps f_t}{\partial_t g_t}_2 - \bigscal{\B_0 \D \widetilde{S}_A^\eps f_t}{g_t}_2 \; \d t \stackrel{\eps \to 0}{\longrightarrow} - \int_0^\infty \bigscal{( \widehat{E}_0^+  + \widehat{E}_0^- )\cE_t f_t}{g_t}_2  \; \d t.
\end{align*}
 The right-hand side coincides with $-\int_0^\infty \scal{P_{\B_0 \cH} \cE_t f_t}{g_t}_2  \; \d t$, whereas the left-hand side tends to $\int_0^\infty \scal{\widetilde{S}_A f_t}{\partial_t g_t}_2 - \scal{\B_0 S_A f_t}{g_t}_2 \; \d t$ using \eqref{Cor: Identities for SAtilde i}. Hence, $\partial_t \widetilde{S}_A f = - \B_0 S_A f + P_{\B_0 \cH} \cE f$ in the distributional sense. This derivative is contained in $\Lloc^2(\IR^+; \L^2(\Omega)^n)$ since $S_A f \in \cY$ due Proposition~\ref{Prop: Boundedness of SA} and as $\cE$ is bounded (see Remark~\ref{Rem: Comparison to standard Carleson norm}). \qedhere
\end{enumerate}
\end{proof}

For the Dirichlet problem in case of pure lateral Neumann boundary conditions we also need the subsequently introduced integral operator $T_A$, which will be responsible for a part of $u$ that is contained in the space of constant functions on $\Omega$.  Note that we obtain its boundedness only on the subspace of $\cY$ containing the weak solutions to the first-order system for $\B$. In fact, boundedness on the whole of $\cY$ would require a stronger integrability condition on $\cE$.

\begin{definition}
\label{Def: N+-}
On $\L^2(\Omega)^n$ define the orthogonal projections $N^\pm$ and the reflection $N$ by
\begin{align*}
 N^- h := \begin{bmatrix} h_\pe \\ 0 \end{bmatrix}, \quad N^+ h := \begin{bmatrix} 0 \\ h_\pa\end{bmatrix}, \quad Nh := N^+h- N^-h.
\end{align*}
\end{definition}

\begin{proposition}
\label{Prop: TA boundedness}
Assume $\|\cE\|_C < \infty$. For every weak solution $f \in \cY$ to the first-order system for $\B$ it holds
\begin{align*}
 \sup_{0<t \leq \infty} \bigg\| \int_0^t \cE_s f_s \; \d s \bigg\|_{\L^2(\Omega)^n} \lesssim \|\cE\|_C \|f\|_\cY,
\end{align*}
where the integrals exist as weak integrals in $\L^2(\Omega)^n$. In particular, the weak integrals 
\begin{align*}
 (T_A f)_t := \int_0^t N^- (1-P_{\B_0 \cH})  \cE_s f_s \; \d s - \int_t^\infty N^+ (1-P_{\B_0 \cH}) \cE_s f_s \; \d s \qquad (t>0)
\end{align*}
are defined in $\L^2(\Omega)^n$ and satisfy $\sup_{t>0} \|T_A f_t\|_{\L^2(\Omega)^n} \lesssim \|\cE\|_C \|f\|_\cY$. Moreover, if $\Dir$ is non-empty, then $(T_A f)_\pe$ is contained in $\Wloc^{1,2}(\IR^+; \IC^m)$, identified with a subset of $\Wloc^{1,2}(\IR^+; \L^2(\Omega)^m)$, has distributional derivative $\partial_t (T_A f)_\pe = ((1-P_{\B_0 \cH}) \cE f)_\pe$, and limits
\begin{align*}
 \lim_{t \to 0} (T_A f_t)_\pe = 0 \quad \text{and} \quad \lim_{t \to \infty} (T_A f_t)_\pe = \int_0^\infty \big((1-P_{\B_0 \cH}) \cE_s f_s \big)_\pe \; \d s
\end{align*}
in $\IC^m$, identified with a subset of $\L^2(\Omega)^m$.
\end{proposition}

\begin{proof}
For $h \in \L^2(\Omega)^n$ we have
\begin{align*}
 \int_0^\infty |\scal{\cE_s f_s}{h}_{\L^2(\Omega)^n}| \; \d s
 &\leq \int_0^\infty \int_\Omega |\cE(s,y)||f(s,y)||h(y)| \; \d y \; \d s \\
 &\simeq \int_0^\infty \int_\Omega \bigg(\bariint_{W(t,x)} |\cE(s,y)||sf(s,y)||h(y)| \; \d s \; \d y \bigg) \; \frac{\d x \; \d t}{t},
\end{align*}
where the ``averaging trick'' in the second step uses Tonelli's theorem and $|W(t,x)| \simeq |W(s,y)|$ for $(t,x) \in W(s,y)$ with implicit constants depending on $c_0$, $c_1$, and $\Omega$. The latter follows since for $s$ small we have $|W(s,y)| \simeq s^{1+d}$ ($d$-Ahlfors regularity of $\Omega$) and for $s$ large we have $|W(s,y)| \simeq s$ (boundedness of $\Omega$). We fix $p \in (2_*,2)$, let $\frac{1}{p} + \frac{1}{q} = 1$, and introduce
\begin{align*}
F(t,x):=  \bigg(\bariint_{2W(t,x)}|sf(s,y)|^2 \; \d s \; \d y \bigg)^{1/2}, \quad G(t,x):= \bigg(\barint_{B(x,c_1t) \cap \Omega} |h(y)|^q \; \d y \bigg)^{1/q}.
\end{align*}
By H\"older's inequality followed by an application of the reverse H\"older estimate for $f$ (Corollary~\ref{Cor: Reverse Holder for semigroup solutions}),
\begin{align*}
\int_0^\infty |\scal{\cE_s f_s}{h}_{\L^2(\Omega)^n}| \; \d s
 &\lesssim \int_0^\infty \int_\Omega \bigg(\sup_{W(t,x)} |\cE| \bigg) F(t,x) G(t,x) \frac{\d x \; \d t}{t}.
\intertext{This is the only part where we use that $f$ is a weak solution to the first-order system and not just an element of $\cY$. Now, the tent space estimate of Coifman-Meyer-Stein \cite{CMS-Tent} applies (see \cite[Prop.~3.15]{Amenta-Tent} for a proof on doubling spaces; notation is explained further below):}
 &\lesssim \|\cE\|_C \int_\Omega \mathcal{A}(F G)(z) \; \d z \\
 &\lesssim \|\cE\|_C \int_\Omega \mathcal{A}(F)(z) M_\Omega(|h|^q)(z)^{1/q} \; \d z \\
 &\lesssim \|\cE\|_C \|\mathcal{A}(F)\|_{\L^2(\Omega)} \|M_\Omega(|h|^q)(z)\|_{\L^{2/q}(\Omega)}^{1/q},
\end{align*}
where $\mathcal{A}$ denotes the area function
\begin{align*}
 \mathcal{A}(g)(z):= \bigg(\iint_{|x-z|<t} |g(t,x)|^2 \; \frac{\d x}{|B(z,t)|} \; \frac{\d t}{t} \bigg)^{1/2}
\end{align*}
and $M_\Omega$ is the centered Hardy-Littlewood maximal operator in $\Omega$ defined by
\begin{align*}
 M_\Omega(g)(z) := \sup_{t>0} \barint_{B(z,t) \cap \Omega} |g(z)| \; \d z.
\end{align*}
The same averaging trick as above reveals
\begin{align*}
 \int_\Omega |\mathcal{A}(F)(z)|^2 \; \d z \simeq \int_0^\infty \int_\Omega |F(t,x)|^2 \; \d x \; \frac{\d t}{t} \simeq \int_0^\infty \int_\Omega |sf(s,y)|^2 \; \d y \; \frac{\d s}{s} \simeq \|f\|_\cY^2.
\end{align*}
Moreover, $M_\Omega$ is bounded on $\L^{q/2}(\Omega)$ since $q>2$ and $\Omega$ is doubling \cite[Thm.~3.13]{Bjorn-Bjorn}. Altogether,
\begin{align*}
 \int_0^\infty |\scal{\cE_s f_s}{h}_{\L^2(\Omega)^n}| \; \d s \lesssim \|\cE\|_C \|f\|_\cY \|h\|_{\L^2(\Omega)^n},
\end{align*}
which proves the integral estimate. Boundedness of the operator $T_A$ is immediate from that. In order to conclude the further properties of $(T_A)_\pe$, we note that by the first part
\begin{align*}
 (T_A f_t)_\pe := \int_0^t \big((1-P_{\B_0 \cH})  \cE_s f_s \big)_\pe \; \d s
\end{align*}
is defined for $0 \leq t \leq \infty$ as a weak integral in $\L^2(\Omega)^m$, but in fact the integrand is valued in the finite-dimensional subspace $\Rg(1-P_{\B_0 \cH})_\pe = \Ke(\D)_\pe$ containing only the constant functions. Hence, these integrals exist as proper $\IC^m$-valued Bochner integrals and the limits as $t \to 0$ and $t \to \infty$ as well as differentiability in $t$ follow easily.
\end{proof}

Now, we are in a position to prove the \emph{a priori} estimates for the Dirichlet problem claimed in our main result Theorem~\ref{Thm: Main result Dir}. Note carefully that in contrast to Theorem~\ref{Thm: Representation for X solutions} the result is formulated at the level of the second-order system and that sufficiency of the representation formulas requires a smallness condition on $\|\cE\|_C$.

\begin{theorem}
\label{Thm: Representation for Dirichlet problem} \quad
\begin{enumerate}
\def\theenumi{\roman{enumi}}
 \item \label{Representation Dir} Assume $\|\cE\|_C < \infty$ and that the lateral Dirichlet part $\Dir$ is non-empty. If $u$ is a weak solution to the second-order system with interior estimate $\nablatx u \in \cY$, then there exists $\widetilde{h}^+ \in \widetilde{E}_0^+ \L^2(\Omega)^n$ such that
 \begin{align*}
  u_t = - \Big(\e^{-t [\B_0 \D]} \widetilde{h}^+ + \widetilde{S}_A (\nablaA u)_t \Big)_\pe \qquad (\text{a.e.\ $t>0$}).
 \end{align*}
 In this case $u \in \C([0,\infty); \L^2(\Omega)^m)$. Moreover, let $\widetilde{h}^- := -\int_0^\infty \e^{-s[\B_0 \D]} \widetilde{E}_0^- \cE_s \nablaA u_s \; \d s \in \widetilde{E}_0^- \L^2(\Omega)^n$ and $v_0 := \widetilde{h}^+ + \widetilde{h}^-$. Then there are $\L^2(\Omega)^m$-limits
 \begin{align*}
  \lim_{t \to 0} u_t = -(v_0)_\pe  \quad \text{and} \quad \lim_{t \to \infty} u_t = 0
 \end{align*}
 and estimates
 \begin{align*}
  \|(v_0)_\pe\|_{\L^2(\Omega)^m}\leq \sup_{t> 0} \|u_t\|_{\L^2(\Omega)^m} \lesssim \|\nablatx u\|_\cY < \infty.
 \end{align*}
 If furthermore $\|\cE\|_C$ is sufficiently small, then $u$ is a weak solution to the second-order system with interior estimate $\nablatx u \in \cY$ if and only if
 \begin{align*}
  \qquad \qquad u = - \Big(\e^{-t [\B_0 \D]} \widetilde{h}^+ + \widetilde{S}_A f \Big)_\pe \quad \text{with} \quad f = (1- S_A)^{-1} \D  \e^{-t [\B_0 \D]} \widetilde{h}^+
 \end{align*}
 for some $\widetilde{h}^+ \in \widetilde{E}_0^+ \L^2(\Omega)^n$ satisfying $\|\widetilde{h}^+\|_{\L^2(\Omega)^n} \simeq \|\nablatx u\|_\cY$. In this case, $f = \nablaA u$.

 \item \label{Representation Dir pure Neumann} Assume $\|\cE\|_C < \infty$ and that the lateral Dirichlet $\Dir$ part is empty. Then a function $u$ is a weak solution to the second-order system with interior estimate $\nablatx u \in \cY$ if and only if for some $\widetilde{h}^+ \in \widetilde{E}_0^+ \L^2(\Omega)^n$ and some $c \in \IC^m$ it satisfies
 \begin{align*}
  u_t = - \Big(\e^{-t [\B_0 \D]} \widetilde{h}^+ + \widetilde{S}_A (\nablaA u)_t + T_A (\nablaA u)_t \Big)_\pe + c \qquad (\text{a.e.\ $t>0$}).
 \end{align*}
 In this case $u \in \C([0,\infty); \L^2(\Omega)^m)$. Let $\widetilde{h}^-$ and $v_0$ be as before. Then there are $\L^2(\Omega)^m$-limits
 \begin{align*}
  \qquad \qquad \lim_{t \to 0} u_t = c -(v_0)_\pe \quad \text{and} \quad \lim_{t \to \infty} u_t = c - \int_0^\infty \big((1- P_{\B_0 \cH})\cE_s (\nablaA u)_s \big)_\pe \; \d s =:u_\infty \in \IC^m
 \end{align*}
 and estimates
 \begin{align*}
  \qquad \|c-(v_0)_\pe\|_{\L^2(\Omega)^m} \leq \sup_{t> 0} \|u_t\|_{\L^2(\Omega)^m} \lesssim |c|+ \|\nablatx u\|_\cY \simeq |u_\infty| + \|\nablatx u\|_\cY < \infty.
 \end{align*}
 If furthermore $\|\cE\|_C$ is sufficiently small, then $u$ is a weak solution to the second-order system with interior estimate $\nablatx u \in \cY$ if and only if
 \begin{align*}
  \qquad \qquad u = - \Big(\e^{-t [\B_0 \D]} \widetilde{h}^+ + (\widetilde{S}_A + T_A) f \Big)_\pe + c  \quad \text{with} \quad f = (1- S_A)^{-1} \D  \e^{-t [\B_0 \D]} \widetilde{h}^+
 \end{align*}
 for some $c \in \IC^m$ and some $\widetilde{h}^+ \in \widetilde{E}_0^+ \L^2(\Omega)^n$ satisfying $\|\widetilde{h}^+\|_{\L^2(\Omega)^n} \simeq \|\nablatx u\|_\cY$. In this case, $f = \nablaA u$.
\end{enumerate}
\end{theorem}

\begin{proof}
We begin with the claim for non-empty lateral Dirichlet part. Putting $f:=\nablaA u$, Theorem~\ref{Thm: Representation for Y solutions} and Proposition~\ref{Prop: f are conormals of u} yield an $\widetilde{h}^+ \in \widetilde{E}_0^+ \L^2(\Omega)^n$ such that
\begin{align}
\label{Eq0: Representation for Dirichlet problem}
 f_t = \D \e^{-t [\B_0 \D]} \widetilde{h}^+ + S_A f_t \qquad (\text{a.e.\ $t>0$}).
\end{align}
We define the potential $v:= \e^{-t [\B_0 \D]} \widetilde{h}^+ + \widetilde{S}_A f$, so that $\D v = f$ by Corollary~\ref{Cor: Identities for SAtilde}. Invoking Corollary~\ref{Cor: Identities for SAtilde}, we find
\begin{align}
\label{Eq1: Representation for Dirichlet problem}
 \partial_t v_\pe = \Big(- \B_0 \D \e^{-t [\B_0 \D]} \widetilde{h}^+ - \B_0 S_A f  + P_{\B_0 \cH} \cE f \Big)_\pe = \Big(-\B f - (1-P_{\B_0 \cH}) \cE f\Big)_\pe.
\end{align}
Since $1-P_{\B_0 \cH}$ projects onto $\Ke(\D)$ and as $\Ke(\D)_\pe = \{0\}$,
\begin{align*}
 \nablatx v_\pe 
= \begin{bmatrix} \partial_t v_\pe \\ \nablax v_\pe \end{bmatrix} 
= - \begin{bmatrix} (\B f)_\pe \\ f_\pa \end{bmatrix}
=- \cl{A}^{-1} f = - \cl{A}^{-1} \nablaA u = \nablatx u.
\end{align*}
Hence, $v_\pe + u$ is constant on the domain $\IR^+ \times \Omega$ and in fact $u = -v_\pe$ again by Poincar\'{e}'s inequality on $\V$. Now, continuity and limits for $u$ follow from the respective properties for $\widetilde{S}_A f$ as provided by Proposition~\ref{Prop: Boundedness of SAtilde} and the analog of Proposition~\ref{Prop: Semigroup solution etDB} for the $[\B_0 \D]$-semigroup. As for the estimates, Proposition~\ref{Prop: Boundedness of SAtilde} implies  
\begin{align*}
\|(v_0)_\pe\|_2  \leq \sup_{t>0} \|u_t\|_2 \leq \sup_{t>0} \|v_t\|_2 \lesssim \|\widetilde{h}^+\|_2 + \|f\|_\cY
\end{align*}
and, taking into account quadratic estimates for $\B_0 \D$ (Theorem~\ref{Thm: Quadratic estimates fo DB}), accretivity of $\B_0$, and Proposition~\ref{Prop: Boundedness of SA},
\begin{align}
\label{Eq2: Representation for Dirichlet problem}
\|\widetilde{h}^+\|_2 
\simeq \|\B_0 \D \e^{-t[\B_0 \D]} \widetilde{h}^+\|_\cY 
\simeq \|(1-S_A) f\|_\cY
\lesssim \|f\|_\cY.
\end{align}

If we additionally assume that $\|\cE\|_C$ is sufficiently small, then Proposition~\ref{Prop: Boundedness of SA} shows that $1-S_A$ is invertible on $\cY$ and given $\widetilde{h}^+ \in \widetilde{E}_0^+ \L^2(\Omega)^n$ we can solve for $f$ in \eqref{Eq0: Representation for Dirichlet problem} by $f = (1-S_A)^{-1}\D \e^{-t [\B_0 \D]} \widetilde{h}^+$. Setting $v = \e^{-t [\B_0 \D]} \widetilde{h}^+ + \widetilde{S}_A f$, we have already proved that $u:=-v_\pe$ is a weak solution to the second-order system with $\nablaA u = f$ and that conversely every such weak solution is of that type. Finally, $\|(1-S_A) f\|_\cY \simeq \|f\|_\cY \simeq \|\nablatx u\|_\cY$ by invertibility of $S_A$ and therefore $\|\widetilde{h}^+\|_2 \simeq  \|f\|_\cY$ due to \eqref{Eq2: Representation for Dirichlet problem}.

Now, we consider the case of empty lateral Dirichlet part. Using the same notation as before, the only critical difference in the argument is that due to $\Ke(\D)_\pe = \IC^m$ we cannot conclude $(\partial_t v)_\pe = -(\B f)_\pe$ from \eqref{Eq1: Representation for Dirichlet problem}. However, we have at hand Proposition~\ref{Prop: TA boundedness} and our substitute for \eqref{Eq1: Representation for Dirichlet problem} becomes
\begin{align*}
 \partial_t (v + T_A f)_\pe = \Big(-\B f - (1-P_{\B_0 \cH}) \cE f \Big)_\pe + \Big((1-P_{\B_0 \cH}) \cE f\Big)_\pe = - (\B f)_\pe.
\end{align*}
The rest of the argument is identical to the case of non-empty lateral Dirichlet part, except that now $u - (v + T_A f)_\pe$ may be a non-zero constant $c \in \IC^m$ and $|c| + \|f\|_\cY$ is comparable to $|u_\infty| + \|f\|_\cY$ due to Proposition~\ref{Prop: TA boundedness}.
\end{proof}

Again results in the case of $t$-independent coefficients are particularly simple.

\begin{corollary}
\label{Cor: Y-solutions t-independent}
Suppose that the coefficients $A = A_0$ are $t$-independent. Then $u$ is a weak solution to the second-order system with Lusin area bound $\nabla_{t,x} u \in \cY$ if and only if there exist $\widetilde{h}^+ \in \widetilde{E}_0^+ \L^2(\Omega)^n$ and a constant $c \in \IC^m$, which is zero if the lateral Dirichlet part $\Dir$ is non-empty, such that
\begin{align*}
 u_t = - (\e^{-t [\B_0 \D]} \widetilde{h}^+)_\pe + c \qquad (\text{a.e.\ $t>0$}).
\end{align*}
In this case, $u$ has additional regularity $u \in \C([0, \infty); \L^2(\Omega)^m) \cap \C^\infty((0,\infty); \V)$.
\end{corollary}

\begin{proof}
Necessity and sufficiency of the semigroup representation for $u$ is due to Theorem~\ref{Thm: Representation for Dirichlet problem}. The additional regularity follows since $t \mapsto \e^{-t[\B_0 \D]}\widetilde{h}^+$ is holomorphic with values in $\dom(\B_0 \D) = \dom(\D)$.
\end{proof}

Our final goal in this section is to prove that weak solutions to the Dirichlet problem obtained in Theorem~\ref{Thm: Representation for Dirichlet problem} also have an $\L^2$-controlled maximal function and, in particular, that they converge to their trace at $t=0$ in the sense of Whitney averages. As a technical tool we need the following $\L^p$-non-tangential maximal functions.

\begin{definition}
\label{Def: NTp}
For $1\leq p \leq 2$ define
\begin{align*}
 \NT^p(f)(x) := \bigg(\sup_{0<t<1} \bariint_{W(t,x)} |f(s,y)|^p \; \d y \; \d s \bigg)^{1/p} \qquad (f \in \Lloc^p(\IR^+ \times \Omega))
\end{align*}
\end{definition}

By Jensen's inequality $\NT^p \leq \NT^2$ pointwisely almost everywhere on $\Omega$. Moreover, $\NT^2 \leq \NT$ since $\NT$ takes the supremum over all Whitney balls. 

The following lemma is proved in \cite[Lem.~10.2(iii)]{AA-Inventiones} using the tent space estimate of Coifman, Meyer, and Stein (see again \cite{Amenta-Tent} for a proof on doubling spaces) and $\L^q$ off-diagonal estimates for $\D \B_0^*$ with $\abs{q-2}$ sufficiently small. The only necessary change in the argument in \cite{AA-Inventiones} is that we have to take the supremum in the definition of $\NT^p$ over $0 < t < t_0$ with $t_0 = 1$ instead of $t_0 = \infty$. The reason for this is that off-diagonal estimates are required for all $t<c_0 t_0$ and Corollary~\ref{Cor: Lp off diagonals} does only provide them on bounded ranges for $t$.

\begin{lemma}
\label{Lem: NTp bound resolvents}
Assume $\|\cE\|_C < \infty$ and $1 \leq p < 2$. Let $f \in \cY$ be such that $\int_0^{c_0} \cE_s f_s \; \d s$ is defined as weak integral. Then
\begin{align*}
 \bigg\|\NT^p \bigg((1+\i t \B_0 \D)^{-1} \int_0^t \cE_s f_s \; \d s\bigg) \bigg\|_{\L^2(\Omega)} \lesssim \|\cE\|_C \|f\|_\cY
\end{align*}
\end{lemma}

\begin{remark}
\label{Rem: NTp bound resolvents}
In view of Proposition~\ref{Prop: TA boundedness} this estimate in particular applies to $1_{(0,t_0)}f$, where $f \in \cY$ is a weak solution to the first-order system for $\B$ and $t_0 > 0$.
\end{remark}

\begin{theorem}
\label{Thm: NT bounds Dirichlet}
Assume $\|\cE\|_C < \infty$. Let $u$ be a weak solution with interior estimate $\nablatx u \in \cY$ and let $u_0 := \lim_{t \to 0} u_t$ and $u_\infty:= \lim_{t \to \infty} u_t$ as in Theorem~\ref{Thm: Representation for Dirichlet problem}. Then
\begin{align*}
 \|u_0\|_{\L^2(\Omega)^m} \lesssim \|\NT(u)\|_{\L^2(\Omega)} \lesssim |u_\infty| + \|\nablatx u\|_\cY 
\end{align*}
and $u$ converges to $u_0$ also in the sense of Whitney averages
 \begin{align*}
  \lim_{t \to 0} \bariint_{W(t,x)} u(s,y) \; \d s \; \d y = u_0(x) \qquad (\text{a.e.\ $x \in \Omega$}).
 \end{align*}
\end{theorem}

\begin{proof}
The proof of the $\L^2$-bounds for $\NT(u)$ is similar to the one of \cite[Thm.~10.1]{AA-Inventiones}. In order to carve out the subtle modifications that are necessary in our geometric framework, we decided to reproduce the argument in some detail, though. Also, this sets the stage for the convergence of Whitney averages, which was left as an open question in \cite{AA-Inventiones} and was addressed further in \cite{Auscher-Stahlhut_APriori}.

We adopt notation from Theorem~\ref{Thm: Representation for Dirichlet problem} and put $f:= \nablaA u \in \cY$. The argument is subdivided into five consecutive steps.

\subsection*{\normalfont \itshape Step 1: Lower bound for $\NT(u)$}

\noindent As $\lim_{t \to 0} u_t = u_0$ in $\L^2(\Omega)^m$, Lemma~\ref{Lem: X inside Y*} gives $\|\NT(u)\|_2^2 \gtrsim \lim_{t \to 0} \barint_t^{2t} \|u_s\|_s^2 \; \d s = \|u_0\|_2$.

\subsection*{\normalfont \itshape Step 2: Non-tangential estimate of $\e^{-t[\B_0 \D]} \widetilde{h}^-$}

\noindent By the intertwining property \eqref{Hinfty 4} in Section~\ref{Sec: Quadratic estimates for DB0} we have for every $\widetilde{h} \in \B_0 \cH$ that $\e^{-t[\B_0 \D]} \widetilde{h} = \B_0 \e^{-t[\D \B_0]} \B_0^{-1}\widetilde{h}$ and hence $\|\e^{-t[\B_0 \D]} \widetilde{h}\|_\cX \lesssim \|\widetilde{h}\|_2$ due to Corollary~\ref{Cor: Reverse Holder for semigroup solutions} and accretivity of $\B_0$. In particular, the non-tangential maximal function of $\e^{-t[\B_0 \D]} \widetilde{h}^-$ is bounded in $\L^2(\Omega)$-norm by 
\begin{align*}
 \|\widetilde{h}^-\|_2
= \sup_{\|g\|_2 = 1} \bigg|\int_0^\infty \scal{\e^{-s[\B_0 \D]} \widetilde{E}_0^- \cE_s f_s}{g}_2 \; \d s \bigg|
 \lesssim \sup_{\|g\|_2 = 1} \|f\|_\cY \cdot \|\cE^*  \e^{-s[\D \B_0^*]} (\widetilde{E}_0^-)^* g\|_{\cY^*}.
\end{align*}
On noting $\Rg( (\widetilde{E}_0^-)^*) = \Ke( \widetilde{E}_0^-)^\pe \subseteq \Ke(\D)^\pe = \cH$ and $\|\cE^*\|_C = \|\cE\|_C$, Theorem~\ref{Thm: * norm equivalent to Carleson} and Theorem~\ref{Thm: NT bound for semigroup solutions} for $\B_0^*$ in place of $\B_0$ yield 
\begin{align*}
\|\widetilde{h}^-\|_2 
\lesssim \sup_{\|g\|_2 = 1} \|\cE\|_C \cdot \| \e^{-s [\D \B_0^*]}  (\widetilde{E}_0^-)^* g\|_{\cX} \cdot \|f\|_\cY
\lesssim \|\cE\|_C \|f\|_\cY.
\end{align*}

\subsection*{\normalfont \itshape Step 2: Splitting off $\cY^*$-terms from the solution formula for $u$}

\noindent In this step we split off $\cY^*$-terms from the solution formula for $u$ that are completely harmless when it comes to non-tangential estimates and Whitney average convergence. We consider the case of non-empty lateral Dirichlet part $\Dir$ first. Then $u$ satisfies the representation formula in item \eqref{Representation Dir} of Theorem~\ref{Thm: Representation for Dirichlet problem}. Starting out with
\begin{align*}
&\widetilde{S}_A f_t - \e^{-t[\B_0 \D]} \widetilde{h}^-   \\
=& \int_0^t \e^{-(t-s)[\B_0 \D]} \widetilde{E}_0^+ \cE_s f_s \; \d s -  \int_t^\infty \e^{-(s-t)[\B_0 \D]} \widetilde{E}_0^- \cE_s f_s \; \d s + \e^{-t[\B_0 \D]} \int_0^\infty \e^{-s [\B_0 \D]} \widetilde{E}_0^- \cE_s f_s \; \d s,
\intertext{we add correction terms in order to obtain regular decaying kernels for the first two terms}
=&\int_0^t \e^{-(t-s)[\B_0 \D]}(1-\e^{-2s [\B_0 \D]})\widetilde{E}_0^+ \cE_s f_s \; \d s - \int_t^\infty \e^{-(s-t)[\B_0 \D]}(1- \e^{-2t [\B_0 \D]}) \widetilde{E}_0^- \cE_s f_s \; \d s \\
\quad & + \e^{-t[\B_0 \D]} \int_0^t \e^{-s[\B_0 \D]} P_{\B_0 \cH} \cE_s f_s \; \d s,
\intertext{and introduce the holomorphic function $\psi(z) = \e^{-[z]} - (1 + \i z)^{-1}$ to eventually discover}
=&\int_0^t \e^{-(t-s)[\B_0 \D]}(1-\e^{-2s [\B_0 \D]})\widetilde{E}_0^+ \cE_s f_s \; \d s - \int_t^\infty \e^{-(s-t)[\B_0 \D]}(1- \e^{-2t [\B_0 \D]}) \widetilde{E}_0^- \cE_s f_s \; \d s \\
&\quad +\psi(t \B_0 \D) \int_0^\infty \e^{-s[\B_0 \D]}  P_{\B_0 \cH} \cE_s f_s \; \d s - \int_t^\infty \psi(t \B_0 \D) \e^{-s[\B_0 \D]}  P_{\B_0 \cH} \cE_s f_s \; \d s \\
& \quad +\int_0^t (1 + \i t \B_0 \D)^{-1}(\e^{-s[\B_0 \D]} - 1)  P_{\B_0 \cH} \cE_s f_s \; \d s +  P_{\B_0 \cH} (1+\i t \B_0 \D)^{-1} \int_0^t \cE_s f_s \; \d s \\
=:& \, I_1 - I_2 + I_3 - I_4 + I_5 + I_6.
\end{align*}
Recall that the integral occurring in $I_6$ is well-defined as a weak integral by Proposition~\ref{Prop: TA boundedness}. Using the bounded $\H^\infty$-calculus of $\B_0 \D$, the kernel $\e^{-(t-s)[\B_0 \D]}(1-\e^{-2s [\B_0 \D]})\widetilde{E}_0^+$ of $I_1$ is found to be controlled in operator norm by $\frac{s}{t}$ (see also \cite{AA-Inventiones}). The same is true for the kernel of $I_5$ and similarly the kernels of $I_2$ and $I_4$ are controlled by $\frac{t}{s}$. This implies for instance
\begin{align*}
\int_0^\infty \|I_1\|_2^2 \; \frac{\d t}{t} 
\lesssim \int_0^\infty \bigg(\int_0^t \frac{s}{t} \; \frac{\d s}{s} \bigg)\bigg(\int_0^t \frac{s}{t} \|\cE_s f_s\|_2^2 \; s \d s\bigg) \; \frac{\d t}{t}
\leq \|\cE\|_\infty^2 \|f\|_\cY^2
\lesssim \|\cE\|_C^2 \|f\|_\cY^2,
\end{align*}
see Remark~\ref{Rem: Comparison to standard Carleson norm} for the last estimate, and similarly we control $I_2$, $I_4$, and $I_5$. In particular, these integrals are absolutely convergent in $s$ and contained in $\cY^*$ as functions of $t$. For $I_3$ quadratic estimates for $\B_0 \D$ and a duality argument similar to Step~1 give
\begin{align*}
\|I_3\|_{\cY^*} 
\lesssim \bigg\|\int_0^\infty \e^{-s[\B_0 \D]}  P_{\B_0 \cH} \cE_s f_s \; \d s \bigg\|_2
\lesssim \|\cE\|_C \|f\|_\cY.
\end{align*}
This uses $\Rg(P_{\B_0 \cH}^*) = \cH$ and we also get that $I_3$ exists as a weak integral. 

Recalling the representation formula for $u$ from Theorem~\ref{Thm: Representation for Dirichlet problem} and $(P_{\B_0 \cH} h)_\pe = h_\pe$ for all $h \in \L^2(\Omega)^n$ due to $\Ke(\D)_\pe = \Ke(\gV) = \{0\}$, the upshot of all this is
\begin{align}
\label{Eq: NT Dirichlet Result Step 2a}
\begin{split}
\bigg\|u + \bigg(\e^{-t[\B_0 \D]} (\widetilde{h}^- + \widetilde{h}^+) + (1+\i t \B_0 \D)^{-1} \int_0^t \cE_s f_s \; \d s \bigg)_\pe \bigg\|_{\cY^*} \lesssim \|\cE\|_C \|f\|_\cY.
\end{split}
\end{align}
A similar result holds in the case of empty lateral Dirichlet part $\Dir$. In fact, employing $\e^{-s [\B_0 \D]} = \Id$ on $\Rg(1-P_{\B_0 \cH}) = \Ke(\B_0 \D)$, we can write
\begin{align*}
&\widetilde{S}_A f_t + T_A f_t - \e^{-t[\B_0 \D]} \widetilde{h}^- + \e^{-t[\B_0 \D]} \int_0^\infty N^+ (1-P_{\B_0 \cH}) \cE_s f_s \; \d s \\
=& \int_0^t \e^{-(t-s)[\B_0 \D]} (\widetilde{E}_0^+ + N^-(1-P_{\B_0 \cH})) \cE_s f_s \; \d s -  \int_t^\infty \e^{-(s-t)[\B_0 \D]} (\widetilde{E}_0^- + N^+(1-P_{\B_0 \cH})) \cE_s f_s \; \d s \\
&+ \e^{-t[\B_0 \D]} \int_0^\infty \e^{-s [\B_0 \D]} (\widetilde{E}_0^- + N^+(1-P_{\B_0 \cH})) \cE_s f_s \; \d s,
\end{align*}
where the integrals on the right-hand side can be split as before. The only difference is that due to $(\widetilde{E}_0^+ + N^-(1-P_{\B_0 \cH})) + (\widetilde{E}_0^- + N^+(1-P_{\B_0 \cH})) = \Id$ on $\L^2(\Omega)^n$ the projection $P_{\B_0 \cH}$ does not occur in $I_5$ and $I_6$. Note that $I_3$ and $I_4$ stay the same since $\psi(t \B_0 \D) = 0$ on $\Ke(\B_0 \D)$. Altogether,
\begin{align}
\label{Eq: NT Dirichlet Result Step 2}
\begin{split}
&\bigg\|u - c + \bigg(\e^{-t[\B_0 \D]} \bigg(\widetilde{h}^- + \widetilde{h}^+ - \int_0^\infty N^+ (1-P_{\B_0 \cH}) \cE_s f_s \; \d s\bigg) + (1+\i t \B_0 \D)^{-1} \int_0^t \cE_s f_s \; \d s \bigg)_\pe \bigg \|_{\cY^*} \\
&\lesssim \|\cE\|_C \|f\|_\cY
\end{split}
\end{align}

\subsection*{\normalfont \itshape Step 3: The non-tangential estimate for $u$}

\noindent We only have to consider Whitney balls of size $t < 1$ and prove the estimate $\|\NT^2(u)\|_2 \lesssim \|f\|_\cY$. In fact, for Whitney balls of size $t\geq 1$ the estimate $\sup_{s>0}\|u_s\|_{\L^2(\Omega)^m} \lesssim \|f\|_\cY + |u_\infty|$ provided by Theorem~\ref{Thm: Representation for Dirichlet problem} directly gives
\begin{align*}
 \bariint_{W(t,x)} |u|^2 
 \lesssim \barint_{c_0^{-1}t}^{c_0 t} \|u_s\|_{\L^2(\Omega)^m}^2 \; \d s
 \lesssim \|f\|_\cY^2 + |u_\infty|^2,
\end{align*}
uniformly in $t>1$ and $x \in \Omega$. For this we have implicitly used $d$-Ahlfors regularity of $\Omega$ and we have $u_\infty = 0$ if the lateral Dirichlet part is non-empty. Now, for $x \in \Omega$ and $t< 1$ we recall from Corollary~\ref{Cor: Poincare-Whitney estimate} that
\begin{align*}
 \bariint_{W(t,x)} |u|^2 \lesssim  \bigg(\bariint_{2W(t,x)} |t \nablatx  u|^2 \bigg) +\bigg(\bariint_{2W(t,x)} |u|\bigg)^2.
\end{align*}
Taking the supremum over $t$ and integrating with respect to $x$ leads us to
\begin{align}
\label{Eq1: Dirichlet NT bound}
 \|\NT^2(u)\|_2 \lesssim \|t\nabla_{t,x}u\|_\cX + \|\NT^1(u)\|_2 \lesssim \|f\|_\cY + \|\NT^1(u)\|_2,
\end{align}
where for the second step we have utilized the embedding $\cY^* \subseteq \cX$ from Lemma~\ref{Lem: X inside Y*}. To be precise, we are using maximal functions that take averages over enlarged Whitney regions $2W(t,x)$ here, but of course this is just a matter of choosing the generic constants $c_0$ and $c_1$. Concerning the estimate of $\NT^1(u)$, we only consider the more difficult case $\Dir = \emptyset$. The simplifications in the other case are obvious. We subtract from $u$ all terms we have control on by \eqref{Eq: NT Dirichlet Result Step 2}. By the triangle inequality along with the embedding $\cY^* \subseteq \cX$ and the pointwise estimate $\NT^1 \leq \NT$ we obtain
\begin{align*}
 \|\NT^1(u)\|_2 
&\lesssim \abs{c} + \bigg\|\e^{-t[\B_0 \D]} \Big(\widetilde{h}^- + \widetilde{h}^+\Big) \bigg\|_\cX +\bigg\| \e^{-t[\B_0 \D]} \int_0^\infty N^+ (1-P_{\B_0 \cH}) \cE_s f_s \; \d s\bigg\|_\cX \\
&\quad + \bigg\|\NT^1\bigg((1+\i t \B_0 \D)^{-1} \int_0^t \cE_s f_s \; \d s\bigg) \bigg\|_2 +  \|\cE\|_C\|f\|_\cY.
\end{align*}
For the second term on the right-hand side we intertwine as in \eqref{Eq: Relation DSA = SATilde} and then use the $\NT$-bound for the $[\D \B_0]$-semigroup from Remark~\ref{Rem: NT bound on negative Hardy space}. The same argument applies to the third term and the $\L^2$-norm of the integral in $s$ can be bounded by means of Proposition~\ref{Prop: TA boundedness}. Eventually, the fourth term can be controlled by means of Lemma~\ref{Lem: NTp bound resolvents}. Thus,
\begin{align*}
\|\NT^1(u)\|_2 &\lesssim \abs{c} + \|\widetilde{h}^-\|_2 + \|\widetilde{h}^+\|_2 + \|f\|_\cY.
\end{align*}
Now, $\|\widetilde{h}^+\|_2 \lesssim \|f\|_\cY$ by Theorem~\ref{Thm: Representation for Dirichlet problem} and $\|\widetilde{h}^-\|_2 \lesssim \|f\|_\cY$ by Step~1. Reinserting these estimates back on the right-hand side of \eqref{Eq1: Dirichlet NT bound} shows $\|\NT^2(u)\|_2 \lesssim \abs{c} + \|f\|_\cY \simeq |u_\infty| + \|f\|_\cY$, where the last equivalence follows again from Theorem~\ref{Thm: Representation for Dirichlet problem}.

\subsection*{\normalfont \itshape Step 4: Almost everywhere convergence of Whitney averages}

\noindent Finally we prove that Whitney averages of $u$ converge to the trace $c - (v_0)_\pe$ for a.e.\ $x\in \Omega$ as $t \to 0$. Here, as usual, $c = 0$ if the lateral Dirichlet part is non-empty. Since the right-hand side of \eqref{Eq: NT Dirichlet Result Step 2} and \eqref{Eq: NT Dirichlet Result Step 2a}, respectively, are contained in $\cY^*$, Whitney averages converge to $0$ thanks to Lemma~\ref{Lem: Whitney averages vanish on Y*}. Concerning the resolvent term we fix $t_0 \in (0, c_0)$ and note for $t<t_0$ and $x \in \Omega$ that
\begin{align*}
 R(f)(t,x):= \bigg((1+\i t \B_0 \D)^{-1} \int_0^t \cE_s f_s \; \d s \bigg)(x) = R(\ind_{(0,t_0)}f)(t,x).
\end{align*}
By Lemma~\ref{Lem: NTp bound resolvents} we have
\begin{align*}
 \int_\Omega &\sup_{s<c_0^{-1}t_0} \bigg( \bariint_{W(s,y)} |R(f)(t,x)| \; \d x \; \d t \bigg)^2 \; \d y \\
&\leq \int_\Omega \sup_{s<c_0^{-1}t_0} \bigg( \bariint_{W(s,y)} |R(\ind_{\{t<t_0 \}}f)(t,x)| \; \d x \; \d t \bigg)^2 \; \d y \\
&\lesssim \|\cE\|_C^2 \int_0^{c_0 t_0} \int_\Omega |f(s,y)|^2 \; \d y \; s \d s,
\end{align*}
where by dominated convergence the final term converges to $0$ in the limit $t_0 \to 0$. This implies that Whitney averages of $Rf$ converge to zero almost everywhere.

We conclude that in the sense of Whitney averages the limit of $-(u - c)$ as $t \to 0$ is the same as the perpendicular part of the limit of the semigroup term in \eqref{Eq: NT Dirichlet Result Step 2} and \eqref{Eq: NT Dirichlet Result Step 2a}, respectively. It follows from Theorem~\ref{Thm: Almost everywhere convergence of Whitney averages} and since $(N^+ h)_\pe = 0$ for every $h \in \L^2(\Omega)^n$, that this latter limit is precisely $\widetilde{h}^- + \widetilde{h}^+ = v_0$. This completes the proof.
\end{proof}

\section{Well-posedness}
\label{Sec: Well-posedness}

\noindent We are finally ready to study well-posedness of the three boundary value problems in the sense of Section~\ref{Sec: Main results}. Eventually, we will prove our third main result, Theorem~\ref{Thm: Well-posedness}. Throughout this section we fix $t$-dependent coefficients $A$ and $t$-independent coefficients $A_0$ satisfying Assumption~\ref{Ass: Ellipticity for BVP} and as before let $\B$ and $\B_0$ correspond to $A$ and $A_0$, respectively. 

We begin by rephrasing well-posedness of the boundary value problems for $A_0$ in terms of Hardy projections. Recall the operators $N^\pm$ and $N$ from Definition~\ref{Def: N+-}.

\begin{lemma}
\label{Lem: Reformulation of well-posedness} \quad
\begin{enumerate}
\def\theenumi{\roman{enumi}}
  \item \label{WP Neu/Reg} The Neumann and regularity problem for $A_0$ are well-posed if and only if $N^- :E_0^+ \cH \to N^- \cH$ and $N^+ : E_0^+ \cH \to N^+ \cH$ are isomorphisms, respectively.

  \item \label{WP Dir} If $\Dir \neq \emptyset$, then the Dirichlet problem for $A_0$ is well-posed if and only if $N^-:\widetilde{E}_0^+ \L^2(\Omega)^n \to N^- \cH$ is an isomorphism.

  \item \label{WP Dir Pure Neumann} If $\Dir = \emptyset$, then the Dirichlet problem for $A_0$ is well-posed if and only if
  \begin{align*}
    N^-: \widetilde{E}_0^+ \L^2(\Omega)^n \oplus \bigg\{\begin{bmatrix}c \\ 0 \end{bmatrix}; \, c \in \IC^m \bigg\} \to \L^2(\Omega)^m \times \{0\}
  \end{align*}
  is an isomorphism.
\end{enumerate}
\end{lemma}

\begin{proof}
Part \eqref{WP Neu/Reg} is a direct consequences of Proposition~\ref{Prop: f are conormals of u} and Corollary~\ref{Cor: X-solutions t-independent}. The map under consideration in part \eqref{WP Dir} is well-defined since $N^- \cH = \L^2(\Omega)^m$ and if it is an isomorphism, then in view of Corollary~\ref{Cor: Y-solutions t-independent} the Dirichlet problem is well-posed. Conversely assume the Dirichlet problem is well-posed. By Corollary~\ref{Cor: Y-solutions t-independent} the map $N^- : \widetilde{E}_0^+ \L^2(\Omega)^n \to N^- \cH$ is onto. Now, suppose $N^- \widetilde{h}^+ = 0$ for some $\widetilde{h}^+ \in \widetilde{E}_0^+ \L^2(\Omega)^n$ and define $v_t = - \e^{-t [\B_0 \D]} \widetilde{h}^+$, $t \geq 0$. Corollary~\ref{Cor: Y-solutions t-independent} reveals $v_\pe$ as a solution of the Dirichlet problem with Lusin area bound and data $N^- \widetilde{h}^+ = 0$. By well-posedness, $v_\pe = 0$. As in the proof of Theorem~\ref{Thm: Representation for Dirichlet problem}, $0 = \nablaA v_\pe = \D v$. This means $v_t \in \Ke(\D)$ for all $t > 0$ and the topological decomposition $\Ke(\D) \oplus \B_0 \cH$ forces $v_t = 0$ for all $t>0$. By strong continuity of the semigroup $\widetilde{h}^+ = 0$ follows. Part \eqref{WP Dir Pure Neumann} is proved analogously, taking into account $\IC^m \times \{0\} \subseteq \Ke(\D)$ and that $\Ke(\D) \oplus \B_0 \cH$ is a topological decomposition.
\end{proof}
\subsection{Small perturbations}
\label{Subsec: Small perturbations}

In this section we establish stability of well-posedness for $t$- independent coefficients under small perturbations of the coefficients with respect to the $\L^\infty$-topology. For the Neumann and regularity problem this is almost immediate from holomorphic dependence of the Hardy projections on $\B_0$ and the following lemma.

\begin{lemma}[{\cite[Lem.~4.2]{AAM-ArkMath}}]
\label{Lem: Projections isomorphisms}
Let $\delta > 0$. Let $P_t$, $-\delta \leq t \leq \delta$ be bounded projections on a Hilbert space $\cK$ depending continuously on $t$ in the $\Lop(\cK)$-topology. Let $S: \cK \to \cJ$ be a bounded operator into a Hilbert space $\cJ$. If $S: P_0 \cK \to \mathcal{J}$ is an isomorphism, then there exists $0<\eps <\delta$, such that $S: P_t \cK \to \cJ$ is an isomorphism when $\abs{t} < \eps$.
\end{lemma}

\begin{proposition}
\label{Prop: Stability of WP for Neu and Reg}
Consider the set of those $t$-independent coefficients satisfying Assumption~\ref{Ass: Ellipticity for BVP} for which the Neumann problem for $A_0$ is well-posed. This is an open subset of $\L^\infty(\Omega; \Lop(\IC^n))$. A similar result holds for the regularity problem.
\end{proposition}

\begin{proof}
If $A_0$ satisfies Assumption~\ref{Ass: Ellipticity for BVP} with respective constant $\lambda > 0$ and $C \in \L^\infty(\Omega; \Lop(\IC^n))$ is any matrix, then for $z \in \IC$ sufficiently close to $1$ the matrices $A_0(z):= (1-z) C + z A_0$ satisfy Assumption~\ref{Ass: Ellipticity for BVP} with respective constant $\frac{\lambda}{2}$. Let $\B_0(z)$ correspond to $A_0(z)$ as usual. Lemma~\ref{Lem: B accretive on the range of D} and Proposition~\ref{Prop: Holomorphic dependence of Hinfty calculus} yield holomorphy of $ z \mapsto \ind_{\IC^+}(\D \B_0(z))$. The claim follows from Lemma~\ref{Lem: Projections isomorphisms} and the characterization for well-posedness given in Lemma~\ref{Lem: Reformulation of well-posedness}.
\end{proof}

The inhomogeneity of considering $N^-$ on $\widetilde{E}_0^+ \L^2(\Omega)^n$ for the Dirichlet problem can be circumvented by the \emph{Dirichlet-regularity duality}.

\begin{proposition}[Dirichlet-regularity duality]
\label{Prop: Dirichlet-Regularity duality}
The Dirichlet problem for $A_0$ is well-posed if and only if the regularity problem for $A_0^*$ is well-posed.
\end{proposition}

This principle is known in the setting $\Omega = \IR^d$. As the adaption to our framework bears some subtle difficulties, we include an elementary and completely abstract proof building on the following two lemmas.

\begin{lemma}
\label{Lem: Similarities of operators}
Let $P$ be the orthogonal projection in $\L^2(\Omega)^n$ onto $\cH$. There are similarities of operators 
\begin{align*}
 \D \B_0|_{\cl{\Rg(\D \B_0)}} = R^{-1} (\B_0 \D_{\cl{\Rg(\B_0 \D)}}) R \quad \text{and} \quad \B_0 \D_{\cl{\Rg(\B_0 \D)}} = S^{-1} (P \B_0 \D|_{\cl{\Rg(\D \B_0)}}) S.
\end{align*}
The isomorphisms $R, S^{-1}: \cl{\Rg(\D \B_0)} \to \cl{\Rg(\B_0 \D)}$ are given by $R = \B_0|_{\cl{\Rg(\D \B_0)}}$ and $S = P|_{\cl{\Rg(\B_0 \D)}}$. Moreover, $S^{-1}$ is the restriction to $\cl{\Rg(\D \B_0)}$ of the projection $Q$ onto $\cl{\Rg(\B_0 \D)}$ along the splitting $\L^2(\Omega)^n = \Ke(\D) \oplus \cl{\Rg(\B_0 \D)}$. 
\end{lemma}

\begin{proof}
The first similarity of operators is proved in \cite[Prop.~2.1(2)]{Auscher-Stahlhut_Remarks} and the second one follows as in \cite[Prop.~2.2]{Auscher-Stahlhut_APriori}.
\end{proof}

\begin{lemma}[{\cite[p.~37]{Auscher-Axelsson-Hofmann_ComplexPerturbations}, \cite[Lem.~6.5.9]{EigeneDiss}}]
\label{Lem: Adjoint projections}
Assume that $N^\pm$ and $E^\pm$ are two pairs of complementary bounded projections on a Hilbert space $\cK$, i.e., $(N^\pm)^2 = N^\pm$ and $N^+ + N^{-} = \Id$, and similarly for $E^\pm$. Then the adjoint operators $(N^\pm)^*$ and $(E^\pm)^*$ are also two pairs of complementary projections on $\cK$ and the restricted projection $N^+: E^+ \cK \to N^+ \cK$ is an isomorphism if, and only if the restricted adjoint projection $(N^-)^*: (E^-)^*\cK \to (N^-)^* \cK$ is an isomorphism.
\end{lemma}

\begin{proof}[Proof of Proposition~\ref{Prop: Dirichlet-Regularity duality}]
Note that $A_0^*$ satisfies the same accretivity condition as $A_0$ and that replacing $A_0$ with $A_0^*$ amounts to replacing $\B_0$ with $\B_0^\bigstar = N \B_0^* N$ and $\D \B_0$ with $\D \B_0^\bigstar = - N \D \B_0^* N$. We abbreviate $E_\bigstar^\pm:= \ind_{\IC^+}(\D \B_0^\bigstar)$.

\subsection*{\normalfont \itshape Step 1: Rephrasing well-posedness of the Dirichlet problem}

\noindent We begin with establishing a representation for the space $\widetilde{E}_0^+ \L^2(\Omega)^n$ that better suits our circumstance. First, $\widetilde{E}_0^+ \L^2(\Omega)^n = \B_0 E_0^+ \cH$ by the intertwining property \eqref{Hinfty 4} in Section~\ref{Sec: Quadratic estimates for DB0}. The similarities from Lemma~\ref{Lem: Similarities of operators} are inherited to the functional calculus. So, $E_0^+ = R^{-1} S^{-1} \ind_{\IC^+} (P\B_0 \D|_\cH) S R$. As $SR$ is an automorphism of $\cH$ and since $\B_0 R^{-1} = \Id$ on $\B_0 \cH$, it follows $B_0  E_0^+ \cH =  S^{-1} \ind_{\IC^+}(P\B_0 \D|_\cH)\cH$. Corollary~\ref{Cor: Adjoint of injective part of DB} with the roles of $\B_0$ and $\B_0^*$ interchanged yields
\begin{align*}
 P\B_0 \D|_\cH = (\D \B_0^*|_{\cH})^* = (-N \D \B_0^\bigstar N|_{\cH})^* = -N (\D \B_0^\bigstar|_{\cH})^* N|_{\cH},
\end{align*}
where all adjoints are taken within $\cH$. Taking into account $\ind_{\IC^+}(z) = \ind_{\IC^-}(-\cl{z})$, $z \in \IC$, and $N^{-1} = N$, this carries over to $\ind_{\IC^+}(P \B_0 \D|_\cH) = N (E_\bigstar^-)^* N|_{\cH}$ as before. Altogether,
\begin{align}
\label{Eq2: Dirichlet-Regularity duality}
 \widetilde{E}_0^+ \L^2(\Omega)^n = \B_0  E_0^+ \cH = S^{-1} N (E_\bigstar^-)^* \cH.
\end{align}

\subsection*{\normalfont \itshape Step 2: The claim for non-empty lateral Dirichlet part}

\noindent Assume $\Dir \neq \emptyset$. By Lemma~\ref{Lem: Similarities of operators} and Step~1, well-posedness of the Dirichlet problem for $A_0$ is equivalent to $N^-: S^{-1} N (E_\bigstar^-)^* \cH \to N^- \cH$ being an isomorphism. From Lemma~\ref{Lem: Similarities of operators} recall that $S^{-1}$ agrees with the projection $Q$ onto $\B \cH$ which annihilates $\Ke(\D)$. Since the first map in the chain
\begin{align}
\label{Eq: Composite map Dirichlet duality}
 \begin{CD}
 (E_\bigstar^-)^* \cH @>{S^{-1}N}>> S^{-1} N (E_\bigstar^-)^* \cH @>{N^-}>> N^- \cH
\end{CD}
\end{align}
is an isomorphism, well-posedness of the Dirichlet problem is equivalent to the composite map being an isomorphism. From the identity
\begin{align}
\label{Eq3: Dirichlet-Regularity duality}
  N^- S^{-1} Nh = N^-Nh - N^-(1- Q)Nh = -N^{-}h - N^-(1- Q)Nh \qquad (h \in \cH)
\end{align}
and the fact that $N^- \Ke(\D) = \{0\}$ by injectivity of $\gV$, we see that the composite map in \eqref{Eq: Composite map Dirichlet duality} acts as $N^{-}: (E_\bigstar^-)^* \cH \to N^- \cH$. Lemmas~\ref{Lem: Adjoint projections} and \ref{Lem: Reformulation of well-posedness} yield the claim.

\subsection*{\normalfont \itshape Step 3: The claim for empty lateral Dirichlet part}

\noindent Finally, we consider the case $\Dir = \emptyset$. First assume that the regularity problem for $A_0^*$ is well-posed. In view of \eqref{Eq2: Dirichlet-Regularity duality} and Lemmas~\ref{Lem: Adjoint projections} and \ref{Lem: Reformulation of well-posedness}, we have that $N^-: (E_\bigstar^-)^* \cH \to N^{-} \cH$ is an isomorphism and have to show that so is
\begin{align}
\label{Eq4: Dirichlet-Regularity duality}
  N^-: S^{-1} N (E_\bigstar^-)^* \cH \oplus \bigg \{ \begin{bmatrix} c \\ 0 \end{bmatrix}; \, c \in \IC^m \bigg\} \to \L^2(\Omega)^m.
\end{align}
Suppose $h \in (E_\bigstar^-)^* \cH$ and $c \in \IC^m$ satisfy $N^{-}(S^{-1} Nh) + c = 0$. By \eqref{Eq3: Dirichlet-Regularity duality},
\begin{align*}
 -N^- h - N^-(1-Q)Nh + c = 0,
\end{align*}
where the first term has zero average on $\Omega$ and the second and third terms are constant on $\Omega$. This forces $N^- h = 0$ and $N^-(1-Q)Nh = c$. By assumption $h = 0$ and therefore $c=0$, proving that the map in \eqref{Eq4: Dirichlet-Regularity duality} is one-to-one. As for ontoness, let $g \in \L^2(\Omega)^m$ be given and define $g_\Omega := \barint_\Omega g$. By assumption, there exists $h \in (E_\bigstar^-)^* \cH$ such that $-N^-h = g - g_\Omega$. Putting $c = g_\Omega + N^-(1-Q)Nh$, it follows once again from \eqref{Eq3: Dirichlet-Regularity duality} that
\begin{align*}
 N^-(S^{-1}Nh) + c = -N^-h + N^-(1- Q)Nh + c = g - g_\Omega + g_\Omega = g.
\end{align*}
Conversely, assume that \eqref{Eq4: Dirichlet-Regularity duality} provides an isomorphism. In order to prove that $N^-: (E_\bigstar^-)^* \cH \to N^{-} \cH$ is an isomorphism as well, first let $h \in (E_\bigstar^-)^* \cH$ satisfy $N^- h = 0$. With $c := - N^-(1-Q)Nh$ we obtain from \eqref{Eq3: Dirichlet-Regularity duality} that $N^- S^{-1} Nh = c$, whence $S^{-1}Nh = \begin{bmatrix} c \\ 0 \end{bmatrix}$. The topological decomposition $\Ke(\D) \oplus \B_0 \cH$ yields $S^{-1}Nh = 0$ and thus $h=0$. Also, given $g \in N^- \cH$, by assumption there exist $h \in (E_\bigstar^-)^* \cH$ and $c \in \IC^m$ such that
\begin{align*}
 g = N^-(S^{-1}Nh) + c = -N^-h - N^-(1- Q)Nh + c.
\end{align*}
Since $g$ and $-N^-h$ have zero average on $\Omega$ and as the other two terms are constant, $g = N^-(-h)$ follows. 
\end{proof}

\begin{corollary}
\label{Cor: Stability of Dir}
Consider the set of those $t$-independent coefficients satisfying Assumption~\ref{Ass: Ellipticity for BVP} for which the Dirichlet problem for $A_0$ is well-posed. This is an open subset of $\L^\infty(\Omega; \Lop(\IC^n))$.
\end{corollary}
\subsection{Well-posedness for block and Hermitean matrices}
\label{Subsec: Well-posedness for block and hermitean matrices}

For the Neumann and regularity problem there are at least two classes of $t$-independent coefficients for which invertibility of the projections in Lemma~\ref{Lem: Reformulation of well-posedness}\eqref{WP Neu/Reg} is known nowadays: The class of matrices $A_0$ in block-form 
\begin{align*}
 A_0 = \begin{bmatrix} (A_0)_{\pe \pe} & 0 \\ 0 & (A_0)_{\pa \pa} \end{bmatrix},
\end{align*}
and the Hermitean matrices satisfying $A_0^* = A_0$. (There are also some results for block-triangular matrices \cite{Auscher-McIntosh-Mourgoglu_L2BVPs}). For elliptic systems on the upper halfspace these results have first been obtained in \cite{AAM-ArkMath}. More precisely, the following was shown \cite[Sec.\ 4.1/2]{AAM-ArkMath}:
\begin{itemize}
 \item If $A_0$ is of block form, then the projections in Lemma~\ref{Lem: Reformulation of well-posedness}\eqref{WP Neu/Reg} can be inverted by a purely algebraic formula relying on the identity $N^{-1}\B_0 N = \B_0$ valid since $\B_0$ is of block form as well. In fact, well-posedness in this case is equivalent to the solution of the Kato problem for elliptic systems on $\Omega$ with Dirichlet condition on $D$, recently solved in \cite{Darmstadt-KatoMixedBoundary}.
 \item Well-posedness of the Neumann and regularity problem for Hermitean $A_0$ follows from well-posedness of these problems with $A_0 = \Id$ the identity matrix (which is of block form), the method of continuity (which we have at our disposal thanks to holomorphic dependence of the Hardy projections on $A_0$ as in the proof of Proposition~\ref{Prop: Stability of WP for Neu and Reg}), and the so-called Rellich identity. The proof of the latter is again an abstract argument that literally applies in our situation as well. 
\end{itemize}
Finally, well-posedness of the Dirichlet problem follows on using Proposition~\ref{Prop: Dirichlet-Regularity duality}. Summing up, we can record the following result.

\begin{proposition}
\label{Prop: Wellposedness for special matrices}
Each of the three boundary value problems for $A_0$ is well-posed if $A_0$ is of block-form 
\begin{align*}
 A_0 = \begin{bmatrix} (A_0)_{\pe \pe} & 0 \\ 0 & (A_0)_{\pa \pa} \end{bmatrix},
\end{align*}
or Hermitean, that is, $A_0^* = A_0$.
\end{proposition}
\subsection{The proof of Theorem~\ref{Thm: Well-posedness}}
\label{Subsec: Proof of well-posedness theorem}

The claim for the $t$-independent coefficients follows from Propositions \ref{Prop: Stability of WP for Neu and Reg}, Corollary~\ref{Cor: Stability of Dir}, and Proposition~\ref{Prop: Wellposedness for special matrices}.

Next, we inquire well-posedness of the Neumann and regularity problems for $A$. Throughout, we assume that the Neumann and regularity problems for $A_0$ are well-posed and that $\|A -A_0\|_C$ is small enough so that $1-S_A$ is invertible on $\cX$ thanks to Proposition~\ref{Prop: Boundedness of SA}. In view of Proposition~\ref{Prop: f are conormals of u} and Theorem~\ref{Thm: Representation for X solutions} the conormal gradients $f$ of weak solutions $u$ to the second-order system are precisely the functions $f = (1-S_A)^{-1} \e^{-[t \D \B_0]} h^+$, $h^+ \in E_0^+ \cH$, which converge in the sense of Whitney averages as well as in square Dini sense to
\begin{align*}
 \Gamma_A h^+:= h^+ + \int_0^\infty \D \B_0 \e^{-s [\D \B_0]} \widehat{E}_0^- \cE_s f_s \; \d s
\end{align*}
as $t \to 0$. Note that $\Gamma_A$ really is a linear operator acting on $h^+$ since $f$ is determined by $h^+$. It follows that the Neumann problem and regularity problem for $A$ is well-posed, if $N^- \Gamma_A: E_0^+ \cH \to N^- \cH$ and $N^+ \Gamma_A: E_0^+ \cH \to N^+ \cH$ is an isomorphism, respectively. Since $\Gamma_{A_0} = \Id$, we can compute
\begin{align*}
 \|N^\pm (\Gamma_A - \Gamma_{A_0}) h^+\|_2 \leq \|(\Gamma_A - \Gamma_{A_0}) h^+\|_2
&= \sup_{\|g\|_2 = 1} \bigg|\int_0^\infty \bigscal{\widehat{E}_0^- \cE_s f_s}{\B_0^* \D \e^{-s [\B_0^* \D]} g}_2 \; \d s \bigg| \\
& \leq \sup_{\|g\|_2 = 1} \|\widehat{E}_0^- \cE f\|_{\cY^*} \|\B_0^* \D \e^{-s [\B_0^* \D]} g\|_{\cY}.
\end{align*}
Employing Theorems~\ref{Thm: * norm equivalent to Carleson} and \ref{Thm: NT bound for semigroup solutions} we can infer a bound by $\|\cE f\|_{\cY^*} \lesssim \|\cE\|_C \|f\|_\cX \lesssim \|\cE\|_C \|h^+\|_2$ for the first term. For the second term, we note that $\B_0^*$ satisfies the same accretivity condition as $\B_0$ and so by quadratic estimates for $\B_0^* \D$ it is bounded by $\|g\|_2 \leq 1$. Altogether, 
\begin{align*}
 \|N^\pm \Gamma_A - N^\pm \Gamma_{A_0}\|_{E_0^+ \cH \to N^\pm \cH} \lesssim \|\cE\|_C \simeq \|A - A_0\|_C,
\end{align*}
see the beginning of Section~\ref{Sec: Representation of solutions} for the last estimate. By assumption and Lemma~\ref{Lem: Reformulation of well-posedness} the operators $N^\pm \Gamma_{A_0}: E_0^+ \cH \to N^\pm \cH$ are isomorphisms, respectively, and hence so are $N^\pm \Gamma_A$ if $\|A-A_0\|_C$ is sufficiently small. Finally, the required estimates are already contained in Theorem~\ref{Thm: Representation for X solutions}.

We tend to the Dirichlet problem for $A$ and first assume $\Dir \neq \emptyset$. We suppose that the Dirichlet problem for $A_0$ is well-posed and that $\|A -A_0\|_C$ is small enough, so that $1-S_A$ is invertible on $\cY$ thanks to Proposition~\ref{Prop: Boundedness of SA}. We have at hand a representation formula, see Theorem~\ref{Thm: Representation for Dirichlet problem} and know from Theorem~\ref{Thm: NT bounds Dirichlet} that these solutions attain there boundary trace on $\Omega$ in the sense of a.e.\ convergence of Whitney averages as well as convergence in $\L^2(\Omega)^m$. Similar as for the Neumann and regularity problems, we only have to consider the operator
\begin{align*}
 \widetilde{\Gamma}_A: \widetilde{h}^+ \mapsto - \widetilde{h}^+ + \int_0^\infty \e^{-s [\B_0 \D]} \widetilde{E}_0^- \cE_s f_s \; \d s \quad \text{with} \quad f = (1-S_A)^{-1} \D \e^{-t [\B_0 \D]} \widetilde{h}^+
\end{align*}
on $\widetilde{E}_0^+ \L^2(\Omega)^n$ and we have to prove that $\|N^-\widetilde{\Gamma}_A - N^- \widetilde{\Gamma}_{A_0} \|_{\L^2(\Omega)^m \to \L^2(\Omega)^m} \lesssim \|\cE\|_C$. To this end, we use $\widetilde{\Gamma}_{A_0} = - \Id$ and compute
\begin{align*}
 \|(\widetilde{\Gamma}_A - \widetilde{\Gamma}_{A_0}) \widetilde{h}^+\|_2
&= \sup_{\|g\|_2 = 1} \bigg| \int_0^\infty \bigscal{f_s}{\cE_s^* \e^{-s [\D \B_0^*]}  (\widetilde{E}_0^-)^* g}_2 \; \d s \bigg| \\
& \leq \sup_{\|g\|_2 = 1} \|f\|_\cY \| \cE_s^* \e^{-s [\D \B_0^*]}  (\widetilde{E}_0^-)^* g\|_{\cY^*}.
\end{align*}
Boundedness of $(1-S_A)^{-1}$ on $\cY$, accretivity of $\B_0$ on $\cH$, and quadratic estimates for $\B_0 \D$ yield $\|f\|_\cY \lesssim \|\B_0 \D \e^{-t \B_0 \D} \widetilde{h}^+\|_\cY \lesssim \|\widetilde{h}^+\|_2$. Moreover, we note $\Rg( (\widetilde{E}_0^-)^*) \subseteq \cH$ (see e.g.\ the proof of Theorem~\ref{Thm: NT bounds Dirichlet}) and $\|\cE^*\|_C = \|\cE\|_C$. So, Theorem~\ref{Thm: * norm equivalent to Carleson} and Remark~\ref{Rem: NT bound on negative Hardy space} for $\D \B_0^*$ give
\begin{align*}
 \| \cE_s^* \e^{-s [\D \B_0^*]}  (\widetilde{E}_0^-)^* g\|_{\cY^*} \lesssim \|\cE\|_C \| \e^{-s [\D \B_0^*]}  (\widetilde{E}_0^-)^* g\|_{\cX} \lesssim \|\cE\|_C \| g\|_2
\end{align*}
and the required bound for $N^-\widetilde{\Gamma}_A - N^- \widetilde{\Gamma}_{A_0}$ follows. For the Dirichlet problem in the case $\Dir = \emptyset$ we are to consider
\begin{align*}
 &\widetilde{\Gamma}_A: \bigg(\widetilde{h}^+ + \begin{bmatrix} c \\ 0 \end{bmatrix}\bigg) \mapsto  \begin{bmatrix} c \\ 0 \end{bmatrix} - \widetilde{h}^+ + \int_0^\infty \e^{-s [\B_0 \D]} \widetilde{E}_0^- \cE_s f_s \; \d s \quad
 \text{with} \quad f = (1-S_A)^{-1} \D \e^{-t [\B_0 \D]} \widetilde{h}^+
\end{align*}
for $\widetilde{h}^+ \in \widetilde{E}_0^+ \L^2(\Omega)^n$ and $c \in \IC^m$, which can be handled by the same argument.

Finally, we prove the required estimates for the Dirichlet problem. To this end let $u$ be a weak solution with $\nablatx u \in \cY$ and let $u_0$ and $u_\infty$ be its limits at $t=0$ and $t = \infty$, respectively. From Theorem~\ref{Thm: Representation for Dirichlet problem} and Theorem~\ref{Thm: NT bounds Dirichlet} we already know
that
\begin{align*}
 \|u_0\|_2 \leq \sup_{t>0} \|u_t\|_2 \lesssim |u_\infty| + \|\nablatx u\|_\cY
\end{align*}
and
\begin{align*}
 \|u_0\|_2 \lesssim \|\NT(u)\|_2 \lesssim |u_\infty| + \|\nablatx u\|_\cY.
\end{align*}
To see that all four quantities are equivalent if $\|A-A_0\|_C$ is small enough, we note firstly that $\|\nablatx u\|_\cY \simeq \|\widetilde{h}^+\|_2$ by Theorem~\ref{Thm: Representation for Dirichlet problem} and secondly, since $N^-\widetilde{\Gamma}_A$ is an isomorphism as we have seen above, that $\|u_0\|_2 \simeq \|\widetilde{h}^+\|_2$ if $\Dir$ is non-empty. If $\Dir$ is empty, then similarly
\begin{align*}
 \|u_0\|_2 \simeq \bigg \|\widetilde{h}^+ + \begin{bmatrix} c\\ 0 \end{bmatrix} \bigg\|_2 \simeq |c| + \|\widetilde{h}^+\|_2 \simeq |c| + \|\nablatx u\|_\cY \simeq  |u_\infty| + \|\nablatx u\|_\cY,
\end{align*}
where the second step uses $\IC^m \times \{0\} \subseteq \Ke(\D)$ along with the topological decomposition $\B_0 \cH \oplus \Ke(\D)$ and the final step follows again from Theorem~\ref{Thm: Representation for Dirichlet problem}.  \hfill $\square$
\subsection*{Acknowledgment} The authors were partially supported by the ANR project ``Harmonic Ana-lysis at its Boundaries'', ANR-12-BS01-0013. M.E.\ thanks ``Studienstiftung des deutschen Volkes'' and ``Fondation Math\'{e}matique Jacques Hadamard'' for their support during the project.
\def\cprime{$'$} \def\cprime{$'$} \def\cprime{$'$}

\end{document}